\documentclass[11pt,a4paper,leqno]{amsart}
\usepackage{version}
\usepackage[utf8]{inputenc}
\usepackage{amsmath,amssymb,amsfonts,amsthm,yfonts}
\usepackage{calrsfs}
\usepackage{esint}
\usepackage{hyperref}
\usepackage{mathtools}
\usepackage{bbm}
\usepackage{enumitem}
\usepackage{graphicx}
\usepackage{caption}
\usepackage{subcaption}
\usepackage{epsfig}
\usepackage{float}
\usepackage{latexsym}

\newcommand*\mathinhead[2]{\texorpdfstring{$#1$}{#2}}
\newcommand*{\dd}{\mathop{}\!\mathrm{d}}
\newcommand*{\QQ}{\pazocal{Q}}
\newcommand*{\Ps}{\pazocal{P}^s}

\newcommand*{\ox}{\overline{x}}
\newcommand*{\oy}{\overline{y}}
\newcommand*{\oz}{\overline{z}}

\hypersetup{
	colorlinks=true,
	linkcolor=magenta,
	filecolor=black,      
	urlcolor=blue,
	citecolor=cyan,
}

\numberwithin{equation}{section}

\newtheoremstyle{note}
{8pt}
{6pt}
{\itshape}
{11.75pt}
{\bfseries}
{.}
{.5em}
{}
\theoremstyle{note}
\newtheorem*{thm*}{Theorem}
\newtheorem{thm}{Theorem}[section]
\newtheorem{lem}[thm]{Lemma}
\newtheorem*{prop*}{Proposition}

\newtheoremstyle{note2}
{8pt}
{1pt}
{}
{11.75pt}
{\bfseries}
{.}
{.5em}
{}
\theoremstyle{note2}
\newtheorem{defn}{Definition}[section]

\newtheoremstyle{note3}
{8pt}
{1pt}
{}
{11.75pt}
{\itshape}
{.}
{.5em}
{}
\theoremstyle{note3}
\newtheorem{rem}{Remark}[section]

\DeclareMathAlphabet{\pazocal}{OMS}{zplm}{m}{n}

\begin{document}
	\title[The fractional Lipschitz caloric capacity of Cantor sets]{The fractional Lipschitz caloric capacity of Cantor sets}
	\author{Joan Hernández}\thanks{ORCID: 0000-0002-2207-5981.\\The author has been supported by PID2020-114167GB-I00 (Mineco, Spain).}
	\begin{abstract}
		We characterize the s-parabolic Lipschitz caloric capacity of corner-like $s$-parabolic Cantor sets in $\mathbb{R}^{n+1}$ for $1/2<s\leq 1$. Despite the spatial gradient of the s-heat kernel lacking temporal anti-symmetry, we obtain analogous results to those known for analytic and Riesz capacities.
		\bigskip
		
		\noindent\textbf{AMS 2020 Mathematics Subject Classification:}  42B20 (primary); 28A12 (secondary).
		
		\medskip
		
		\noindent \textbf{Keywords:} Fractional heat equation, singular integrals, Cantor set.
	\end{abstract}
	
	\maketitle
    \vspace{-1cm}
    \section{Introduction}

    Recent progress in the theory of parabolic equations on time-varying domains has led to major developments, notably through the work of Hofmann, Lewis, Nyström, and Strömqvist \cite{Ho1, HoL, NSt}, among others. As expected, there has also been growing interest in understanding the properties of so-called caloric capacities and the removable singularities of bounded solutions to certain parabolic equations. For example, Mourgoglou and Puliatti \cite{MoPu} studied a specific caloric capacity related to a capacity density condition at a particular scale, enabling them to establish several PDE estimates near the boundary, which are essential for their blow-up-type arguments. Another example is the recent work by Mateu, Prat, and Tolsa \cite{MPr, MPrT}, who investigated removable singularities for Lipschitz caloric functions in terms of capacities, as well as their fractional generalizations, such as those discussed in \cite{H}.

Building upon this existing framework, the main goal of the present article is to estimate the fractional Lipschitz caloric capacities, related to fractional heat equations, of corner-like Cantor sets in $\mathbb{R}^{n+1}$. More precisely, we focus on the fractional variants of the heat equation associated with the following pseudo-differential operator: \begin{equation*} \Theta^s:=(-\Delta_x)^s+\partial_t, \qquad s\in(0,1], \end{equation*} where for $s=1$, we recover the usual heat equation, and $(-\Delta_x)$ denotes the Laplacian with respect to the spatial variables. When $s<1$, $(-\Delta_x)^s$, known as the $s$-fractional Laplacian with respect to the spatial variables, is defined via its Fourier transform: for each $t$ fixed,
\begin{equation*} 
    \pazocal{F}[(-\Delta_x)^{s}f](\xi,t)=|\xi|^{2s}\pazocal{F}[f](\xi,t),
\end{equation*} 
or via the integral representation: 
\begin{align*} 
    (-\Delta_x)^{s} f(x,t)& =c_{n,s}\, \text{p.v.}\int_{\mathbb{R}^n}\frac{f(x,t)-f(y,t)}{|x-y|^{n+2s}}\dd y.
\end{align*} 
For further properties of $(-\Delta_x)^s$, the reader may consult \cite[\textsection{3}]{DPV} and \cite{Ste}, as well as the works of Ros-Oton and Serra \cite{RoSe1, RoSe2} regarding regularity theory for these fractional operators.

In what follows, we denote by $P_s$ the fundamental solution of $\Theta^s$, that is $\Theta^sP_s=\delta_0$ in the distributional sense, where $\delta_0$ denotes the Dirac delta at $0\in\mathbb{R}^{n+1}$. The function $P_s$ can be computed by taking the inverse spatial Fourier transform of $e^{-4\pi^2 t|\xi|^{2s}}$ for $t>0$, and it is null when $t\leq 0$. For $s=1$, $P_s$ coincides with the classical heat kernel, denoted $P_1=W$. When $0<s<1$, an explicit formula for $P_s$ is not available, but Blumenthal and Getoor \cite[Theorem 2.1]{BG} proved that 
\begin{equation*} 
P_s(x,t) \approx_{n,s} \frac{t}{(|x|^2+t^{1/s})^{(n+2s)/2}}\chi_{t>0},
\end{equation*} 
where $\approx_{n,s}$ indicates that $P_s$ is bounded above and below by this quotient up to constants depending only on $n$ and $s$. It is important to notice that for $s=1/2$ the relation above becomes an exact equality (with implicit dimensional constants), since the inverse Fourier transform can be computed explicitly.

To simplify notation, we represent points in $\mathbb{R}^{n+1}$ as $\ox:=(x,t)\in\mathbb{R}^n\times\mathbb{R}$ and define the $s$-parabolic distance for $0<s\leq 1$ as 
\begin{equation*} 
    \text{dist}_{p_s}\big((x,t),(y,\tau)\big):=\max\left\{|x-y|^2,|t-\tau|^{1/s}\right\} \approx_{n,s} \left(|x-y|^2+|t-\tau|^{1/s}\right)^{1/2}.
\end{equation*} 
From this, one naturally defines $s$-parabolic cubes and balls. We also define the $s$-parabolic norm as $|\ox|_{p_s}:= \text{dist}_{p_s}(\ox,0)$, so that
\begin{equation*} 
    P_s(\ox)\approx_{n,s}\frac{t}{|\ox|_{p_s}^{n+2s}}\chi_{t>0}. 
\end{equation*} 
Moreover, as shown in \cite{HMPr}, the kernels $P_s$ and $\nabla_x P_s$ satisfy Calderón-Zygmund (CZ) type estimates of order $n$ and $n+1$, respectively, with respect to the $s$-parabolic norm. Specifically, for $\ox=(x,t)\neq 0$ and $s\in(0,1)$: \begin{align*} 
    |\nabla_x P_s(\ox)|\lesssim \frac{|x t|}{|\ox|_{p_s}^{n+2s+2}}, \qquad |\Delta_x P_s(\ox)|\lesssim \frac{|t|}{|\ox|_{p_s}^{n+2s+2}}, \qquad |\partial_t\nabla_x P_s(\ox)| \lesssim \frac{|x|}{|\ox|_{p_s}^{n+2s+2}}. 
\end{align*} 
The last bound holds only when $t\neq 0$. Furthermore, if $\ox'$ satisfies $|\ox-\ox'|_{p_s}\leq |\ox|_{p_s}/2$, then 
\begin{equation*} 
    |\nabla_x P_s(\ox)-\nabla_xP_s(\ox')|\lesssim \frac{|\ox-\ox'|_{p_s}^{2\zeta}}{|\ox|_{p_s}^{n+1+2\zeta}}, \qquad \text{ where $2\zeta:=\min\{1,2s\}$.} 
\end{equation*}

Motivated by these results, and following the approach of \cite{MPrT}, the author studied, in \cite{H}, the characterization of removable singularities for $s$-parabolic Lipschitz solutions of the fractional heat equation in the regime $1/2<s\leq 1$. In this context, imposing a fractional parabolic Lipschitz condition on solutions is equivalent to requiring that
\begin{equation} 
\label{eq1.1} 
    \|\nabla_x f\|_{L^\infty(\mathbb{R}^{n+1})}<\infty, \qquad \|\partial_t^{\frac{1}{2s}}f\|_{\ast,p_s}<\infty. 
\end{equation} 
That is, the function must be Lipschitz in space and satisfy an $s$-parabolic BMO estimate for the fractional time derivative. A function $f\in L^1_{\text{loc}}$ belongs to the $s$-parabolic BMO space, $\text{BMO}_{p_s}$, if
\begin{equation*} 
    \|f\|_{\ast,p_s}:= \sup_Q \frac{1}{|Q|}\int_Q |f(\ox)-f_Q|\dd\ox<\infty, 
\end{equation*} 
where the supremum is taken over all $s$-parabolic cubes, and $f_Q$ denotes the average of $f$ over $Q$. The fractional time derivative for $s\in(1/2,1]$ is defined by 
\begin{equation*} 
    \partial_t^{\frac{1}{2s}}f(x,t):=\int_\mathbb{R}\frac{f(x,\tau)-f(x,t)}{|\tau-t|^{1+\frac{1}{2s}}}\dd\tau. 
\end{equation*}

In \cite{H}, it is proven that the bounds in \eqref{eq1.1} imply $\|f\|_{\text{Lip}{\frac{1}{2s},t}}<\infty$, meaning that such functions are $(1,\frac{1}{2s})$-Lipschitz: Lipschitz in space and $\frac{1}{2s}$-Lipschitz in time. When $s=1$, this corresponds to the functions studied by Nyström and Strömqvist \cite{NSt}. The motivation for imposing this Lipschitz property stems from the work of Hofmann, Lewis, Murray, and Silver \cite{LS, LMu, HoL, Ho2}, where the relationship between parabolic singular integral operators and caloric layer potentials on graphs is explored. Their analysis suggests that the appropriate graphs to consider are indeed $(1,1/2)$-Lipschitz graphs.

Accordingly, for each $1/2<s\leq 1$, we define the $(1,\frac{1}{2s})$-Lipschitz caloric capacity $\Gamma_{\Theta^s}$ of a compact set $E\subset\mathbb{R}^{n+1}$ as the supremum of $|\langle T,1\rangle|$ over distributions $T$ supported on $E$ such that \begin{equation*} \|\nabla_x P_s\ast T\|_{L^\infty(\mathbb{R}^{n+1})}\leq 1, \qquad \|\partial_t^{\frac{1}{2s}}P_s\ast T\|_{\ast,p_s}\leq 1. \end{equation*} Variants of the capacity arise by restricting $T$ to positive measures or by imposing, additionally, the same normalization conditions on the conjugate kernel $P_s^\ast(\ox):=P_s(-\ox)$, resulting in $\Gamma_{\Theta^s,+}$ and $\widetilde{\Gamma}_{\Theta^s,+}$, respectively.

In \cite{H}, it is shown that $\widetilde{\Gamma}_{\Theta^s,+}$ can be characterized in terms of $L^2$-boundedness of a certain convolution operator and that the nullity of $\Gamma_{\Theta^s}$ characterizes removability for $(1,\frac{1}{2s})$-Lipschitz solutions of the fractional heat equation. It is also established that the critical $s$-parabolic Hausdorff dimension of $\Gamma_{\Theta^s}$ in $\mathbb{R}^{n+1}$ is $n+1$, and a removable fractal set $E_{p_s}$ with positive and finite $\pazocal{H}^{n+1}_{p_s}$ measure is constructed.

Regarding the critical dimension, it is worth noting that the fractional parameter $s$ appears only in defining the $s$-parabolic distance, which determines the corresponding $s$-parabolic Hausdorff dimension. This suggests that as $s\to 1/2$, and the metric approaches the Euclidean one, the critical dimension equals that of the ambient space. Inspired by Uy’s work \cite{U} on analytic capacity, it is conjectured that the capacity $\Gamma_{\Theta^{1/2}}$\footnote{The Lipschitz capacity for the $\Theta^{1/2}$ equation, i.e. that obtained by imposing $\|\nabla P\ast T\|_{L^\infty(\mathbb{R}^{n+1})}$, where now $\nabla=(\nabla_x,\partial_t)$ is a full gradient.} of a compact set $E$ should be comparable to its Lebesgue measure $\pazocal{L}^{n+1}(E)$, although this remains an open question.

In this article, we generalize the construction of $E_{p_s}$ by means of a sequence of contraction ratios $(\lambda_j)_j$ satisfying certain conditions depending on $s$ and an absolute parameter $\tau_0$. To be precise, we fix $1/2<s\leq 1$ and pick the smallest positive integer $d\geq 2$ such that
\begin{equation*}
	d+1 < d^{2s},
\end{equation*}
and we consider $0<\lambda_j\leq \tau_0<1/d$, for every $j$. Using the spatial antisymmetry of the kernel $\nabla_xP_s$ and techniques from Mateu and Tolsa’s work on Riesz kernels \cite{MT, T2}, our main result reads as follows:

\begin{thm*} Let $E_{p_s}$ be the $s$-parabolic Cantor set associated to $(\lambda_j)_j$, whose construction will be precised later on. If $\ell_j:=\lambda_1\cdots\lambda_j$, we have
\begin{equation*} C^{-1}\left(\sum_{j=0}^{\infty}\theta_{j,p_s}^2\right)^{-1/2}\leq \widetilde{\Gamma}_{\Theta^s,+}(E_{p_s})\leq \Gamma_{\Theta^s,+}(E_{p_s})\leq C\left(\sum_{j=0}^{\infty}\theta_{j,p_s}^2\right)^{-1/2}, 
\end{equation*} 
where $C$ depends only on $n,s$ and $\tau_0$ and
\begin{equation*}
    \theta_{j,p_s}:=\frac{\ell_j^{-(n+1)}}{(d+1)^jd^{nj}}.
\end{equation*}
\end{thm*}

This paper is organized into five sections that develop the proof of the main result. In Section \ref{sec2}, we fix important notation and present the explicit construction of the $s$-parabolic Cantor set. We also establish some basic properties related to the convolution operator associated with the kernel $\nabla_xP_s$ and a positive Borel regular measure $\mu$, denoted by $\pazocal{P}^s_\mu$.

Section \ref{sec3} is devoted to proving upper $L^2$ estimates for expressions of the form $\pazocal{P}^s_{\mu_k}\chi_Q$, where $Q$ is any $s$-parabolic cube and $\mu_k$ is the uniform probability measure associated with the $k$-th generation of the Cantor set. This estimate is essential for proving the lower bound of the main theorem in Section \ref{sec4}, where we apply a $T1$-theorem valid in geometrically doubling spaces.

In Section \ref{sec5}, we derive lower $L^2$ estimates for $\pazocal{P}^s_{\mu_k}1$ using arguments analogous to those of Tolsa in \cite[\textsection 5]{T2}, originally developed for Riesz kernels. These estimates are then used in Section \ref{sec6} to establish the remaining upper bound of the main theorem, relying on a local $Tb$ theorem by Auscher and Routin \cite[Theorem 3.5]{AR}.

It is in these final arguments that the distinct nature of the kernel $\nabla_xP_s$, in contrast to Riesz kernels, becomes evident. Due to the lack of anti-symmetry, we must construct a system of accretive functions that highlights how $\nabla_xP_s$ and its conjugate differ by a temporal reflection.

\textit{About the notation used in the sequel}: Constants appearing in the sequel may depend on the dimension of the ambient space and the  parameter $s$, and their value may change at different occurrences. They will frequently be denoted by the letters $c$ or $C$. The notation $A\lesssim B$ means that there exists $C$, such that $A\leq CB$. Moreover, $A\approx B$ is equivalent to $A\lesssim B \lesssim A$, while $A \simeq B$ will mean $A= CB$. If the reader finds expressions of the form $\lesssim_{\beta}$ or $\approx_\beta$, for example, this indicates that the implicit constants depend on $n,s$ and $\beta$.
    
	Since Laplacian operators (fractional or not) will frequently appear in our discussion and will always be taken with respect to spatial variables, we will adopt the notation:
	\begin{equation*}
		(-\Delta)^s:=(-\Delta_x)^s, \qquad s\in(0,1].
	\end{equation*}
    From this point on, we shall fix $s\in(1/2,1]$ throughout the whole text.
    
	\section{Basic definitions and properties}
    \label{sec2}
    Let us briefly recall the construction presented in \cite[\textsection 4.1]{H} of the particular Cantor set associated with the capacity $\Gamma_{\Theta^s}$. Choose a positive integer $d\geq 2$ such that
\begin{equation*}
	d+1 < d^{2s}.
\end{equation*}
By convention we fix the minimum value of $d\geq 2$ that satisfies the above condition so that $d=d(s)$. In what follows, we shall also fix an absolute constant $\tau_0<1/d$.

Let $Q^0:=[0,1]^{n+1}$ be the unit cube of $\mathbb{R}^{n+1}$ and pick $(d+1)d^n$ disjoint $s$-parabolic cubes $Q_i^1$, contained in $Q^0$, with sides parallel to the coordinate axes, side length $0<\lambda_1\leq \tau_0$, and with the following locations: for each of the first $n$ intervals $[0,1]$ of the cartesian product defining $Q^0$, we set
\begin{equation*}
	l_d:=\frac{1-d\lambda_1}{d-1}, \qquad J_j:=\big[(j-1)(\lambda_1+l_d), j\lambda_1+(j-1)l_d\big], \quad j=1,\ldots, d,
\end{equation*}
and take $T_d:=\bigcup_{j=1}^d J_j$. The remaining temporal interval $[0,1]$ is split in $d+1$ intervals of length $\lambda_1^{2s}$ in an analogous way. That is, we set
\begin{equation*}
	\widetilde{l}_{d}:=\frac{1-(d+1)\lambda_1^{2s}}{d}, \qquad \widetilde{J}_j:=\big[(j-1)(\lambda_1^{2s}+\widetilde{l}_{d}), j\lambda_1^{2s}+(j-1)\widetilde{l}_d\big], \quad j=1,\ldots, d+1,
\end{equation*}
and keep the subset $\widetilde{T}_d:=\bigcup_{j=1}^{d+1} \widetilde{J}_j$. This way, the first generation of the Cantor set is
\begin{equation*}
	E_{1,p_s}:=(T_d)^n\times \widetilde{T}_d.
\end{equation*}
This procedure continues inductively, so that the $k$-th generation $E_{k,p_s}$ will be formed by $(d+1)^kd^{nk}$ disjoint $s$-parabolic cubes with side length $\ell_k := \lambda_1\cdots \lambda_k$, $0<\lambda_k\leq \tau_0<1/d$, and with locations determined by the above iterative process. We name such cubes $Q^k_j$, with $j=1,\ldots, (d+1)^kd^{nk}$. The resulting $s$-\textit{parabolic} Cantor set is
\begin{equation}
\label{eq4.1.1}
	E_{p_s} = E_{p_s}(\lambda) :=  \bigcap_{k=1}^{\infty} E_{k,p_s}.
\end{equation}
For each generation $k$, consider the probability measure
   \begin{equation*}
           \mu_k := \frac{1}{|E_{k,p_s}|}\pazocal{L}^{n+1}|_{E_{k,p_s}}.
   \end{equation*}
Moreover, given $0\leq j \leq k$ we define
\begin{equation*}
    \theta_{j,p_s} := \frac{\mu_k(Q^j_i)}{\ell_j^{n+1}} = \frac{1}{(d+1)^jd^{nj} \ell_j^{n+1}}.
\end{equation*}
In \cite[\textsection 4.1]{H} it is argued that if one chooses $\lambda_k := ((d+1)d^n)^{-1/(n+1)}$ for every generation $k$, then $0<\pazocal{H}_{p_s}^{n+1}(E_{p_s})<\infty$. Observe that, as a consequence, this would imply $\theta_{k,p_s}=1$, for all $k$.

    There are more general assumptions that imply the lower bound $\pazocal{H}_{p_s}^{n+1}(E_{p_s})>0$. For instance, assume that there exists $\kappa>0$ so that $\theta_{k,p_s} \leq \kappa$ for every $k$. Now consider the probability measure $\mu$ defined on $E_{p_s}$ such that for each generation $k$, $\mu(Q_j^k):=(d+1)^{-k}d^{-nk},\, 1 \leq j \leq (d+1)^kd^{nk}$. We claim that that the previous measure presents upper $s$-parabolic growth of degree $n+1$ with constant depending only on $n,s$ and $\kappa$. If this held, by Frostman's lemma \cite[Theorem 8.8]{Mat} we would get $\pazocal{H}_{p_s}^{n+1}(E_{p_s})>0$ and thus, in particular, $\text{dim}_{\pazocal{H}_{p_s}}(E_{p_s})\geq n+1$.
    
    To prove the desired growth of $\mu$, let $Q$ be any $s$-parabolic cube, that we may assume to be contained in $Q^0$, and pick $k$ with the property $\ell_{k+1}\leq \ell(Q)\leq \ell_k$, so that $Q$ can meet, at most, $(d+1)d^n$ cubes $Q_j^k$. Thus $\mu(Q)\leq (d+1)^{-(k-1)}d^{-n(k-1)}$ and we directly deduce
\begin{align*}
	\mu(Q)\lesssim \frac{1}{(d+1)^{k+1}d^{n(k+1)}}\leq \kappa \ell_{k+1}^{n+1}\leq \ell(Q)^{n+1}.
\end{align*}
    
    For a fixed generation $k$, the assumption $\theta_{j,p_s} \leq \kappa$ for $0\leq j \leq k$ also implies the same growth property for $\mu_k$. Indeed, fix $Q\subset \mathbb{R}^{n+1}$ any $s$-parabolic cube contained in $Q^0$ and distinguish two cases: whether if $\ell(Q)\leq \ell_k$ or not. If $\ell(Q)\leq \ell_k$, notice that $|E_{k,p_s}|=(d+1)^kd^{nk}\ell_k^{n+2s}$, so we have
	   \begin{align*}
	   		\mu_k(Q)\leq \frac{1}{(d+1)^kd^{nk}\ell_k^{n+2s}}\,\ell(Q)^{n+2s}=\theta_{k,p_s}\frac{\ell(Q)^{2s-1}}{\ell_k^{2s-1}}\ell(Q)^{n+1}\leq \kappa \ell(Q)^{n+1}.
	   \end{align*}
	If $\ell(Q)> \ell_k$, there exists $0\leq m \leq k-1$ such that $\ell_{m+1}< \ell(Q)\leq\ell_{m}$ and, in this case, the number of cubes of the $m$-th generation that $Q$ can intersect is bounded by $(d+1)d^n$ (the latter is not the best bound, but it suffices for our computations). Therefore
	  \begin{align*}
	   	\mu_k(Q)&\leq (d+1)d^n\mu_k(Q^m_0) \simeq ((d+1)d^n)^{-m}\simeq  \theta_{m+1,p_s}\,\ell_{m+1}^{n+1} \leq \kappa \ell(Q)^{n+1},
	\end{align*}
	and we deduce the desired growth of $\mu_k$.
	
	Let us fix some more notation and establish some useful properties. In the sequel, $k$ will be a fixed integer so that $E_{k,p_s}$ will be a fixed generation of the Cantor set constructed above and we will write
    \begin{equation*}
        \|\cdot\|:=\|\cdot\|_{L^2(\mu_k)} \quad \text{as well as} \quad \langle f,g \rangle := \int fg \dd\mu_k 
    \end{equation*}
    For each generation $0\leq j\leq k$ we write
    \begin{equation*}
          \pazocal{Q}^j:=\{Q_i^j\,:\, 1\leq i\leq (d+1)^jd^{jn}\} \qquad \text{and} \qquad \pazocal{Q}:=\bigcup_{j=0}^k \pazocal{Q}^j.
    \end{equation*}
    If we write $Q^j$ we mean an arbitrary cube of $\pazocal{Q}^j$.

    Given $\mu$ a real compactly supported Borel measure with upper $s$-parabolic growth of degree $n+1$, we define for each $f\in L^1_{\text{loc}}(\mu)$ the convolution operator associated to $\nabla_xP_s$,
\begin{equation*}
	\pazocal{P}_{\mu}^{s}f(\ox):=\int_{\mathbb{R}^{n+1}}\nabla_x P_s(\ox-\oy)f(\oy)\text{d}\mu(\oy), \hspace{0.5cm} \ox\notin \text{supp}(\mu).
\end{equation*}
In the particular case in which $f$ is the constant function $1$ we write $\pazocal{P}^{s}\mu(\ox):=\pazocal{P}^s_{\mu}1 (\ox)$. We also introduce the truncated version of $\pazocal{P}_{\mu}^{s}$,
\begin{equation*}
	\pazocal{P}_{\mu,\varepsilon}^sf(\ox):=\int_{|\ox-\oy|>\varepsilon}\nabla_xP_s(\ox-\oy)f(\oy)\text{d}\mu(\oy),\hspace{0.5cm} \ox\in \mathbb{R}^{n+1}, \; \varepsilon>0.
\end{equation*}
For a given $1\leq p \leq \infty$, we will say that $\pazocal{P}_{\mu}^{s} f$ belongs to $L^p(\mu)$ if the $L^p(\mu)$-norm of the truncations $\|\pazocal{P}^s_{\mu,\varepsilon}f\|_{L^p(\mu)}$ is  uniformly bounded on $\varepsilon$, and we write
	\begin{equation*}
		\|\pazocal{P}_{\mu}^{s} f\|_{L^p(\mu)}:=\sup_{\varepsilon>0}\|\pazocal{P}^s_{\mu,\varepsilon}f\|_{L^p(\mu)}
	\end{equation*}
	We will say that the operator $\pazocal{P}^s_\mu$ is bounded on $L^p(\mu)$ if the operators $\pazocal{P}^s_{\mu,\varepsilon}$ are bounded on $L^p(\mu)$ uniformly on $\varepsilon$, and we equally set
	\begin{equation*}
		\|\pazocal{P}^s_\mu\|_{L^p(\mu)\to L^p(\mu)}:=\sup_{\varepsilon>0}\|\pazocal{P}^s_{\mu,\varepsilon}\|_{L^p(\mu)\to L^p(\mu)}.
	\end{equation*}
    Let us begin by establishing some basic properties of $\pazocal{P}_\mu^s$.
    \begin{lem}
    \label{lem1.1}
        For any $R\subset\mathbb{R}^{n+1}$\, $s$-parabolic cube, we have $\pazocal{P}^s_{\mu_k} \chi_R \in L^1_{\text{\normalfont{loc}}}(\mu_k)$. This implies, in particular, $\pazocal{P}^s \mu_k \in L^1_{\text{\normalfont{loc}}}(\mu_k).$
    \end{lem}
    \begin{proof}
        Given any $s$-parabolic cube $Q^k\in \pazocal{Q}^k$ and $R\in \pazocal{Q}$,
    \begin{align*}
        \int_{Q^k}\int_{R}\frac{\dd \mu_k(\oy)\dd \mu_k(\ox)}{|\ox-\oy|_{p_s}^{n+1}} &= \int_{Q^k}\int_{R\cap E_{k,p_s}}\frac{\dd \mu_k(\oy)\dd \mu_k(\ox)}{|\ox-\oy|_{p_s}^{n+1}}\\
            &\hspace{-2cm} = \int_{Q^k}\int_{(R\cap E_{k,p_s})\setminus{Q^k}}\frac{\dd \mu_k(\oy)\dd \mu_k(\ox)}{|\ox-\oy|_{p_s}^{n+1}} + \int_{Q^k}\int_{Q^k}\frac{\dd \mu_k(\oy)\dd \mu_k(\ox)}{|\ox-\oy|_{p_s}^{n+1}}=:\text{I}+\text{II}.
        \end{align*}
    Observe that the points of $Q^k$ and $(R\cap E_{k,p_s})\setminus{Q^k}$ are separated, at least, by an $s$-parabolic distance comparable to
    \begin{equation*}
            \min\bigg\{ \ell_{k-1}\bigg(\frac{1-d\lambda_k}{d-1}\bigg),\ell_{k-1} \bigg( \frac{1-(d+1)\lambda_1^{2s}}{d} \bigg)^{\frac{1}{2s}} \bigg\} \gtrsim \ell_{k-1},
        \end{equation*}
    where the previous implicit constant depends on $\tau_0$ and $s$. Therefore, it is clear that $\text{I}\lesssim 1/\ell_{k-1}^{n+1}<\infty$. On the other hand, to deal with $\text{II}$, for each $\oy\in Q^k$ we shall contain $Q^k$ in the $s$-parabolic cube $\widetilde{Q}$ centered at $\oy$
    with side length $2\,\text{diam}_{p_s}(Q^k)$, and split the previous set into the $s$-parabolic annuli
    \begin{equation*}
            A_{j}:= Q\big(\oy, 2^{-j}\text{diam}_{p_s}(Q^k)\big)\setminus Q\big(\oy, 2^{-j-1}\text{diam}_{p_s}(Q^k)\big), \;\; j\geq -1.
        \end{equation*}
    Hence, by definition of $\mu_k$ we get, for each $\oy\in Q^k$,
    \begin{align*}
            \int_{Q^k}\frac{\dd \mu_k(\ox)}{|\ox-\oy|_{p_s}^{n+1}}&\leq \frac{1}{|E_{k,p_s}|}\sum_{j=-1}^{\infty}\int_{A_j} \frac{\dd \ox}{|\ox-\oy|_{p_s}^{n+1}} \lesssim \frac{1}{|E_{k,p_s}|}\sum_{j=-1}^{\infty} \frac{\big( 2^{-j}\ell(Q^k) \big)^{n+2s}}{\big( 2^{-j-1}\ell(Q^k) \big)^{n+1}} \\
            &\lesssim \frac{1}{(d+1)^kd^{nk} \ell_k^{n+1}} = \theta_{k,p_s}<\infty.
        \end{align*}
    Since $Q^k\in \QQ^k$ and $R\in \QQ$ were arbitrary, combining the previous estimates and the fourth estimate of \cite[Lemma 2.2]{MPr}, we deduce the desired result.
    \end{proof}
    \begin{lem}
    \label{lem1.2}
        For any $R\subset \mathbb{R}^{n+1}$ $s$-parabolic cube,
        \begin{equation*}
            \int_R\pazocal{P}^s_{\mu_k}\chi_R\dd \mu_k=0.
        \end{equation*}
    \end{lem}
    \begin{proof}
        Fix any $R\subset \mathbb{R}^{n+1}$ $s$-parabolic cube and notice that $R\cap E_{k,p_s}$ can be written as a translated copy of a cartesian product $\pazocal{X}_{R,k}\times \pazocal{T}_{R_k}\subset \mathbb{R}^n\times \mathbb{R}$, where $\pazocal{X}_{R,k}$ and $\pazocal{T}_{R,k}$ are sets contained in some  generations involved in the construction of a Cantor set in $\mathbb{R}^n$ and $\mathbb{R}$ respectively. Let us also observe that by \cite[Lemma 2.1]{HMPr} for any $(x,t)\neq (0,t)$ we have
    \begin{equation*}
        \nabla_xP_s(x,t) \simeq  xt^{-\frac{n+2}{2s}}\phi_{n+2,s}\big( |x|t^{-\frac{1}{2s}}\big)\chi_{t>0},
    \end{equation*}
    where $\phi_{n,s}$ is a smooth function, radially decreasing function. So for each $t\in \mathbb{R}$, the kernel $\nabla_xP_s(\cdot,t)$ is anti-symmetric. Then, by Fubini's theorem, that can be applied due to Lemma \ref{lem1.1},\medskip\\
    \begin{align*}
        \int_R\pazocal{P}^s_{\mu_k}\chi_R(\ox)&\dd \mu_k(\ox)=\int_R\int_R\nabla_xP_s(\ox-\oy)\dd \mu_k(\oy)\dd \mu_k(\ox)\\
        &=\frac{1}{|E_{k,p_s}|^2}\int_{R\cap E_{k,p_s}}\int_{R\cap E_{k,p_s}}\nabla_x P_s(\ox-\oy) \dd \oy\dd \ox\\
        &=\frac{1}{|E_{k,p_s}|^2}\int_{\pazocal{T}_{R,k}}\int_{\pazocal{T}_{R,k}}\Bigg( \int_{\pazocal{X}_{R,k}}\int_{\pazocal{X}_{R,k}} \nabla_x P_s(x-y,t-s)\dd y\dd x \Bigg)\dd s\dd t\\
        &=\frac{-1}{|E_{k,p_s}|^2}\int_{\pazocal{T}_{R,k}}\int_{\pazocal{T}_{R,k}}\Bigg( \int_{\pazocal{X}_{R,k}}\int_{\pazocal{X}_{R,k}} \nabla_x P_s(y-x,t-s)\dd x\dd y \Bigg)\dd s\dd t\\
        &=-\int_R\pazocal{P}^s_{\mu_k}\chi_R(\ox)\dd \mu_k(\ox).
    \end{align*}
    \end{proof}
	For each $Q\in\pazocal{Q}$ and $f\in L^1_{\text{loc}}(\mu_k)$ we write
    \begin{align*}
        S_Qf(\ox):=\bigg( \frac{1}{\mu_k(Q)}\int_Q f\dd\mu_k \bigg)\chi_Q(\ox), \qquad S_jf(\ox):=\sum_{Q\in\pazocal{Q}^j} S_Qf(\ox), \quad 0\leq j \leq k.
    \end{align*}
	Let $\pazocal{CH}(Q)$ be the set of $s$-parabolic cubes that are \textit{children} of $Q$. For any $Q\in \pazocal{Q}\setminus \pazocal{Q}^k=:\pazocal{Q}^\ast$ and $f\in L^1_{\text{loc}}(\mu_k)$ we also write
    \begin{align*}
        D_Qf(\ox):=\bigg[ \sum_{P\in \pazocal{CH}(Q)} S_Pf(\ox) \bigg] - S_Qf(\ox), \qquad D_jf(\ox):=\sum_{Q\in\pazocal{Q}^j} D_Qf(\ox), \quad 0\leq j \leq k-1.
    \end{align*}
    Notice that $\int D_Q f \dd\mu_k = 0$ for all $Q\in \QQ^\ast$ and
    \begin{equation*}
        D_jf(\ox)=S_{j+1}f(\ox)-S_jf(\ox), \qquad 0\leq j \leq k-1.
    \end{equation*}
    \begin{lem}
    \label{lem1.3}
        For any pair of different $s$-parabolic cubes $Q_1,Q_2\in \QQ^\ast$,
        \begin{equation*}
            \langle D_{Q_1}f, D_{Q_2}f \rangle = 0
        \end{equation*}
    \end{lem}
    \begin{proof}
        Let $Q_1,Q_2\in \QQ^\ast$ with disjoint support. Then it is clear that $\langle D_{Q_1}f, D_{Q_2}f \rangle = 0$. If on the contrary $Q_1\subsetneq Q_2$, let $\widetilde{Q}_2$ be the unique son of $Q_2$ with $Q_1\subseteq \widetilde{Q}_2$, and observe that
        \begin{align*}
            D_{Q_1}f(&\ox)\cdot D_{Q_2}f(\ox)\\
            &= \Bigg[ \Bigg(\sum_{P\in\pazocal{CH}(Q_1)}S_Pf(\ox)\Bigg)-S_{Q_1}f(\ox) \Bigg]\cdot \Bigg[ \Bigg(\sum_{R\in\pazocal{CH}(Q_2)}S_Pf(\ox)\Bigg)-S_{Q_2}f(\ox) \Bigg]\\
            &=\Bigg( \frac{1}{\mu_k(\widetilde{Q}_2)}\int_{\widetilde{Q}_2}f\,\dd \mu_k \Bigg)\cdot \sum_{P\in \pazocal{CH}(Q_1)}S_pf(\ox) -\Bigg( \frac{1}{\mu_k(Q_2)}\int_{Q_2}f\,\dd \mu_k \Bigg)\cdot \sum_{P\in \pazocal{CH}(Q_1)}S_pf(\ox)\\
            &\hspace{2cm}-\Bigg( \frac{1}{\mu_k(Q_2)}\int_{Q_2}f\,\dd \mu_k \Bigg)\cdot\Bigg( \frac{1}{\mu_k(\widetilde{Q}_2)}\int_{\widetilde{Q}_2}f\,\dd \mu_k \Bigg)\chi_{Q_1}(\ox)\\
            &\hspace{2cm}+\Bigg( \frac{1}{\mu_k(Q_2)}\int_{Q_2}f\,\dd \mu_k \Bigg)\cdot\Bigg( \frac{1}{\mu_k(Q_2)}\int_{Q_2}f\,\dd \mu_k \Bigg)\chi_{Q_1}(\ox)\\
            &=\Bigg( \frac{1}{\mu_k(\widetilde{Q}_2)}\int_{\widetilde{Q}_2}f\,\dd \mu_k \Bigg)\cdot D_{Q_1}f(\ox)-\Bigg( \frac{1}{\mu_k(Q_2)}\int_{Q_2}f\,\dd \mu_k \Bigg)\cdot D_{Q_1}f(\ox).
        \end{align*}
        Then, $\langle D_{Q_1}f, D_{Q_2}f \rangle = 0$. So, in general, if $Q_1\neq Q_2$ belong to $\QQ^\ast$, the functions $D_{Q_1}f$ and $D_{Q_2}f$ are orthogonal in an $L^2(\mu_k)$-sense.
    \end{proof}
	If $f\in L^1_{\text{loc}}(\mu_k)$ with $\int f\dd\mu_k = 0$, then Lemma \ref{lem1.3} implies $\|S_kf\|^2 = \sum_{Q\in\pazocal{Q}^\ast}\|D_Qf\|^2$, so in particular, applying Lemma \ref{lem1.2} with $R=Q^0$,
		\begin{equation}
			\label{eq1.5}
			\|S_k\Ps\mu_k\|^2 = \sum_{Q\in\pazocal{Q}^\ast}\|D_Q\Ps\mu_k\|^2.
		\end{equation}
	\section{The upper \mathinhead{L^2}{}-estimate for \mathinhead{\Ps\mu_k}{}}
    \label{sec3}
	\begin{lem}
		\label{lem4.1.1}
		Let $Q\in \pazocal{Q}^j$, for any $0\leq j \leq k$, and $\ox, \ox'\in Q$. Then,
		\begin{equation*}
			\big\rvert \pazocal{P}^s_{\mu_k}\chi_{\mathbb{R}^{n+1}\setminus{Q}}(\ox) - \pazocal{P}^s_{\mu_k}\chi_{\mathbb{R}^{n+1}\setminus{Q}} (\ox') \big\rvert \lesssim_{\tau_0} p(Q),
		\end{equation*}
		where $ p(R):=\sum_{r=0}^{j}\theta_{r,p_s}\frac{\ell_j}{\ell_r}$, for $R\in \pazocal{Q}^j$.
	\end{lem}
	\begin{proof}
		It is clear that if $j=0$ each term of the above difference is null, so let us assume $j>0$. Let $\widehat{Q}$ be the parent of $Q$ and write $\ox=(x,t), \ox'=(x',t')$ and $\widehat{x}:=(x',t)$. Then, applying the mean value theorem component-wise similarly as in \cite[Theorem 2.2]{HMPr} and writing 
        \begin{equation*}
            \ox\in Q_k(\ox)\subset Q_{k-1}(\ox)\subset \cdots \subset Q_1(\ox)\subset Q_0(\ox)=Q^0,
        \end{equation*}
        the chain of $s$-parabolic cubes of each generation that contains $\ox$ successively, we have
		\begin{align*}
			\big\rvert \pazocal{P}^s_{\mu_k}\chi_{\mathbb{R}^{n+1}\setminus{Q}} &(\ox) - \pazocal{P}^s_{\mu_k}\chi_{\mathbb{R}^{n+1}\setminus{Q}} (\ox') \big\rvert \\
			&\hspace{-0.95cm}\leq \int_{\mathbb{R}^{n+1}\setminus{Q}} \big\rvert \nabla_x P_s(\ox-\oy)-\nabla_x P_s(\widehat{x}-\oy) \big\rvert \dd \mu_k(\oy)\\
			&\hspace{2.55cm}+\int_{\mathbb{R}^{n+1}\setminus{Q}} \big\rvert \nabla_x P_s(\widehat{x}-\oy)-\nabla_x P_s(\ox'-\oy) \big\rvert \dd \mu_k(\oy)\\
			&\hspace{-0.95cm}\lesssim \int_{\mathbb{R}^{n+1}\setminus{Q}}\frac{|x-x'|}{|\ox-\oy|_{p_s}^{n+2}}\dd \mu_k(\oy)+\int_{\mathbb{R}^{n+1}\setminus{Q}}\frac{|t-t'|}{|\ox-\oy|_{p_s}^{n+2s+1}}\dd \mu_k(\oy)\\
			&\hspace{-0.95cm}\lesssim \ell(Q)\sum_{r=1}^{j}\int_{Q_{r-1}(\ox)\setminus{Q_{r}(\ox)}} \frac{\dd \mu_k(\oy)}{|\ox-\oy|_{p_s}^{n+2}}+\ell(Q)^{2s}\sum_{r=1}^{j}\int_{Q_{r-1}(\ox)\setminus{Q_{r}(\ox)}} \frac{\dd \mu_k(\oy)}{|\ox-\oy|_{p_s}^{n+2s+1}}\\
			&\hspace{-0.95cm}\lesssim \ell(Q)\sum_{r=0}^{j-1}\frac{\mu_k(Q_{r}(\ox))}{\ell_{r}^{n+2}}+\ell(Q)^{2s}\sum_{r=0}^{j-1}\frac{\mu_k(Q_{r}(\ox))}{\ell_{r}^{n+2s+1}}\\
			&\hspace{-0.95cm}\lesssim \frac{\ell(Q)}{\ell(\widehat{Q})}\sum_{r=0}^{j-1}\theta_{r,p_s}\frac{\ell_{j-1}}{\ell_{r}}+\frac{\ell(Q)^{2s}}{\ell(\widehat{Q})^{2s}}\sum_{r=0}^{j-1}\theta_{r,p_s}\frac{\ell_{j-1}^{2s}}{\ell_{r}^{2s}}\leq \Bigg[ \frac{\ell(Q)}{\ell(\widehat{Q})} + \frac{\ell(Q)^{2s}}{\ell(\widehat{Q})^{2s}} \Bigg]p(\widehat{Q})\lesssim \frac{\ell(Q)}{\ell(\widehat{Q})}p(\widehat{Q}).
		\end{align*}
		Finally observe that
		\begin{align*}
			p(\widehat{Q})&=\theta_{j-1,p_s}+\big(\theta_{j-2,p_s}\lambda_{j-1}\big)+\cdots+\big(\theta_{1,p_s}\lambda_{j-1}\cdots\lambda_2\big)+\big(\lambda_{j-1}\cdots\lambda_1\big)\\
			&=\frac{1}{\lambda_j}\big( p(Q)-\theta_{j,p_s}\big)\leq \frac{1}{\lambda_{j}}p(Q) = \frac{\ell(\widehat{Q})}{\ell(Q)}p(Q),
		\end{align*}
		and the result follows.
	\end{proof}
	\begin{lem}
		\label{lem4.1.2}
		If $Q\in \pazocal{Q}^j$ with $j<k$, then
		\begin{equation*}
			\bigg\rvert  S_Q(\pazocal{P}^s \mu_k) - \sum_{P\in \pazocal{CH}(Q)} S_P(\pazocal{P}^s \mu_k) \bigg\rvert\lesssim_{\tau_0} p(Q).
		\end{equation*}
	\end{lem}
	\begin{proof}
		Notice that by Lemma \ref{lem1.2}, for any $P\in \pazocal{CH}(Q)$ we have $S_P(\pazocal{P}^s_{\mu_k}\chi_{P})=0$ and $S_Q(\pazocal{P}^s_{\mu_k}\chi_{Q})=0$. Hence,
		\begin{align*}
			\bigg\rvert  S_Q(\pazocal{P}^s \mu_k) &- \sum_{P\in \pazocal{CH}(Q)} S_P(\pazocal{P}^s \mu_k) \bigg\rvert = \bigg\rvert  S_Q(\pazocal{P}^s_{\mu_k}\chi_{\mathbb{R}^{n+1}\setminus{Q}}) - \sum_{P\in \pazocal{CH}(Q)} S_P(\pazocal{P}^s_{\mu_k}\chi_{\mathbb{R}^{n+1}\setminus{P}})\bigg\rvert\\
			&\leq \sum_{P\in \pazocal{CH}(Q)}|S_P(\pazocal{P}^s_{\mu_k}\chi_{Q\setminus{P}})|+\bigg\rvert  S_Q(\pazocal{P}^s_{\mu_k}\chi_{\mathbb{R}^{n+1}\setminus{Q}}) - \sum_{P\in \pazocal{CH}(Q)} S_P(\pazocal{P}^s_{\mu_k}\chi_{\mathbb{R}^{n+1}\setminus{Q}})\bigg\rvert.
		\end{align*}
		It is clear that $|\pazocal{P}^s_{\mu_k}\chi_{Q\setminus{P}}(\ox)|\lesssim \mu_k(Q)/\ell(Q)^{n+1}=\theta_{j,p_s}\leq p(Q)$, for each $\ox\in P$. So the first sum satisfies
		\begin{equation*}
			\sum_{P\in \pazocal{CH}(Q)}|S_P(\pazocal{P}^s_{\mu_k}\chi_{Q\setminus{P}})|\lesssim \sum_{P\in \pazocal{CH}(Q)}\theta_{j,p_s}\chi_P\leq \theta_{j,p_s}\leq p(Q).
		\end{equation*}
		For the remaining term write
		\begin{align*}
			\bigg\rvert  &S_Q(\pazocal{P}^s_{\mu_k}\chi_{\mathbb{R}^{n+1}\setminus{Q}}) - \sum_{P\in \pazocal{CH}(Q)} S_P(\pazocal{P}^s_{\mu_k}\chi_{\mathbb{R}^{n+1}\setminus{Q}})\bigg\rvert\\
			&=\bigg\rvert \sum_{P\in \pazocal{CH}(Q)} \bigg( \frac{1}{\mu_k(P)}\int_P \pazocal{P}^s_{\mu_k}\chi_{\mathbb{R}^{n+1}\setminus{Q}}(\ox)\,\dd \mu_k(\ox) - \frac{1}{\mu_k(Q)}\int_Q \pazocal{P}^s_{\mu_k}\chi_{\mathbb{R}^{n+1}\setminus{Q}}(\ox)\,\dd \mu_k(\ox) \bigg)\chi_P \bigg\rvert\\
			&=\bigg\rvert \sum_{P\in \pazocal{CH}(Q)}\sum_{P'\in \pazocal{CH}(Q)} \bigg( \frac{1}{(d+1)d^n\mu_k(P)}\int_P \pazocal{P}^s_{\mu_k}\chi_{\mathbb{R}^{n+1}\setminus{Q}}(\ox)\,\dd \mu_k(\ox) \\
			&\hspace{7cm}- \frac{1}{\mu_k(Q)}\int_{P'} \pazocal{P}^s_{\mu_k}\chi_{\mathbb{R}^{n+1}\setminus{Q}}(\ox)\,\dd \mu_k(\ox) \bigg)\chi_P \bigg\rvert\\
			&=\bigg\rvert \sum_{P\in \pazocal{CH}(Q)}\sum_{P'\in \pazocal{CH}(Q)}  \frac{1}{\mu_k(Q)} \bigg(\int_P \pazocal{P}^s_{\mu_k}\chi_{\mathbb{R}^{n+1}\setminus{Q}}(\ox)\,\dd \mu_k(\ox)-\int_{P'} \pazocal{P}^s_{\mu_k}\chi_{\mathbb{R}^{n+1}\setminus{Q}}(\ox)\,\dd \mu_k(\ox) \bigg)\chi_P\bigg\rvert\\
			&\leq \sum_{P\in \pazocal{CH}(Q)}\sum_{P'\in \pazocal{CH}(Q)}  \frac{1}{\mu_k(Q)}\bigg( \int_Q \big\rvert \pazocal{P}^s_{\mu_k}\chi_{\mathbb{R}^{n+1}\setminus{Q}}(\ox)-\pazocal{P}^s_{\mu_k}\chi_{\mathbb{R}^{n+1}\setminus{Q}}(\tau_{P\to P'}(\ox)) \big\rvert \dd \mu_k(\ox) \bigg)\chi_P,
		\end{align*}
		where $\tau_{P\to P'}$ is the translation of $\mathbb{R}^{n+1}$ satisfying $\tau_{P\to P'}(P)=P'$. Thus, by Lemma \ref{lem4.1.1},
		\begin{align*}
			\bigg\rvert  S_Q(\pazocal{P}^s_{\mu_k}\chi_{\mathbb{R}^{n+1}\setminus{Q}}) &- \sum_{P\in\pazocal{CH}(Q)} S_P(\pazocal{P}^s_{\mu_k}\chi_{\mathbb{R}^{n+1}\setminus{Q}})\bigg\rvert \\
			&\lesssim \sum_{P\in \pazocal{CH}(Q)}\sum_{P'\in \pazocal{CH}(Q)}  p(Q)\chi_P = (d+1)d^np(Q)\chi_Q \lesssim p(Q).
		\end{align*}
	\end{proof}
	\begin{rem}
		\label{rem4.1.5}
		Observe that as an immediate consequence of the previous lemma we have
		\begin{align*}
			\big\| D_Q(\pazocal{P}^s \mu_k) \big\|^2= \int_Q \bigg\rvert S_Q(\pazocal{P}^s \mu_k)-\sum_{P\in \pazocal{CH}(Q)}S_P(\pazocal{P}^s \mu_k) \bigg\rvert^2\dd \mu_k\lesssim_{\tau_0} p(Q)^2\mu_k(Q).
		\end{align*}
	\end{rem}
	\begin{lem}
		\label{lem4.1.3}
		Let $M\geq 0$ be an integer and $Q^j\in \pazocal{Q}^j$, for $0\leq j \leq M$. Then,
		\begin{equation*}
			\sum_{j=0}^{M}p(Q^j)^2 \lesssim \sum_{j=0}^M \theta_{j,p_s}^2.
		\end{equation*}
	\end{lem}
	\begin{proof}
		It follows from Cauchy-Schwarz's inequality and the following computation:
		\begin{align*}
			\sum_{j=0}^{M}p(Q^j)^2 &=\sum_{j=0}^{M}\bigg( \sum_{r=0}^j \theta_{r,p_s}\frac{\ell_j}{\ell_r}\bigg)^2=\sum_{j=0}^{M}\ell_j^2\bigg( \sum_{r=0}^j \frac{\theta_{r,p_s}}{\sqrt{\ell_r}}\frac{1}{\sqrt{\ell_r}}\bigg)^2\\
			&\leq \sum_{j=0}^{M}\ell_j^2\bigg(\sum_{r=0}^j \frac{\theta_{r,p_s}^2}{\ell_r}\bigg) \bigg( \sum_{r=0}^j \frac{1}{\ell_r}\bigg)\leq \sum_{j=0}^{M}\bigg(\sum_{r=0}^j \theta_{r,p_s}^2 \frac{\ell_j}{\ell_r}\bigg) \bigg( \sum_{r=0}^j \frac{1}{d^r}\bigg)\\
			&\lesssim \sum_{j=0}^M \sum_{r=0}^j \theta_{r,p_s}^2\frac{\ell_j}{\ell_r} =\sum_{r=0}^M\theta_{r,p_s}^2\sum_{j=r}^M\frac{\ell_j}{\ell_r}\lesssim \sum_{r=0}^M\theta_{r,p_s}^2.
		\end{align*}
	\end{proof}
	The previous three lemmas will be used to prove the next auxiliary estimate, analogous to \cite[Theorem 3.1]{T1}.  Once proved, it will imply Lemma \ref{lem4.1.7}, the main result of this subsection.
	\begin{lem}
		\label{lem4.1.4}
		The following estimate holds:
		\begin{equation*}
			\| \pazocal{P}^s \mu_k \|^2\lesssim_{\tau_0} \sum_{j=0}^k \theta_{j,p_s}^2.
		\end{equation*}
	\end{lem}
	\begin{proof}
		Begin by noticing that \eqref{eq1.5} and Remark \ref{rem4.1.5} imply
		\begin{align*}
			\big\| S_k (\pazocal{P}^s \mu_k) \big\|^2 &= \sum_{Q\in \pazocal{Q}\setminus{\pazocal{Q}^k}} \big\| D_Q (\pazocal{P}^s \mu_k) \big\|^2\lesssim \sum_{Q\in \pazocal{Q}\setminus{\pazocal{Q}^k}} p(Q)^2\mu(Q) \\
			&= \sum_{j=0}^{k-1}\sum_{Q^j\in \pazocal{Q}^j} p(Q^j)^2\mu(Q^j) = \sum_{j=0}^{k-1}p(Q^j)^2.
		\end{align*}
		Moreover, by Lemma \ref{lem1.2} and Lemma \ref{lem4.1.1} we also have for each $Q\in \pazocal{Q}^k$ and $\ox\in Q$,
		\begin{align*}
			\big\rvert S_Q(\pazocal{P}^s \mu_k)(\ox) &- \chi_Q\pazocal{P}^s_{\mu_k}\chi_{\mathbb{R}^{n+1}\setminus{Q}}(\ox) \big\rvert \\
			&= \frac{1}{\mu(Q)} \int_Q\big\rvert \pazocal{P}^s_{\mu_k}\chi_{\mathbb{R}^{n+1}\setminus{Q}}(\oy) - \pazocal{P}^s_{\mu_k}\chi_{\mathbb{R}^{n+1}\setminus{Q}}(\ox)\big\rvert\dd \mu_k(\oy)\lesssim p(Q).
		\end{align*}
		In addition, for each $Q\in \pazocal{Q}^k$ and $\ox\in Q$, \cite[Lemma 2.2]{MPr} and polar integration yield,
		\begin{equation*}
			\big\rvert \pazocal{P}^s_{\mu_k}\chi_Q (\ox) \big\rvert \lesssim \frac{1}{|E_{k,p_s}|}\int_Q\frac{\dd \oy}{|\ox-\oy|_{p_s}^{n+1}}\lesssim \frac{\ell_k}{|E_{k,p_s}|}=\theta_{k,p_s}.
		\end{equation*}
		Notice the need of $s>1/2$ in the previous estimate. Combining the three above computations and Lemma \ref{lem4.1.3} we finally conclude:
		\begin{align*}
			\|\pazocal{P}^s \mu_k\|^2 &= \sum_{Q\in \pazocal{Q}^k}\|\chi_Q\pazocal{P}^s \mu_k \|^2 \leq 2\sum_{Q\in \pazocal{Q}^k}\Big( \|\chi_Q\pazocal{P}^s_{\mu_k}\chi_Q \|^2 + \|\chi_Q\pazocal{P}^s_{\mu_k}\chi_{\mathbb{R}^{n+1}\setminus{Q}} \|^2 \Big)\\
			&\lesssim \sum_{Q\in \pazocal{Q}^k}\Big( \|\chi_Q\pazocal{P}^s_{\mu_k}\chi_Q \|^2 + \|\chi_Q\pazocal{P}^s_{\mu_k}\chi_{\mathbb{R}^{n+1}\setminus{Q}} -S_Q(\pazocal{P}^s \mu_k) \|^2 + \|S_Q(\pazocal{P}^s \mu_k) \|^2 \Big)\\
			&\lesssim \sum_{Q\in \pazocal{Q}^k}\theta_{k,p_s}^2\mu_k(Q)+\sum_{Q\in \pazocal{Q}^k}p(Q)^2\mu_k(Q) + \|S_k(\pazocal{P}^s \mu_k) \|^2\\
			&\lesssim \theta_{k,p_s}^2+p(Q^k)^2+\sum_{j=0}^{k-1}p(Q^j)^2\lesssim \sum_{j=0}^kp(Q^j)^2\lesssim \sum_{j=0}^{k}\theta_{j,p_s}^2.
		\end{align*}
	\end{proof}
	
	Notice that we have just proved an $L^2(\mu_k)$-bound for $\pazocal{P}^s \mu_k = \pazocal{P}^s_{\mu_k}\chi_{Q^0}$. Our next goal is to obtain a bound for $\pazocal{P}^s_{\mu_k}\chi_{Q^m}$ for any $s$-parabolic cube $Q^m$ of the $m$-th generation, $0 < m \leq k$, generalizing the estimate of Lemma \ref{lem4.1.4} if $m=0$. But as it is pointed out in \cite[\textsection 3]{T1}, the procedure to obtain the estimate for $\pazocal{P}^s_{\mu_k}\chi_{Q^0}$ already illustrates the computations needed to deduce the corresponding estimate for $\pazocal{P}^s_{\mu_k}\chi_{Q^m}$. This is due to the \textit{self-similarity} of the Cantor set we consider.
	
	Let us tackle first the case $0<m<k$ (the arguments that follow will be general enough to allow us to set $m=0$ and recover all the previous results). Fix a cube $Q^m$ of the $m$-th generation and consider the following truncation of $\mu_k$,
	\begin{equation*}
		\mu_{k,m}:=\frac{1}{\mu_k(Q^m)}\mu_k|_{Q^m}.
	\end{equation*}
	It is clear that Lemma \ref{lem1.1} and Lemma \ref{lem1.2} are also valid in this setting. More precisely, performing essentially the same computations we get $\pazocal{P}^s_{\mu_{k,m}}\chi_R\in L^1_{\text{loc}}(\mu_{k,m})$ as well as
	\begin{equation}
		\label{eq4.1.2}
		\int_{R}\pazocal{P}^s_{\mu_{k,m}}\chi_{R}\, \dd \mu_{k,m}=0, \hspace{0.75cm} \forall R \in \pazocal{Q}.
	\end{equation}
	We also consider the set (analogous to $\pazocal{Q}$)
	\begin{equation*}
		\pazocal{Q}(m):=\bigcup_{j=m}^k\pazocal{Q}^j\cap Q^m,
	\end{equation*}
	and the functions (analogous to $S_Q, S_j, D_Q$ and $D_j$) defined for $f\in L^1_{\text{loc}}(\mu_{k,m})$,
	\begin{align*}
		S_Q^mf(\ox)&:=\Bigg( \frac{1}{\mu_{k,m}(Q)}\int_Q f\,\dd \mu_{k,m} \Bigg)\chi_Q(\ox), \hspace{2.6cm} Q\in \pazocal{Q}(m),\\
		S_j^mf(\ox)&:=\sum_{Q\in \pazocal{Q}^j\cap Q^m} S_{Q}^mf(\ox), \hspace{4.6cm} m\leq j \leq k,\\
		D_Q^mf(\ox)&:=\Bigg(\sum_{P\in \pazocal{CH}(Q)}S_P^mf(\ox)\Bigg)-S_Q^mf(\ox), \hspace{2.6cm} Q\in \pazocal{Q}(m)\setminus{(\pazocal{Q}^k\cap Q^m)},\\
		D_j^mf(\ox)&:=\sum_{Q\in \pazocal{Q}^j\cap Q^m}D_Q^mf(\ox)=S_{j+1}^mf(\ox)-S_j^m f(\ox), \hspace{0.75cm} m\leq j \leq k-1.
	\end{align*}
	It is clear that $\int D^m_Qf\, \dd \mu_{k,m}=0$ for any $Q\in  \pazocal{Q}(m)\setminus{(\pazocal{Q}^k\cap Q^m)}$. So analogously to the case $m=0$, $D_{Q_1}^mf$ and $D_{Q_2}^mf$ are orthogonal in an $L^2(\mu_{k,m})$-sense if $Q_1\neq Q_2$ belong to $\pazocal{Q}(m)\setminus{(\pazocal{Q}^k\cap Q^m)}$. Thus, Lemma \ref{lem1.3} also admits a generalization in the current setting. Moreover, if $f:=\pazocal{P}^s \mu_{k,m}$, by \eqref{eq4.1.2} we have a $Q^m$-truncated version of \eqref{eq1.5},
	\begin{equation}
		\label{eq4.1.3}
		\big\| S_k^m (\pazocal{P}^s \mu_{k,m}) \big\|_{L^2(\mu_{k,m})}^2 = \sum_{Q\in \pazocal{Q}(m)\setminus{(\pazocal{Q}^k\cap Q^m})} \big\| D_Q^m (\pazocal{P}^s \mu_{k,m}) \big\|_{L^2(\mu_{k,m})}^2.
	\end{equation}
	The previous relations allow to generalize Lemmas \ref{lem4.1.1} and \ref{lem4.1.2}. We will only give the details of the proof of the former since they suffice to illustrate that the methods of proof are analogous to those presented for the aforementioned lemmas.
	\begin{lem}
		\label{lem4.1.5}
		Let $Q\in \pazocal{Q}^j\cap Q^m$ for $m\leq j \leq k$, and $\ox, \ox'\in Q$. Then,
		\begin{equation*}
			\big\rvert \pazocal{P}^s_{\mu_{k,m}}\chi_{\mathbb{R}^{n+1}\setminus{Q}} (\ox) - \pazocal{P}^s_{\mu_{k,m}}\chi_{\mathbb{R}^{n+1}\setminus{Q}} (\ox') \big\rvert \lesssim_{\tau_0} \frac{1}{\mu_k(Q^m)} p_m(Q),
		\end{equation*}
		where now $p_m(R):=\sum_{r=m}^{j}\theta_{r,p_s}\frac{\ell_j}{\ell_r}$, for $R\in \pazocal{Q}^j$.
	\end{lem}
	\begin{proof}
		The proof is analogous to that of Lemma \ref{lem4.1.1}, but taking into account the support of $\mu_{k,m}$. Indeed, assume $j>m$ and notice that
		\begin{align*}
			\big\rvert \pazocal{P}^s_{\mu_{k,m}}\chi_{\mathbb{R}^{n+1}\setminus{Q}} &(\ox) - \pazocal{P}^s_{\mu_{k,m}}\chi_{\mathbb{R}^{n+1}\setminus{Q}} (\ox') \big\rvert\\
			&\hspace{-2cm}\lesssim \ell(Q)\sum_{r=m+1}^{j}\int_{Q_{r-1}(\ox)\setminus{Q_{r}(\ox)}} \frac{\dd \mu_{k,m}(\oy)}{|\ox-\oy|_{p_s}^{n+2}}+\ell(Q)^{2s}\sum_{r=m+1}^{j}\int_{Q_{r-1}(\ox)\setminus{Q_{r}(\ox)}} \frac{\dd \mu_{k,m}(\oy)}{|\ox-\oy|_{p_s}^{n+2s+1}}\\
			&\hspace{-2cm} = \frac{\ell(Q)}{\mu_k(Q^m)}\sum_{r=m+1}^{j}\int_{Q_{r-1}(\ox)\setminus{Q_{r}(\ox)}} \frac{\dd \mu_{k}(\oy)}{|\ox-\oy|_{p_s}^{n+2}}+\frac{\ell(Q)^{2s}}{\mu_k(Q^m)}\sum_{r=m+1}^{j}\int_{Q_{r-1}(\ox)\setminus{Q_{r}(\ox)}} \frac{\dd \mu_{k}(\oy)}{|\ox-\oy|_{p_s}^{n+2s+1}}\\
			&\hspace{-2cm} \lesssim \frac{1}{\mu(Q^m)}\frac{\ell(Q)}{\ell(\widehat{Q})}p_m(\widehat{Q})\lesssim  \frac{1}{\mu_k(Q^m)} p_m(Q).
		\end{align*}
	\end{proof}
	\begin{lem}
		\label{lem4.1.6}
		If $Q\in \pazocal{Q}^j\cap Q^m$ with $m\leq j<k$, then
		\begin{equation*}
			\bigg\rvert  S_Q^m(\pazocal{P}^s \mu_{k,m}) - \sum_{P\in \pazocal{CH}(Q)} S_P^m(\pazocal{P}^s \mu_{k,m}) \bigg\rvert\lesssim_{\tau_0} \frac{1}{\mu_k(Q^m)}p_m(Q).
		\end{equation*}
	\end{lem}
	As a direct consequence of Lemma \ref{lem4.1.6} we also have
	\begin{equation}
		\label{eq4.1.4}
		\big\| D_Q^m(\pazocal{P}^s \mu_{k,m}) \big\|_{L^2(\mu_{k,m})}^2 \lesssim_{\tau_0} \frac{1}{\mu_k(Q^m)^2} p_m(Q)^2\mu_{k,m}(Q),
	\end{equation}
	analogous to the estimate of Remark \ref{rem4.1.5}. Combining all of the above results and observations, we finally deduce the result we were interested in
	\begin{lem}
		\label{lem4.1.7}
		The following estimate holds for any $0<m<k$:
		\begin{equation*}
			\| \pazocal{P}^s \mu_{k,m} \|_{L^2(\mu_{k,m})}^2\lesssim_{\tau_0} \frac{1}{\mu_k(Q^m)^2}\sum_{j=0}^k \theta_{j,p_s}^2.
		\end{equation*}
	\end{lem}
	\begin{proof}
		By relations \eqref{eq4.1.3} and \eqref{eq4.1.4} we now have
		\begin{align*}
			\big\| S_k^m (\pazocal{P}^s \mu_{k,m}) \big\|_{L^2(\mu_{k,m})}^2 &= \sum_{Q\in \pazocal{Q}(m)\setminus{(\pazocal{Q}^k\cap Q^m})} \big\| D_Q^m (\pazocal{P}^s \mu_{k,m}) \big\|_{L^2(\mu_{k,m})}^2 \\
			&\hspace{-1.5cm}\lesssim \frac{1}{\mu_k(Q^m)^2}\sum_{Q\in \pazocal{Q}(m)\setminus{(\pazocal{Q}^k\cap Q^m})} p_m(Q)^2\mu_{k,m}(Q)\\
			&\hspace{-1.5cm}=\frac{1}{\mu_k(Q^m)^2} \sum_{j=m}^{k-1}\sum_{Q^j\in \pazocal{Q}^j\cap Q^m} p_m(Q^j)^2\mu_{k,m}(Q^j) = \frac{1}{\mu_k(Q^m)^2}\sum_{j=m}^{k-1}p_m(Q^j)^2.
		\end{align*}
		Moreover, \eqref{eq4.1.2} and Lemma \ref{lem4.1.5} imply that for each $Q\in \pazocal{Q}^k\cap Q^m$ and $\ox\in Q$,
		\begin{align*}
			\big\rvert S_Q^m(\pazocal{P}^s \mu_{k,m})(\ox) &- \chi_Q\pazocal{P}^s_{\mu_{k,m}}\chi_{\mathbb{R}^{n+1}\setminus{Q}}(\ox) \big\rvert \lesssim \frac{1}{\mu_k(Q^m)} p_m(Q).
		\end{align*}
		It is also clear that for $Q\in \pazocal{Q}^k\cap Q^m$ and $\ox\in Q$, we have $|\pazocal{P}^s_{\mu_{k,m}}\chi_Q (\ox)|\lesssim \theta_{k,p_s}/\mu_k(Q^m)$. All in all, we finally conclude:
		\begin{align*}
			\|\pazocal{P}^s_{\mu_{k,m}}&1\|_{L^2(\mu_{k,m})}^2\\
			&\hspace{-0.5cm}\lesssim \sum_{Q\in \pazocal{Q}^k\cap Q^m}\Big( \|\chi_Q\pazocal{P}^s_{\mu_{k,m}}\chi_Q \|_{L^2(\mu_{k,m})}^2 + \|\chi_Q\pazocal{P}^s_{\mu_{k,m}}\chi_{\mathbb{R}^{n+1}\setminus{Q}} -S_Q^m(\pazocal{P}^s \mu_{k,m}) \|_{L^2(\mu_{k,m})}^2 \\
			&\hspace{8.75cm}+ \|S_Q^m(\pazocal{P}^s \mu_{k,m}) \|_{L^2(\mu_{k,m})}^2 \Big)\\
			&\hspace{-0.5cm}\lesssim \sum_{Q\in \pazocal{Q}^k\cap Q^m} \frac{\theta_{k,p_s}^2}{\mu_k(Q^m)^3} \mu_k(Q)+\sum_{Q\in \pazocal{Q}^k\cap Q^m} \frac{p_m(Q)^2}{\mu_k(Q^m)^3} \mu_k(Q) + \|S_k^m(\pazocal{P}^s \mu_{k,m}) \|_{L^2(\mu_{k,m})}^2\\
			&\hspace{-0.5cm}\lesssim \frac{1}{\mu_k(Q^m)^2}\Bigg( \theta_{k,p_s}^2+p_m(Q^k)^2+\sum_{j=m}^{k-1}p_m(Q^j)^2 \Bigg) \lesssim \frac{1}{\mu_k(Q^m)^2} \sum_{j=m}^kp_m(Q^j)^2\\
			&\hspace{-0.5cm}\leq \frac{1}{\mu_k(Q^m)^2} \sum_{j=0}^k p(Q^j)^2 \lesssim \frac{1}{\mu_k(Q^m)^2} \sum_{j=0}^{k}\theta_{j,p_s}^2,
		\end{align*}
		where for the last inequality we have used Lemma \ref{lem4.1.3}.
	\end{proof}
	Observe that we can rewrite the previous $L^2(\mu_{k,m})$ norm as
	\begin{equation*}
		\|\pazocal{P}^s \mu_{k,m}\|_{L^2(\mu_{k,m})} = \frac{1}{\mu_k(Q^m)^{3/2}} \|\pazocal{P}^s_{\mu_{k}}\chi_{Q^m}\|_{L^2(\mu_{k}|_{Q^m})},
	\end{equation*}
	and the previous result can be restated as
	\begin{equation}
		\label{eq4.1.5}
		\|\pazocal{P}^s_{\mu_{k}}\chi_{Q^m}\|_{L^2(\mu_{k}|_{Q^m})} \lesssim_{\tau_0} \Bigg( \sum_{j=0}^k \theta_{j,p_s}^2 \Bigg)^{1/2}\,\mu_k(Q^m)^{1/2}.
	\end{equation}
	In fact, bearing in mind Lemma \ref{lem4.1.4}, \eqref{eq4.1.5} is also valid for $0\leq m < k$. For the case $m=k$ simply notice that for any $Q^k\in \pazocal{Q}^k$ and $\ox\in Q$, polar integration yields
	\begin{equation*}
		|\pazocal{P}^s_{\mu_{k}}\chi_{Q^k}(\ox)|\lesssim \frac{1}{|E_{k,p_s}|}\int_{Q^k}\frac{\dd \oy}{|\ox-\oy|_{p_s}^{n+1}}\lesssim \frac{\ell_k}{|E_{k,p_s}|}=\theta_{k,p_s},
	\end{equation*}
	so \eqref{eq4.1.5} also holds in this case. Again, we need $s>1/2$ in the above estimate.
	
	Finally, as it is remarked in \cite[\textsection 3]{MT} and \cite[\textsection 3]{T1}, since the support of $\mu_k$ is $\pazocal{Q}^k$, relation \eqref{eq4.1.5} suffices to deduce the same result not only for $Q^m$, but also for any $s$-parabolic cube $Q\subset \mathbb{R}^{n+1}$. 
	Moreover, by the arguments used to prove \eqref{eq4.1.5}, it is clear that such estimate is also valid for the operator $\pazocal{P}^{s,\ast}_{\mu_k}$, associated with the kernel $(\nabla_x P_s)^\ast(\ox):= \nabla_x P_s(-\ox)$. With this, we are finally ready to state the main theorem of this subsection:
	\begin{thm}
		\label{thm4.1.8}
		Let $Q\subset \mathbb{R}^{n+1}$ be any $s$-parabolic cube. Then,
		\begin{equation*}
			\|\pazocal{P}^s_{\mu_{k}}\chi_{Q}\|_{L^2(\mu_{k}|_{Q})}+\|\pazocal{P}^{s,*}_{\mu_{k}}\chi_{Q}\|_{L^2(\mu_{k}|_{Q})} \lesssim_{\tau_0} \Bigg( \sum_{j=0}^k \theta_{j,p_s}^2 \Bigg)^{1/2}\,\mu_k(Q)^{1/2}.
		\end{equation*}
	\end{thm}

	\section{The lower bound for the capacity}
    \label{sec4}
	
	Bearing in mind \cite[Theorem 3.21]{T3}, it may seem that from Theorem \ref{thm4.1.8} we could directly obtain the desired estimate for the $\Gamma_{\Theta^s}$ capacity of the $s$-parabolic Cantor set. However, such result does not apply to our case, since our ambient space $\mathbb{R}^{n+1}$ is not endowed with the usual Euclidean distance. Nevertheless, there are $T1$-like theorems available in more general settings. More precisely, we may apply \cite[Theorem 2.3]{HyMa}, since $\mathbb{R}^{n+1}$ is geometrically doubling once endowed with the $s$-parabolic distance. In essence, the latter result adapted to our context implies that, for a fixed generation $k$, to control the boundedness of $\pazocal{P}^s_{\mu_k}$ as an $L^2(\mu_k)$-operator it suffices to verify that 
	\begin{enumerate}[nolistsep,leftmargin=*]
		\item[\textit{i}.] $\pazocal{P}^s \mu_k$ and $\pazocal{P}^{s,*} \mu_k$ belong to a certain $s$-parabolic BMO space (that is precised below),
		\item[\textit{ii}.] $\pazocal{P}^s \mu_k$ is $\mu_k$-weakly bounded.
	\end{enumerate}
	In fact, the following observations will also be important to simplify our proof:
	\begin{enumerate}[leftmargin=*]
        \item[\textit{1}.] By Lemma \ref{lem1.2}, the weak boundedness property follows trivially, since any pairing of the form $\langle \chi_R, \pazocal{P}^s_{\mu_k}\chi_R \rangle$ is null, for any $R\subset \mathbb{R}^{n+1}\,$ $s$-parabolic cube.
		\item[\textit{2}.] By the second point of \cite[Remark 2.4]{HyMa}, since the $s$-parabolic distance is a proper distance (and not a \textit{quasi-distance}, as in the statement of \cite[Theorem 2.3]{HyMa}), it suffices to show that $\pazocal{P}^s \mu_k$ and $\pazocal{P}^{s,*} \mu_k$ belong to some $s$-parabolic $\text{BMO}_{\rho,p_s}(\mu_k)$ space, for some $\rho>1$. Recall:
		\begin{defn}
			\label{def4.1.1}
			Given $\rho>1$ and $f\in L^1_{\text{loc}}(\mu_k)$, we say that $f$ belongs to the $\textit{BMO}_{\rho,p_s}(\mu_k)$ space if for some constant $c>0$,
			\begin{equation*}
				\sup_{Q} \frac{1}{\mu_k(\rho Q)} \int_Q \big\rvert f(\ox) - f_{Q,\mu_k}\big\rvert \text{d}\mu_k(\ox) \leq c,
			\end{equation*}
			where the supremum is taken among all $s$-parabolic cubes such that $\mu_k(Q)\neq 0$, and $f_{Q,\mu_k}$ is the average of $f$ in $Q$ with respect to $\mu_k$. The infimum over all values $c$ satisfying the above inequality is the so-called $\textit{BMO}_{\rho,p_s}(\mu_k)$ \textit{norm of} $f$.
		\end{defn}
		\item[\textit{3}.] 
		As it is verified in \cite[Theorem 2.2]{HMPr}, the operator defined through the kernel $\nabla_x P_s$ defines a $(n+1)$-dimensional CZ convolution operator in the $s$-parabolic space $\mathbb{R}^{n+1}$. With this we mean that it satisfies the required bounds of an $(n+1)$-dimensional CZ convolution kernel but changing the usual distance $|\cdot|$ by $|\cdot|_{p_s}$.
		\item[\textit{4}.] In light of the previous observation, we should impose that for each generation $k$, the measure $\mu_k$ presents upper $s$-parabolic growth of degree $n+1$. To satisfy such property, we will assume that there exists an absolute constant $\kappa>0$ so that $\theta_{j,p_s} \leq \kappa$ for every $j\geq 0$. Recall that such condition implied the desired growth restriction for $\mu_k$ with a constant $C$ depending only on $n,s$ and $\kappa$. Renormalizing $\mu$ with such constant, we shall assume $C=1$. With this, and borrowing the notation of \cite[\textsection 2.2]{HyMa}, $\mu_k$ is \textit{upper doubling} with dominating function $r^{n+1}$.
		\item[\textit{5.}] Recall that given $A>0$ and $\mu$ Borel measure on $\mathbb{R}^{n+1}$, we say that an $s$-parabolic cube $Q\subset \mathbb{R}^{n+1}$ has $A$-\textit{small boundary} $($\textit{with respect to $\mu$}$)$ if
		\begin{equation*}
			\mu \big( \{\ox\in 2Q\,:\; \text{dist}_{p_s}(\ox, \partial Q) \leq \alpha\,\ell(Q) \} \big) \leq A\alpha\, \mu(2Q), \hspace{0.75cm} \forall \alpha >0.
		\end{equation*}
	\end{enumerate}
	
	Previous to the main lemma, we prove two additional preliminary results, the first one being an $s$-parabolic version of \cite[Lemma 9.43]{T3} and the second one deals with the existence of \textit{large} doubling balls. It can be understood as a direct consequence of \cite[Lemma 3.2]{Hy}. Recall that for a given real Borel measure $\mu$ in $\mathbb{R}^{n+1}$ and $\alpha,\beta>1$, a $s$-parabolic cube $Q\subset \mathbb{R}^{n+1}$ is said to be $(\alpha,\beta)$-\textit{doubling} (\textit{with respect to $\mu$}) if $\mu(\alpha Q)\leq \beta \mu(Q)$. . Let us remark that in some of the forthcoming statements, the reader will encounter expressions of the form $\alpha Q$, for some $\alpha>0$ and $Q$ an $s$-parabolic cube. This has to be understood as an $s$-\textit{parabolic dilation}: i.e. if $Q=Q_1\times I_Q\subset \mathbb{R}^{n}\times \mathbb{R}$, then $\alpha Q = (\alpha Q_1)\times (\alpha^{2s} I_Q)$.
	
	\begin{lem}
		\label{lem4.1.9}
		Let $\mu$ be a real finite Borel measure on $\mathbb{R}^{n+1}$ and $A(n,s)>0$ some big enough constant. Let $Q\subset \mathbb{R}^{n+1}$ be any fixed $s$-parabolic cube. Then, there exists a concentric $s$-parabolic cube $Q'$ with $Q\subset Q' \subset 1.1Q$ with $A$-small boundary with respect to $\mu$.
	\end{lem}
	\begin{proof}
		We shall follow the proof of \cite[Lemma 9.43]{T3} and adapt it to our $s$-parabolic setting. Assume that $Q$ is centered at the origin and write $\sigma:=\mu|_{2Q}$. For $a\in \mathbb{R}$ and $1\leq j \leq n+1$, let $H_j(a)$ be the hyperplane
		\begin{equation*}
			H_j(a):=\big\{ \ox\in \mathbb{R}^{n+1}\,:\; x_j=a \big\},
		\end{equation*}
		where we convey $x_{n+1}:=t$. For $\delta>0$, write $U_\delta$ the (Euclidean) $\delta$-neighborhood of a set. The existence of $Q'$ will follow from the existence of some $a\in[\ell(Q), 1.05\ell(Q)]$ such that
		\begin{align}
			\label{eq4.1.6}
			\frac{1}{\eta\ell(Q)}\sigma\Big( U_{\eta\ell(Q)}\big(H_j(\pm a)\big)\Big)&\leq A \frac{\|\sigma\|}{\ell(Q)}, \hspace{0.75cm} \forall \eta >0, \; j=1, \ldots, n,\\
			\frac{1}{\eta\ell(Q)}\sigma\Big( U_{\eta^{2s}\ell(Q)^{2s}}\big(H_{n+1}(\pm a)\big)\Big)&\leq A \frac{\|\sigma\|}{\ell(Q)}, \hspace{0.75cm} \forall \eta >0.
			\label{eq4.1.7}
		\end{align}
		Recall that $\|\sigma\|:=|\sigma|(\mathbb{R}^{n+1})$, where $|\sigma|$ is the variation of $\sigma$. Let $\pi_j,\widetilde{\pi}_j:\mathbb{R}^{n+1}\to\mathbb{R}$ be the projections defined by $\pi_j(\ox):=x_j,\, \widetilde{\pi}_j(\ox):=-x_j$, for $j=1,\ldots n$; as well as $\pi_{n+1}':\mathbb{R}^n\times[0,\infty)\to \mathbb{R}$ given by $\pi_{n+1}'(\ox):=x_{n+1}^{\frac{1}{2s}}$, and $\widetilde{\pi}_{n+1}':\mathbb{R}^{n}\times(-\infty,0]\to\mathbb{R}$ given by $\widetilde{\pi}_{n+1}'(\ox):=(-x_{n+1})^{\frac{1}{2s}}$. Consider the image measures
		\begin{align*}
			\nu_j&:=\pi_j\# \sigma, \hspace{1.55cm} \widetilde{\nu}_j:=\widetilde{\pi}_j\# \sigma, \hspace{0.75cm} j=1, \ldots, n,\\
			\nu_{n+1}&:=\pi_{n+1}'\# \sigma, \hspace{0.7cm} \widetilde{\nu}_{n+1}:= \widetilde{\pi}_{n+1}'\# \sigma,
		\end{align*}
		where $f \# \mu (\cdot) := \mu(f^{-1}(\cdot))$. This way, conditions \eqref{eq4.1.6} and \eqref{eq4.1.7} can be simply rewritten as
		\begin{align*}
			\frac{1}{\eta\ell(Q)}\nu_j\Big( I\big( a,\eta\ell(Q) \big) \Big)\leq A\frac{\|\sigma\|}{\ell(Q)}, \hspace{0.75cm} \frac{1}{\eta\ell(Q)}\widetilde{\nu}_j\Big( I\big( a,\eta\ell(Q) \big) \Big)\leq A\frac{\|\sigma\|}{\ell(Q)}, \hspace{0.75cm} \forall\eta>0,
		\end{align*}
		where now $1\leq j\leq n+1$ and $I(y,\ell)$ denotes the real interval centered at $y$ with length $2\ell$. In fact, the above condition can be rephrased as
		\begin{equation}
			\label{eq4.1.8}
			M\nu_j(a)\leq A\frac{\|\sigma\|}{\ell(Q)}, \hspace{0.75cm} M\widetilde{\nu}_j(a)\leq A\frac{\|\sigma\|}{\ell(Q)}, \hspace{0.75cm} j=1, \ldots, n+1,
		\end{equation}
		where $M\equiv M_{\pazocal{L}^1}$ is maximal Hardy-Littlewood operator in $\mathbb{R}$. We now define the measure $\nu:=\sum_{j=1}^{n+1}\nu_j+\widetilde{\nu}_j$. Observe that $\|\nu_j\|=\|\widetilde{\nu}_j\|=\|\sigma\|$ for $j=1,\ldots, n$, and $\|\nu_{n+1}\|+\|\widetilde{\nu}_{n+1}\|=\|\sigma\|$. Therefore $\|\nu\|=(2n+1)\|\sigma\|$. Notice that if we prove
		\begin{equation*}
			M\nu(a)\leq A \frac{\|\sigma\|}{\ell(Q)}=A\frac{\|\nu\|}{(2n+1)\ell(Q)},
		\end{equation*}
		condition \eqref{eq4.1.8} will hold. But due to \cite[Theorem 2.5]{T3} (a standard result concerning the weak boundedness of $M$ in a general non-doubling setting),
		\begin{equation*}
			\pazocal{L}^1\Bigg( \bigg\{ a\in \mathbb{R}\,:\; M\nu(a)>A\frac{\|\nu\|}{(2n+1)\ell(Q)} \bigg\} \Bigg) \leq C \frac{(2n+1)\ell(Q)}{A}.
		\end{equation*}
		So for $A$ big enough there is $a\in[\ell(Q), 1.05\ell(Q)]$ with $M\nu(a)\leq A\frac{\|\nu\|}{(2n+1)\ell(Q)}$.
	\end{proof}
	
	\begin{lem}
		\label{lem4.1.10}
		Let $Q\subset \mathbb{R}^{n+1}$ be an $s$-parabolic cube and $\mu$ a real Borel measure on $\mathbb{R}^{n+1}$ that has upper $s$-parabolic growth of degree $n+1$ with constant $1$. Then, there exists $j_0\in \mathbb{N}$ such that $Q_0:=3^{j_0}Q$ is $(3,3^{n+2})$-doubling.
	\end{lem}
	\begin{proof}
		Apply \cite[Lemma 3.2]{Hy} with $C_\lambda:=2^{n+1}$, $\alpha=3$ and $\beta=3^{n+2}$.
	\end{proof}
	We are now ready to prove the result we were initially interested in:
	\begin{lem}
		\label{lem4.1.11}
		Let $Q\subset \mathbb{R}^{n+1}$ be any $s$-parabolic cube and $\mu$ a compactly supported positive Borel measure. Assume that $\mu$ has upper $s$-parabolic growth of degree $n+1$ with constant 1 and that $|\langle \chi_R, \Ps_\mu \chi_R \rangle| \lesssim 1$ for any $R\subset \mathbb{R}^{n+1}$\, $s$-parabolic cube with $A$-small boundary, $A=A(n,s)$. Then,
		\begin{equation*}
			\|\pazocal{P}^s \mu\|_{\text{\normalfont{BMO}}_{3,p_s}(\mu)} + \|\pazocal{P}^{s,*} \mu\|_{\text{\normalfont{BMO}}_{3,p_s}(\mu)} \lesssim 1+ \frac{\|\pazocal{P}^s_{\mu}\chi_{Q}\|_{L^2(\mu|_{Q})}}{\mu(Q)^{1/2}}
		\end{equation*}
	\end{lem}
	
	\begin{proof}
		We give the details to estimate $ \|\pazocal{P}^s \mu\|_{\text{BMO}_{3,p_s}(\mu)}$, since the arguments can be directly adapted for $\|\pazocal{P}^{s,*} \mu\|_{\text{BMO}_{3,p_s}(\mu)}$. We clarify that the arguments below are inspired by those given for \cite[Proposition 9.45]{T3}. 
        
        Let $A=A(n,s)>0$ be big enough (as in Lemma \ref{lem4.1.9}) and consider an $s$-parabolic cube $Q$ with $A$-small boundary. By Lemma \ref{lem4.1.10} let $Q_0:=3^{j_0}Q$ be a $(3,3^{n+2})$-doubling $s$-parabolic cube (with respect to $\mu$) with the minimal $j_0\in \mathbb{N}$ such that this property is satisfied. That is, we require that $\mu(3^{j}Q)>3^{n+2}\mu(3^{j-1}Q)$, for $j=1,\ldots, j_0-1$. Iterating the previous inequality we also deduce
		\begin{equation}
			\label{eq4.1.9}
			\mu(3^jQ)\leq \frac{\mu(3^{j_0-1}Q)}{3^{(n+2)(j_0-1-j)}}, \hspace{0.5cm} \text{for}\; j=1,\ldots, j_0-1.
		\end{equation}
		By Lemma \ref{lem4.1.9} we can take $\widehat{Q}$ with $A$-small boundary concentric with $Q_0$ such that $Q_0\subset \widehat{Q} \subset 1.1Q_0$. Since $2\widehat{Q}\subset 3Q_0$, it is clear that $\widehat{Q}$ is $(2,3^{n+2})$-doubling. Assume that for any $s$-parabolic cube $Q$ with $A$-small boundary we prove the estimate
		\begin{equation}
			\label{eq4.1.10}
			\Bigg\rvert \bigg( \int_Q\Big\rvert \pazocal{P}^s \mu-(\pazocal{P}^s \mu)_{\widehat{Q},\mu} \Big\rvert^2 \dd \mu \bigg)^{1/2}-\bigg( \int_Q \big\rvert \pazocal{P}^s_{\mu}\chi_Q \big\rvert^2 \dd \mu \bigg)^{1/2} \Bigg\rvert \leq C \mu(2Q)^{1/2},
		\end{equation}
		for some constant $C(n,s)$ say bigger than 1, where recall that $(\pazocal{P}^s \mu)_{\widehat{Q},\mu}$ is the average of $\pazocal{P}^s \mu$ in $\widehat{Q}$ with respect to $\mu$. Then, by Cauchy-Schwarz's inequality we infer that for any $s$-parabolic cube $Q$ with $A$-small boundary,
		\begin{equation*}
			\int_Q \Big\rvert \pazocal{P}^s \mu-(\pazocal{P}^s \mu)_{\widehat{Q},\mu} \Big\rvert \dd \mu\lesssim \bigg(C+ \frac{\|\pazocal{P}^s_{\mu}\chi_{Q}\|_{L^2(\mu|_{Q})}}{\mu(Q)^{1/2}}\bigg) \mu(2Q) \lesssim \bigg(1+ \frac{\|\pazocal{P}^s_{\mu}\chi_{Q}\|_{L^2(\mu|_{Q})}}{\mu(Q)^{1/2}}\bigg) \mu(2Q).
		\end{equation*}
		Now observe for an arbitrary $s$-parabolic cube $P$, we can take $Q$ with $A$-small boundary concentric with $P$ and such that $P\subset Q \subset 1.1P$. Hence,
		\begin{align*}
			\int_P \Big\rvert \pazocal{P}^s \mu -(\pazocal{P}^s \mu)_{P,\mu} \Big\rvert \dd \mu &\leq \int_P \Big\rvert \pazocal{P}^s \mu-(\pazocal{P}^s \mu)_{\widehat{Q},\mu} \Big\rvert \dd \mu +\Big\rvert(\pazocal{P}^s \mu)_{P,\mu} -(\pazocal{P}^s \mu)_{\widehat{Q},\mu} \Big\rvert\mu(P)\\
			&\leq \int_Q \Big\rvert \pazocal{P}^s \mu-(\pazocal{P}^s \mu)_{\widehat{Q},\mu} \Big\rvert \dd \mu +\Big( \Big\rvert \pazocal{P}^s \mu -(\pazocal{P}^s \mu)_{\widehat{Q},\mu} \Big\rvert \Big)_{P,\mu}\mu(P)\\
			&\leq 2\int_Q \Big\rvert \pazocal{P}^s \mu-(\pazocal{P}^s \mu)_{\widehat{Q},\mu} \Big\rvert \dd \mu\lesssim \bigg( 1+\frac{\|\pazocal{P}^s_{\mu}\chi_{Q}\|_{L^2(\mu|_{Q})}}{\mu(Q)^{1/2}}\bigg)\mu(3P).
		\end{align*}
		Therefore, it suffices to prove \eqref{eq4.1.10} in order to deduce the desired result. Begin by noticing that the triangle inequality applied to the left-hand side of \eqref{eq4.1.10} yields
		\begin{align*}
			\Bigg\rvert \bigg( &\int_Q\Big\rvert \pazocal{P}^s \mu-(\pazocal{P}^s \mu)_{\widehat{Q},\mu} \Big\rvert^2 \dd \mu \bigg)^{1/2}-\bigg( \int_Q \big\rvert \pazocal{P}^s_{\mu}\chi_Q \big\rvert^2 \dd \mu \bigg)^{1/2} \Bigg\rvert\\
			&\leq \Bigg( \int_Q \Big\rvert \pazocal{P}^s \mu-(\pazocal{P}^s \mu)_{\widehat{Q},\mu}-\pazocal{P}^s_{\mu}\chi_Q \Big\rvert^2 \dd \mu \Bigg)^{1/2} =\Bigg( \int_Q \Big\rvert \pazocal{P}^s_{\mu}\chi_{\mathbb{R}^{n+1}\setminus{Q}}-(\pazocal{P}^s \mu)_{\widehat{Q},\mu} \Big\rvert^2 \dd \mu \Bigg)^{1/2}.
		\end{align*}
		For each $\ox\in Q$ write the previous integrand as follows:
		\begin{align}
			\label{eq4.1.11}
			\pazocal{P}^s_{\mu}\chi_{\mathbb{R}^{n+1}\setminus{Q}}(\ox)-(\pazocal{P}^s \mu)_{\widehat{Q},\mu} &= \pazocal{P}^s_{\mu}\chi_{2Q\setminus{Q}}(\ox)+\pazocal{P}^s_{\mu}\chi_{2\widehat{Q}\setminus{2Q}}(\ox)\\
			&\hspace{1cm}-\Big( \pazocal{P}^s_{\mu}\chi_{\widehat{Q}} \Big)_{\widehat{Q},\mu}-\Big( \pazocal{P}^s_{\mu}\chi_{2\widehat{Q}\setminus{\widehat{Q}}} \Big)_{\widehat{Q},\mu} \nonumber \\
			&\hspace{1cm} +\bigg[ \pazocal{P}^s_{\mu}\chi_{\mathbb{R}^{n+1}\setminus{2\widehat{Q}}}(\ox)- \Big( \pazocal{P}^s_{\mu}\chi_{\mathbb{R}^{n+1}\setminus{2\widehat{Q}}} \Big)_{\widehat{Q},\mu} \bigg]. \nonumber
		\end{align}
		Let us begin by estimating the second term of the right-hand side. Since $2\widehat{Q}\subset 3^{j_0+1}Q$ and $\mu$ satisfies an upper $s$-parabolic growth condition, we have 
		\begin{align*}
			\Big\rvert \pazocal{P}^s_{\mu}\chi_{2\widehat{Q}\setminus{2Q}}(\ox) \Big\rvert &\lesssim \int_{3^{j_0+1}Q\setminus{Q}}\frac{\dd \mu(\oy)}{|\ox-\oy|_{p_s}^{n+1}}= \sum_{j=1}^{j_0+1}\int_{3^{j}Q\setminus{3^{j-1}Q}}\frac{\dd \mu(\oy)}{|\ox-\oy|_{p_s}^{n+1}}\\
			&\lesssim \sum_{j=1}^{j_0+1} \frac{\mu(3^{j}Q)}{\big( 3^{j}\ell(Q) \big)^{n+1}} \leq 2\frac{\mu\big( 3^{j_0+1}Q \big)}{\big( 3^{j_0}\ell(Q) \big)^{n+1}}+\sum_{j=1}^{j_0-1} \frac{\mu(3^{j}Q)}{\big( 3^{j}\ell(Q) \big)^{n+1}}\\
			&\lesssim 1+\sum_{j=1}^{j_0-1} \frac{\mu(3^{j}Q)}{\big( 3^{j}\ell(Q) \big)^{n+1}}.
		\end{align*}
		For the remaining sum, relation \eqref{eq4.1.9} implies
		\begin{align*}
			\sum_{j=1}^{j_0-1} \frac{\mu(3^{j}Q)}{\big( 3^{j}\ell(Q) \big)^{n+1}} &\leq \frac{\mu\big(3^{j_0-1}Q\big)}{\big(3^{j_0-1}\ell(Q)\big)^{n+1}}\sum_{j=1}^{j_0-1} \frac{1}{3^{(n+2)(j_0-1-j)}3^{(-j_0+1+j)(n+1)}}\\
			&\lesssim \sum_{j=1}^{j_0-1} \frac{1}{3^{j_0-1-j}}=\sum_{j=0}^{j_0-2}\frac{1}{3^j}\lesssim 1,
		\end{align*}
		so indeed $|\pazocal{P}^s_{\mu}\chi_{2\widehat{Q}\setminus{2Q}}(\ox)|\lesssim 1$ for $\ox\in Q$. The modulus of the third term of \eqref{eq4.1.11} is bounded by a constant depending on $n$ and $s$ by hypothesis, since $\langle \chi_R, \Ps_\mu \chi_R \rangle \lesssim 1$ for any $R\subset \mathbb{R}^{n+1}$\, $s$-parabolic cube with $A$-small boundary. Observe that the fourth term satisfies
		\begin{equation*}
			\Big( \pazocal{P}^s_{\mu}\chi_{2\widehat{Q}\setminus{\widehat{Q}}} \Big)_{\widehat{Q},\mu}\leq \frac{1}{\mu(\widehat{Q})} \int_{\widehat{Q}} \bigg(  \int_{2\widehat{Q}\setminus{\widehat{Q}}}\frac{\dd \mu(\oy)}{|\ox-\oy|_{p_s}^{n+1}} \bigg)\dd \mu(\ox).
		\end{equation*}
		
		Such expression can be dealt with as in \cite[Lemma 9.44]{T3}. Indeed, the above domains of integration imply $|\ox-\oy|_{p_s}\geq \text{dist}_{p_s}(\ox,\partial \widehat{Q})$. Then, defining
		\begin{equation*}
			Q_j:=Q\big(\ox,2^{j}\text{dist}_{p_s}(\ox,\partial\widehat{Q})\big), \hspace{0.75cm} 0\leq j \leq \Bigg\lceil \log_2\Bigg(\frac{4\ell(\widehat{Q})}{\text{dist}_{p_s}(\ox,\partial\widehat{Q})}\Bigg)\Bigg\rceil=:N,
		\end{equation*}
		integration over annuli and the upper $s$-parabolic growth of degree $n+1$ of $\mu$ yield
		\begin{align*}     \int_{2\widehat{Q}\setminus{\widehat{Q}}}\frac{\dd \mu(\oy)}{|\ox-\oy|_{p_s}^{n+1}} &\leq \int_{\text{dist}_{p_s}(\ox,\partial \widehat{Q})\,\leq\, |\ox-\oy|_{p_s}\, \leq  \,4\ell(\widehat{Q})}\frac{\dd \mu(\oy)}{|\ox-\oy|_{p_s}^{n+1}}\leq \sum_{j=0}^N\int_{Q_{j+1}\setminus{Q_j}} \frac{\dd \mu(\oy)}{|\ox-\oy|_{p_s}^{n+1}}\\
			&\leq \sum_{j=0}^N \frac{\mu(Q_{j+1})}{\ell(Q_j)^{n+1}}\lesssim  N = \Bigg\lceil \log_2\Bigg(\frac{4\ell(\widehat{Q})}{\text{dist}_{p_s}(\ox,\partial\widehat{Q})}\Bigg)\Bigg\rceil.
		\end{align*}
		For $j\geq 0$ let
		\begin{equation*}
			V_j:=\Big\{ \ox\in \widehat{Q}\,:\; 2^{-j-1}\ell(\widehat{Q}) < \text{dist}_{p_s}(\ox,\partial \widehat{Q}) \leq 2^{-j}\ell(\widehat{Q}) \Big\},
		\end{equation*}
		and observe that the $A$-small boundary property of $\widehat{Q}$ implies $\mu(V_j)\leq A2^{-j}\mu(2\widehat{Q})$. Then,
		\begin{align*}
			\frac{1}{\mu(\widehat{Q})} \int_{\widehat{Q}} \bigg(  \int_{2\widehat{Q}\setminus{\widehat{Q}}}\frac{\dd \mu(\oy)}{|\ox-\oy|_{p_s}^{n+1}} \bigg)\dd \mu(\ox) &\lesssim \frac{1}{\mu(\widehat{Q})} \sum_{j\geq 0}\int_{V_j}\Big\lceil \log_2\big(4\cdot 2^{j+1}\big)\Big\rceil \dd \mu(\ox)\\
			&\leq A\frac{\mu(2\widehat{Q})}{\mu(\widehat{Q})}\sum_{j\geq 0} \frac{\big\lceil \log_2\big(4\cdot 2^{j+1}\big)\big\rceil }{2^j}\lesssim 1,
		\end{align*}
		where in the last step we have used that $\widehat{Q}$ is $(2,3^{n+2})$-doubling. Finally, applying \cite[Lemma 9.12]{T3} (that admits a straightforward generalization to the $s$-parabolic setting) with $f:=\chi_{\mathbb{R}^{n+1}\setminus{2\widehat{Q}}}$, we also deduce
		\begin{equation*}
			\bigg\rvert \pazocal{P}^s_{\mu}\chi_{\mathbb{R}^{n+1}\setminus{2\widehat{Q}}}(\ox)- \Big( \pazocal{P}^s_{\mu}\chi_{\mathbb{R}^{n+1}\setminus{2\widehat{Q}}} \Big)_{\widehat{Q},\mu} \bigg\rvert \lesssim 1
		\end{equation*}
		Therefore, the left-hand side of \eqref{eq4.1.10} is bounded above by
		\begin{equation*}
			\bigg( \int_Q \Big\rvert \pazocal{P}^s_{\mu}\chi_{2Q\setminus{Q}} \Big\rvert^2 \dd \mu \bigg)^{1/2}+C\mu(Q)^{1/2},
		\end{equation*}
		for some $C(n,s)$. Notice that the remaining integral is such that
		\begin{equation*}
			\int_Q \Big\rvert \pazocal{P}^s_{\mu}\chi_{2Q\setminus{Q}} \Big\rvert^2 \dd \mu  \leq \int_{Q} \bigg(  \int_{2Q\setminus{Q}}\frac{\dd \mu(\oy)}{|\ox-\oy|_{p_s}^{n+1}} \bigg)^2\dd \mu(\ox).
		\end{equation*}
		As $Q$ has $A$-small boundary, we can proceed as we have done for the fourth term of the right-hand side of \eqref{eq4.1.11} (again, see \cite[Lemma 9.44]{T3} for more details), and deduce
		\begin{equation*}
			\int_Q \Big\rvert \pazocal{P}^s_{\mu}\chi_{2Q\setminus{Q}} \Big\rvert^2 \dd \mu  \lesssim \mu(2Q). 
		\end{equation*}
		All in all, we get that the left-hand side of \eqref{eq4.1.10} is bounded by $\mu(2Q)^{1/2}$, up to a multiplicative constant depending only on $n$ and $s$, that implies the desired result.
	\end{proof}

    Let $(\lambda_j)_j$ be such that $0<\lambda_j\leq \tau_0<1/d$, for every $j$, and denote by $E_{p_s}$ its associated $s$-parabolic Cantor set as in \eqref{eq4.1.1}. Assume that there exists an absolute constant $\kappa>0$ so that $\theta_{j,p_s} \leq \kappa$ for every $j\geq 0$. Fix a generation $k$ and let $\mu_k$ be the usual uniform probability measure of $E_{k,p_s}$. Then, by Theorem \ref{thm4.1.8}, Lemma \ref{lem1.2}, the fact that $\mu_k$ satisfies an upper $s$-parabolic growth condition and that $\theta_{0,p_s}=1$,
		\begin{equation*}
			\|\pazocal{P}^s \mu_k\|_{\text{\normalfont{BMO}}_{3,p_s}(\mu_k)} + \|\pazocal{P}^{s,*} \mu_k\|_{\text{\normalfont{BMO}}_{3,p_s}(\mu_k)} \lesssim_{\tau_0,\kappa} 1+ \Bigg( \sum_{j=0}^{k} \theta_{j,p_s}^2 \Bigg)^{1/2} \leq \Bigg( \sum_{j=0}^{k} \theta_{j,p_s}^2 \Bigg)^{1/2}.
		\end{equation*}
	
	Lemma \ref{lem4.1.11} allows us to deduce the desired estimate for $\Gamma_{\Theta}(E_{k,p_s})$:
	\begin{thm}
		\label{thm4.1.12}
		Let $(\lambda_j)_j$ be such that $0<\lambda_j\leq \tau_0<1/d$, for every $j$, and denote by $E_{p_s}$ its associated $s$-parabolic Cantor set as in \eqref{eq4.1.1}. Assume that there exists an absolute constant $\kappa>0$ so that $\theta_{j,p_s} \leq \kappa$ for each $j\geq 0$. Then, for every generation $k$,
		\begin{equation*}
			\widetilde{\Gamma}_{\Theta^s,+} (E_{k,p_s}) \gtrsim_{\tau_0,\kappa} \Bigg( \sum_{j=0}^{k} \theta_{j,p_s}^2 \Bigg)^{-1/2}.
		\end{equation*}
	\end{thm}
	\begin{proof}
		By a direct application of Lemma \ref{lem4.1.11} and \cite[Theorem 2.3]{HyMa} we deduce
		\begin{equation}
        \label{eq3.7}
			\big\|\pazocal{P}^s_{\mu_k} \big\|_{L^{2}(\mu_k)\to L^{2}(\mu_k)}\leq C\Bigg( \sum_{j=0}^{k} \theta_{j,p_s}^2 \Bigg)^{1/2},
		\end{equation}
		for some $C=C(n,s,\tau_0, \kappa)$. Then, by \cite[Theorem 4.3]{H}, $C^{-1}\big(\sum_{j=0}^{k} \theta_{j,p_s}^2\big)^{-1/2}\mu_k$ becomes an admissible measure for $\widetilde{\Gamma}_{\Theta^s,+}(E_{k,p_s})$, and we are done.
	\end{proof}
	
	The next lemma will allow us to extend the result to the final $s$-parabolic Cantor set $E_{p_s}$.
	
	\begin{lem}
		\label{lem4.1.13}
		If $(E_k)_k$ is a nested sequence of compact sets of $\mathbb{R}^{n+1}$ that decreases to $\pazocal{E}:=\cap_{k=1}^\infty E_k$, then
		\begin{equation*}
			\lim_{k\to\infty}\widetilde{\Gamma}_{\Theta^s,+}(E_k)=\widetilde{\Gamma}_{\Theta^s,+}(\pazocal{E}).
		\end{equation*}
	\end{lem}
	\begin{proof}
		It is clear that $\widetilde{\Gamma}_{\Theta^s,+}(\pazocal{E})\leq \lim_{k\to\infty}\widetilde{\Gamma}_{\Theta^s,+}(E_k)$, so we are left to prove the converse inequality. For each $k$ consider an admissible measure $\mu_k$ for $\widetilde{\Gamma}_{\Theta^s,+}(E_k)$ with
		\begin{equation*}
			\widetilde{\Gamma}_{\Theta^s,+}(E_k)-\frac{1}{k} \leq \mu_k(E_k) \leq \widetilde{\Gamma}_{\Theta^s,+}(E_k),
		\end{equation*}
		We shall verify that there exists an admissible measure $\mu$ for $\widetilde{\Gamma}_{\Theta^s,+}(\pazocal{E})$ so that 
		\begin{equation}
			\label{eq4.1.12}
			\limsup_{k\to\infty} \mu_k(E_k) \leq \mu(\pazocal{E}).
		\end{equation}
		If this is the case,
		\begin{equation*}
			\lim_{k\to\infty} \widetilde{\Gamma}_{\Theta^s,+}(E_k)\leq \limsup_{k\to\infty}  \mu_k(E_k) \leq \mu(\pazocal{E}) \leq \widetilde{\Gamma}_{\Theta^s,+}(\pazocal{E}),
		\end{equation*}
		and we are done. To construct such $\mu$, notice that \cite[Theorem 3.1]{H} implies that each $\mu_k$ has upper $s$-parabolic growth of degree $n+1$ with an absolute constant $C$. Then $\mu_k(\mathbb{R}^{n+1})\leq C\,\text{diam}_{p_s}(E_1)^{n+1},\, \forall k\geq 0$, so by \cite[Theorem 1.23]{Mat} there exists a positive Radon measure $\mu$ on $\mathbb{R}^{n+1}$ such that $\mu_k \rightharpoonup\mu$ weakly. Arguing by contradiction it is not difficult to verify that $\text{supp}(\mu)\subseteq \pazocal{E}$, and it is also clear that \eqref{eq4.1.12} is satisfied (in fact, taking $\varphi\in \pazocal{C}_0(\mathbb{R}^{n+1})$ with $\varphi \equiv 1$ on a neighborhood of $E_1$, \eqref{eq4.1.12} holds with a proper limit and an equality). So we are left to estimate the quantities $\|\nabla_x P_s \ast \mu\|_{\infty}$ and $\|\partial_t^{\frac{1}{2s}}P_s \ast\mu\|_{\ast,p_s}$, as well as the same norms changing $P$ by its conjugate $P^\ast$.
		
		By assumption $\nabla_xP_s\ast \mu_k$ belongs to the unit ball of $L^{\infty}(\mathbb{R}^{n+1})\cong L^1(\mathbb{R}^{n+1})^\ast$, and it is clear that $L^1(\mathbb{R}^{n+1})$ is separable. Then, by the sequential version of Banach-Alaoglu's theorem there exists some $S\in L^{\infty}(\mathbb{R}^{n+1})$ with $\|S\|_\infty\leq 1$ and $\nabla_xP_s\ast \mu_k\to S$ as $k\to \infty$ in a $\text{weak}^\star$-$L^\infty$ sense. Now take $\psi\in \pazocal{C}_c^\infty(B(0,1))$ positive and radial with $\int \psi = 1$ and set $\psi_\varepsilon:=\varepsilon^{-(n+2s)}\psi(\cdot/\varepsilon)$. Since $\nabla_x P_s\ast \mu_k$ converges to $S$ in a $\text{weak}^\star$-$L^\infty$ sense and by construction $\|\psi_\varepsilon\|_{L^1(\mathbb{R}^{n+1})}=1$,
		\begin{equation*}
			\lim_{k\to\infty} \big( \psi_\varepsilon \ast \nabla_xP_s \ast \mu_k \big)(\ox)= \psi_\varepsilon \ast S (\ox), \hspace{0.5cm} \ox\in \mathbb{R}^{n+1}.
		\end{equation*}
		In addition, since $\psi_\varepsilon\ast \nabla_x P_s \in \pazocal{C}^\infty(\mathbb{R}^{n+1})$ and $\mu_k$ converges to $\mu$ in the weak topology of (compactly supported) real Radon measures, we have
		\begin{equation*}
			\lim_{k\to\infty} \big( \psi_\varepsilon \ast  \nabla_xP_s \ast \mu_k \big)(\ox)= \big( \psi_\varepsilon \ast \nabla_xP_s \ast \mu \big) (\ox), \hspace{0.5cm} \ox\in \mathbb{R}^{n+1}.
		\end{equation*}
		Hence $\psi_\varepsilon\ast S = \psi_\varepsilon \ast \nabla_xP_s \ast \mu$ for every $\varepsilon>0$, so $S=\nabla_xP_s\ast \mu$ and in particular $\|\nabla_x P_s \ast \mu\|_{\infty}\leq 1$. Following exactly the same argument above, one can prove that there exists $S^\ast\in L^\infty(\mathbb{R}^{n+1})$ with $\|S^\ast\|_\infty \leq 1$ so that $S^\ast = \nabla_xP_s^\ast \ast \mu$. Finally, applying \cite[Lemma 4.2]{HMPr} with $\beta:=\frac{1}{2s}$ we also deduce $\|\partial_{t}^{\frac{1}{2s}}P_s\ast \mu\|_{\ast, p_s}\lesssim 1$ and $\|\partial_{t}^{\frac{1}{2s}}P_s^\ast \ast \mu\|_{\ast, p_s}\lesssim 1$, and the proof is complete.
	\end{proof}
	
	\begin{thm}
    \label{thm3.6}
		Let $(\lambda_j)_j$ be such that $0<\lambda_j\leq \tau_0<1/d$, for every $j$, and denote by $E_{p_s}$ its associated $s$-parabolic Cantor set as in \eqref{eq4.1.1}. Then, there exists a constant $C=C(n,s,\tau_0)$ such that
		\begin{equation*}
			\widetilde{\Gamma}_{\Theta^s,+}(E_{p_s})\geq C \Bigg( \sum_{j=0}^{\infty} \theta_{j,p_s}^2 \Bigg)^{-1/2}.
		\end{equation*}
	\end{thm}
	
	\begin{proof}
		We assume, without loss of generality, that the sum involved in the estimate is convergent. If this is the case, it is clear that there exists some $\kappa^2>0$ such that $\theta_{j,p_s}^2\leq \kappa^2$ for every $j$. Therefore, we are under the hypothesis of Theorem \ref{thm4.1.12} and we shall apply Lemma \ref{lem4.1.13} to deduce the desired result.
	\end{proof}
	\section{The lower \mathinhead{L^2}{}-estimate for \mathinhead{\Ps\mu_k}{} }
    \label{sec5}
	The goal of this section is to prove the following lower estimate:
	\begin{thm}
		\label{thm4.1}
		The following bound holds for each $s\in(1/2,1]$,
		\begin{equation*}
			\|\Ps \mu_k\|^2\gtrsim \sum_{j=0}^{k} \theta_{j,p_s}^2.
		\end{equation*}
	\end{thm}
	
	The proof will be analogous to that given in \cite[\textsection 5]{T2} for Riesz kernels. We will only carry out the steps where necessary modifications are needed. These stem as a result of the convolution kernel used in our context, which is not the Riesz kernel, as well as the different nature of the Cantor set we are considering. Theorem \ref{thm4.1} will follow from the next lemma:
	\begin{lem}
		\label{lem2.2}
		The following estimate holds,
		\begin{equation*}
			\sum_{Q\in \QQ^\ast} \|D_Q\Ps\mu_k\|^2\gtrsim \sum_{j=0}^{k-1}\theta_{j,p_s}^2.
		\end{equation*}
	\end{lem}
	Indeed, begin by observing that
	\begin{align*}
		\|S_k\Ps\mu_k\|^2&=\int \bigg\rvert \sum_{Q\in \QQ^k}S_Q\Ps\mu_k\bigg\rvert^2\dd\mu_k =\int \bigg\rvert \sum_{Q\in\QQ^k} \bigg( \frac{1}{\mu_k(Q)}\int_Q\Ps\mu_k\dd\mu_k\bigg)\chi_Q  \bigg\rvert^2\dd\mu_k\\
		&=\sum_{Q\in\QQ^k} \frac{1}{\mu_k(Q)}\bigg\rvert \int_Q\Ps\mu_k\dd\mu_k\bigg\rvert^2 \leq \sum_{Q\in\QQ^k}  \int_Q|\Ps\mu_k|^2\dd\mu_k=\|\Ps\mu_k\|^2,
	\end{align*}
	by the Cauchy-Schwarz inequality. Therefore, using \eqref{eq1.5} and Lemma \ref{lem2.2}, we get
	\begin{equation*}
		\|\Ps\mu_k\|^2 \geq \|S_k\Ps\mu_k\|^2=\sum_{Q\in\QQ^\ast}\|D_Q\Ps\mu_k\|^2\gtrsim \sum_{j=0}^{k-1}\theta_{j,p_s}^2.
	\end{equation*}
	Hence, we only have to prove $\|\Ps\mu_k\|^2\gtrsim \theta_{k,p_s}^2$. Consider $Q\in \pazocal{Q}^k$ and compute
	\begin{align}
		\|\Ps\mu_k\|&=\bigg\| \sum_{Q\in\QQ^k} \chi_Q\Ps\mu_k \bigg\| \nonumber \\
		&=\bigg\| \sum_{Q\in\QQ^k} \chi_Q\Ps_{\mu_k}\chi_Q + \sum_{Q\in\QQ^k} \chi_Q\Ps{\mu_k}\chi_{\mathbb{R}^{n+1}\setminus{Q}} + S_k\Ps\mu_k-S_k\Ps\mu_k \bigg\| \nonumber\\
		&\geq \bigg\| \sum_{Q\in\QQ^k} \chi_Q\Ps_{\mu_k}\chi_Q \bigg\| - \|S_k\Ps\mu_k\|- \bigg\| \sum_{Q\in\QQ^k} \chi_Q\Ps{\mu_k}\chi_{\mathbb{R}^{n+1}\setminus{Q}} - S_k\Ps\mu_k \bigg\| \label{eq2.1}
	\end{align}
	
	We deal with the first term. Name $Q_{\text{up-ri}}\subset Q$ the $s$-parabolic cube sharing the upper right-most vertex with $Q$ and with side length $\ell(Q)/4$. We name $Q_{\text{lo-le}}\subset Q$ the analogous $s$-parabolic cube that shares the lower left-most vertex with $Q$. We pick $\ox=(x,t)\in Q_{\text{up-ri}}$ and use \cite[Theorem 2.2]{HMPr} to compute:
	\begin{align*}
		|\Ps_{\mu_k}\chi_Q(\ox)| = \frac{1}{|E_{k,p_s}|}\bigg\rvert \int_Q \nabla_xP_s(\ox-\oy)\dd\oy \bigg\rvert \approx \frac{1}{|E_{k,p_s}|}\bigg\rvert \int_Q \frac{(x-y)(t-u)}{|\ox-\oy|_{p_s}^{n+2s+2}}\chi_{t-u>0} \dd y \dd u \bigg\rvert 
	\end{align*}
	
	To estimate the previous modulus, we fix the first component of $x-y$ and denote it by $x_1-y_1$ and study the integral
	\begin{equation}
		\label{eq2.2}
		\bigg\rvert \int_Q \frac{(x_1-y_1)(t-u)}{|\ox-\oy|_{p_s}^{n+2s+2}}\chi_{t-u>0} \dd y \dd u \bigg\rvert 
	\end{equation}
	We name $R_{\rightarrow}:=\{(y,u)\in Q\,:\, y_1>x_1\}$ and $R_{\leftarrow}:=\{(y,u)\in Q\,:\, 2x_1-\ell(Q)<y_1<x_1\}$. We depict such regions in Figure \ref{fig1}. By the spatial anti-symmetry of the integration kernel one gets
	\begin{equation*}
		\int_{R_{\rightarrow}} \frac{(x_1-y_1)(t-u)}{|\ox-\oy|_{p_s}^{n+2s+2}}\chi_{t-u>0} \dd y \dd u + \int_{R_{\leftarrow}} \frac{(x_1-y_1)(t-u)}{|\ox-\oy|_{p_s}^{n+2s+2}}\chi_{t-u>0} \dd y \dd u = 0.
	\end{equation*}
    \begin{figure}[t]
		\centering
		\includegraphics[width=0.5\textwidth]{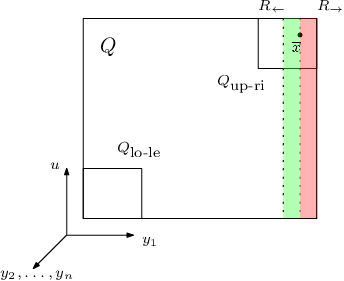}
		\caption{Here $Q$ is an $s$-parabolic cube of the $k$-th generation, $Q_{\text{up-ri}}$ is the $s$-parabolic cube of side length $\ell(Q)/4$ contained in $Q$ and sharing its upper right-most corner. $Q_{\text{lo-le}}$ is the analogous $s$-parabolic cube sharing the lower left-most corner. In red we depict the region $R_{\rightarrow}$ and in green the region $R_{\leftarrow}$.}
		\label{fig1}
	\end{figure}
    
	Then, since $x_1-y_1$ is nonnegative if $y_1\in Q\setminus(R_\rightarrow\cup R_\leftarrow)$, returning to \eqref{eq2.2} we obtain
	\begin{align*}
		\bigg\rvert \int_Q \frac{(x_1-y_1)(t-u)}{|\ox-\oy|_{p_s}^{n+2s+2}}\chi_{t-u>0} \dd y \dd u \bigg\rvert &= \int_{Q\setminus{(R_\rightarrow\cup R_\leftarrow)}} \frac{(x_1-y_1)(t-u)}{|\ox-\oy|_{p_s}^{n+2s+2}}\chi_{t-u>0} \dd y \dd u \\
		&\geq \int_{Q_{\text{lo-le}}} \frac{(x_1-y_1)(t-u)}{|\ox-\oy|_{p_s}^{n+2s+2}}\chi_{t-u>0} \dd y \dd u \\
		&\gtrsim \frac{\ell_k^{1+2s}}{\ell_k^{n+2s+2}}|Q_{\text{lo-le}}|\simeq \theta_{k,p_s}|E_{k,p_s}|.
	\end{align*}
	Hence we deduce $|\Ps_{\mu_k}\chi_Q(\ox)|\gtrsim \theta_{k,p_s}$ for all $\ox \in Q_{\text{up-ri}}$.	Then,
	\begin{align*}
		\bigg\| \sum_{Q\in\QQ^k} \chi_Q\Ps_{\mu_k}\chi_Q \bigg\|^2&= \sum_{Q\in \QQ^k}\int_Q|\Ps_{\mu_k}\chi_Q|^2\dd\mu_k\geq \sum_{Q\in \QQ^k}\int_{Q_{\text{up-ri}}}|\Ps_{\mu_k}\chi_Q|^2\dd\mu_k\gtrsim \theta_{k,p_s}^2.
	\end{align*}
	To deal with the second term in \eqref{eq2.1} we simply use that $\|S_k\Ps\mu_k\|\leq \|\Ps\mu_k\|$. On the other hand, by Lemma \ref{lem4.1.1} and \eqref{eq2.1}, if $\ox\in Q\in \QQ^k$,
	
	\begin{align*}
		|\Ps_{\mu_k}\chi_{\mathbb{R}^{n+1}\setminus{Q}}(\ox)&-S_k\Ps\mu_k(\ox)|=\bigg\rvert \Ps_{\mu_k}\chi_{\mathbb{R}^{n+1}\setminus{Q}}(\ox)-\sum_{Q\in \QQ^k} S_Q\Ps\mu_k(\ox) \bigg\rvert\\
		&=\bigg\rvert \Ps_{\mu_k}\chi_{\mathbb{R}^{n+1}\setminus{Q}}(\ox)- S_Q\Ps_{\mu_k}\chi_{\mathbb{R}^{n+1}\setminus{Q}}(\ox) \bigg\rvert\\
		&=\bigg\rvert \frac{1}{\mu_k(Q)}\int_Q\Big( \Ps_{\mu_k}\chi_{\mathbb{R}^{n+1}\setminus{Q}}(\ox)- \Ps_{\mu_k}\chi_{\mathbb{R}^{n+1}\setminus{Q}}(\ox')\Big)\dd\ox' \bigg\rvert\\
		&\lesssim \sum_{j=0}^{k-1}\theta_{j,p_s}\frac{\ell_{k-1}}{\ell_j} \leq \bigg( \sum_{j=0}^{k-1} \theta_{j,p_s}^2 \bigg)^{1/2}\bigg( \sum_{j=0}^{k-1} \frac{1}{d^{2j}}\bigg)^{1/2}\lesssim \bigg( \sum_{j=0}^{k-1} \theta_{j,p_s}^2 \bigg)^{1/2}.
	\end{align*}
	Then, for the third term in \eqref{eq2.1} we get
	\begin{equation*}
		\bigg\| \sum_{Q\in\QQ^k} \chi_Q\Ps{\mu_k}\chi_{\mathbb{R}^{n+1}\setminus{Q}} - S_k\Ps\mu_k \bigg\| \lesssim \bigg( \sum_{j=0}^{k-1} \theta_{j,p_s}^2 \bigg)^{1/2}.
	\end{equation*}
	Combining all three estimates in \eqref{eq2.1} we get
	\begin{equation*}
		\|\Ps\mu_k\|\gtrsim \theta_{k,p_s}-\|\Ps\mu_k\|-\bigg( \sum_{j=0}^{k-1} \theta_{j,p_s}^2 \bigg)^{1/2},
	\end{equation*}
	so that applying Lemma \ref{lem2.2} together with \eqref{eq1.5},
	\begin{equation*}
		\theta_{k,p_s} \lesssim \|\Ps\mu_k\|+\bigg( \sum_{j=0}^{k-1} \theta_{j,p_s}^2 \bigg)^{1/2} \lesssim \|\Ps\mu_k\|,
	\end{equation*}
	and this finishes the proof of Theorem \ref{thm4.1} assuming Lemma \ref{lem2.2}. In the forthcoming subsections we give basic notation and some details on how to prove the previous Lemma. The arguments are inspired by those given by Tolsa in \cite{T2} for Riesz kernels, so we only focus on specifying the computations which depend on the kernel itself and the nature of the $s$-parabolic Cantor set we are considering.
	
	\subsection{The stopping scales and intervals \mathinhead{I_k}{}} 
	Let $B$ be some big constant (say $B>100$) to be fixed below. We define inductively the following subset
	\begin{equation*}
		\text{Stop}:=\{ s_0,\ldots,s_m \}\subset \{0,1,\ldots,k\}.
	\end{equation*}
	First, set $s_0:=0$. If for some $j\geq 0$, $s_j$ has already been defined and $s_j<k-1$, then $s_{j+1}$ is the least integer $i>s_j$ which verifies at least one of the following:
	\begin{enumerate}
		\item[a)]  $i=k$, or
		\item[b)] $\theta_{i,p_s}>B\theta_{s_j,p_s}$, or
		\item[c)] $\theta_{i,p_s}<B^{-1}\theta_{s_j,p_s}$.
	\end{enumerate}
	We finish our construction of Stop when we find some $s_{j+1}=k$. Notice that we have
	\begin{equation*}
		[0,k-1]=\bigcup_{j=0}^{m-1} [s_j,s_{j+1}) =: \bigcup_{j=0}^{m-1} I_j,
	\end{equation*}
	with $I_j$ pairwise disjoint. Observe that $|I_j|=s_{j+1}-s_j$ coincides with $\#(I_j\cap \mathbb{Z})$. We write
	\begin{equation*}
		T_j\mu_k:=\sum_{s_j\leq i <s_{j+1}} D_i\Ps\mu_k, \qquad \text{for $0\leq j \leq m$}.
	\end{equation*}
	Then, $S_k\Ps\mu_k = \sum_{i=0}^{k-1} S_{i+1}\Ps\mu_k - S_i\Ps\mu_k = \sum_{i=0}^{k-1} D_i\Ps\mu_k = \sum_{j=0}^{m-1} T_j\mu_k$, and since functions $D_j\Ps\mu_k$ are pairwise orthogonal,
	\begin{equation*}
		\|S_k\Ps\mu_k\|^2=\sum_{j=0}^{m-1} \|T_j\mu_k\|^2.
	\end{equation*}	
	
	\subsection{Good and bad scales} To simplify notation we shall write, for any $A\subset [0,k]$,
	\begin{equation*}
		\sigma(A):=\sum_{j\in A\cap \mathbb{Z}} \theta_{j,p_s}^2.
	\end{equation*}
	We will say that $j\in \{0,1,\ldots,k-1\}$ is a \textit{good scale} and we write $j\in \pazocal{G}$, if
	\begin{equation*}
		\sum_{i=0}^j \theta_{i,p_s}\frac{\ell_j}{\ell_i} =: p_j \leq 40\theta_{j,p_s}.
	\end{equation*}
	Otherwise we that $j$ is a \textit{bad scale} and we write $j\in \pazocal{B}$.
	\begin{lem}
		\label{lem3.3}
		The following holds,
		\begin{equation*}
			\sigma(\pazocal{B}) \leq \frac{1}{400} \sigma([0,k-1]).
		\end{equation*}
	\end{lem}
	\begin{proof}
		Proceeding as in the proof of Lemma \ref{lem4.1.3} we get
		\begin{align*}
			\sum_{j=0}^{k-1} p_j^2 \leq \bigg( \frac{d}{d-1} \bigg)^2 \sum_{i=0}^{k-1}\theta_{i,p_s}^2 = \bigg( \frac{d}{d-1} \bigg)^2\sigma([0,k-1]) \leq 4\sigma([0,k-1]).
		\end{align*}
		Then,
		\begin{equation*}
			\sigma(\pazocal{B})=\sum_{j\in\pazocal{B}}\theta_{j,p_s}^2 \leq \frac{1}{1600} \sum_{j=0}^{k-1}p_j ^2 \leq \frac{1}{400} \sigma([0,k-1]).
		\end{equation*}
	\end{proof}
	
	\subsection{Good and bad intervals} We will say that an interval $I_j$ is \textit{good} if
	\begin{equation*}
		\sigma(I_j\cap \pazocal{G})\geq \frac{1}{400}\sigma(I_j).
	\end{equation*}
	Otherwise we say that it is bad, meaning,
	\begin{equation*}
		\sigma(I_j\cap \pazocal{B}) > \frac{399}{400}\sigma(I_j).
	\end{equation*}
	
	\begin{lem}
		\label{lem4.4}
		The following holds,
		\begin{equation*}
			\sigma([0,k-1])\leq \frac{399}{398} \sum_{I_j \,\text{\normalfont{good}}} \sigma(I_j).
		\end{equation*}
	\end{lem}
	\begin{proof}
		If $I_j$ is bad, then $\sigma(I_j\cap \pazocal{G})<\frac{1}{400}\sigma(I_j)$ and therefore
		\begin{equation*}
			\sigma(I_j)<\frac{1}{400}\sigma(I_j)+\sigma(I_j\cap \pazocal{B}), \quad \text{implying} \quad \sigma(I_j\cap \pazocal{B})>\frac{399}{400}\sigma(I_j).
		\end{equation*}
		By Lemma \ref{lem3.3} we then obtain
		\begin{equation*}
			\sum_{I_j\, \text{bad}} \sigma(I_j) \leq \frac{400}{399}\sigma(\pazocal{B}) \leq \frac{1}{399}\sigma([0,k-1]),
		\end{equation*}
		so we get
		\begin{equation*}
			\sigma([0,k-1])\leq \sum_{I_j\, \text{good}} \sigma(I_j) + \frac{1}{399}\sigma([0,k-1]),
		\end{equation*}
		and the result follows.
	\end{proof}
	
	\subsection{Long and short intervals} Let $N_L$ be some (large) integer to be fixed below. We say that an interval $I_j$ i \textit{long} if
	\begin{equation*}
		|I_j|=s_{j+1}-s_j\geq N_L,
	\end{equation*}
	and otherwise we say it is \textit{short}.
	\begin{lem}
		\label{lem4.5}
		Let $I_j$ be good and $j_0:=\min(I_j\cap \pazocal{G})$. Then,
		\begin{equation*}
			j_0-s_j\leq \frac{400B^4}{400B^4+1}(s_{j+1}-s_j).
		\end{equation*}
	\end{lem}
	\begin{proof}
		Write $\ell:=s_{j+1}-s_j$ and $\lambda=j_0-s_j$. Then, by definition of $s_{j+1}$ we have
		\begin{align*}
			\sigma(I_j\cap \pazocal{G}) \leq \sum_{s_j\leq i <s_{j+1}} \theta_{i,p_s}^2\leq B^2\theta_{s_j,p_s}^2(\ell-\lambda),
		\end{align*}
		as well as
		\begin{align*}
			\sigma(I_j\cap \pazocal{B}) \leq \sum_{i\in \pazocal{B} \,:\, s_j\leq i <s_{j+1}} \theta_{i,p_s}^2\geq B^{-2}\theta_{s_j,p_s}^2\lambda.
		\end{align*}
		Since $I_j$ is good, $\sigma(I_j\cap \pazocal{B})\leq 400 \sigma(I_j\cap \pazocal{G})$, so
		\begin{equation*}
			B^2\lambda \leq 400B^2(\ell-\lambda), \quad \text{that is} \quad \lambda \leq \frac{400B^4}{400B^4+1}\ell,
		\end{equation*}
		and the result follows.
	\end{proof}	
	
	Next we prove a couple of technical lemmas that will be fundamental to provide estimates for intervals $I_j$ which are long and good.
	
	\begin{lem}
    \label{lem4.6}
		Let $0< j \leq k-1$ and $h\geq 0$ integer. If there exists a constant $C_6(n,s,h)$ such that
		\begin{equation*}
			\frac{\ell_j}{\ell_{j-1}}p_{j-1} \leq C_6\big( \theta_{j,p_s}+\theta_{j+1,p_s}+\cdots + \theta_{j+h,p_s} \big),
		\end{equation*}
		then for some other constant $C_7(n,s,h)$,
		\begin{equation*}
			\sum_{i=j}^{j+h}\|D_i\Ps\mu_k\|^2\geq C_7^{-1}\frac{1}{(d+1)^hd^{hn}}\big( \theta_{j,p_s}+\theta_{j+1,p_s}+\cdots + \theta_{j+h,p_s} \big)^2.
		\end{equation*}
	\end{lem}
	\begin{proof}
		Let $f:= \sum_{i=j}^{j+h} D_i\Ps\mu_k$. Take $P\in \QQ^{j+h+1}$ and $Q\in \QQ^j$ containing $P$. Then, for $\ox\in P$ we have $f(\ox)=S_P\Ps\mu_k(\ox)-S_Q\Ps\mu_k(\ox)$. By Lemma \ref{lem1.2} we have
		\begin{equation}
        \label{eq3.3}
			f(\ox)=S_P\Ps_{\mu_k}\chi_{Q\setminus{P}}(\ox)+S_P\Ps_{\mu_k}\chi_{\mathbb{R}^{n+1}\setminus{Q}}(\ox)-S_Q\Ps_{\mu_k}\chi_{\mathbb{R}^{n+1}\setminus{Q}}(\ox).
		\end{equation}
		Proceeding as in the proof of Lemma \ref{lem4.1.2} we get
		\begin{equation*}
			\big\rvert S_P\Ps_{\mu_k}\chi_{\mathbb{R}^{n+1}\setminus{Q}}(\ox)-S_Q\Ps_{\mu_k}\chi_{\mathbb{R}^{n+1}\setminus{Q}}(\ox)\big\rvert \lesssim C_8\frac{\ell_j}{\ell_{j-1}}p_{j-1}, \qquad \text{where $C_8:=(d+1)^hd^{hn}$}.
		\end{equation*}
		Assume that $P$ is an $s$-parabolic cube of the $(j+h+1)$-generation sharing the upper right-most corner of $Q$. In this setting, we write
		\begin{align*}
			\bigg\rvert \frac{1}{\mu_k(P)}\int_P &\Ps_{\mu_k}\chi_{Q\setminus{P}}\dd\mu_k \bigg\rvert \\
			& \gtrsim \frac{1}{\mu_k(P)}\bigg\rvert \int_P \int_{Q\setminus{P}} \frac{(x_1-y_1)(t-u)}{|(x,t)-(y,u)|_{p_s}^{n+2s+2}}\chi_{t-u>0} \dd\mu_k(y,u) \dd\mu_k(x,t) \bigg\rvert
		\end{align*}
		In $Q\setminus{P}$ there are cubes of $\QQ^{j+h+1}$ whose centers share the first spatial component of the center of $P$. We name this family $F_1$. By spatial anti-symmetry, the above double integral is null when the domain of integration of the inner integral is precisely $F_1$. Indeed, if we denote by $\overline{c}_{P}=(c_{P_1},\ldots,c_{P_n},c_{P_t})$ the center of $P$ and write
        \begin{equation*}
            P_{>}:=\{(x_1,\ldots,x_n,t)\in P\,:\, y_1>c_{P_1} \}, \qquad P_{<}:=P\setminus{P_{>}},
        \end{equation*}
        its corresponding halves with respect to the first coordinate, the above double integral over $F_1$ can be rewritten as
        \begin{align*}
            \int_{P_{>}} \int_{F_1} &\frac{(x_1-y_1)(t-u)}{|(x,t)-(y,u)|_{p_s}^{n+2s+2}}\chi_{t-u>0} \dd\mu_k(y,u) \dd\mu_k(x,t)\\
            &+\int_{P_{<}} \int_{F_1} \frac{(x_1-y_1)(t-u)}{|(x,t)-(y,u)|_{p_s}^{n+2s+2}}\chi_{t-u>0} \dd\mu_k(y,u) \dd\mu_k(x,t).
        \end{align*}
        For the second integral, we consider the isometric change of variables
        \begin{equation*}
            \pazocal{R}(x_1,\ldots,x_n,t,y_1,\ldots,y_n,u)=(2c_{P_1}-x_1,x_2,\ldots,x_n,t,2c_{P_1}-y_1,y_2,\ldots,y_n,u),
        \end{equation*}
        that is nothing but a reflection with respect to the first variable centered at the center of $P$ in both the outer and inner domains of integration. Notice that $c_{P_1}$ is also the first coordinate of all the cubes conforming $F_1$, by construction. It is clear then, that $\pazocal{R}(P_{<})=P_{>}$ and $\pazocal{R}(F_1)=F_1$. Therefore,\medskip\\
        \begin{align*}
            \int_{P_{<}} \int_{F_1} &\frac{(x_1-y_1)(t-u)}{|(x,t)-(y,u)|_{p_s}^{n+2s+2}}\chi_{t-u>0} \dd\mu_k(y,u) \dd\mu_k(x,t)\\
            &=\int_{P_{>}} \int_{F_1} \frac{(2c_{P_1}-x_1-2c_{P_1}+y_1)(t-u)\chi_{t-u>0}}{|(2c_{P_1}-x_1,\ldots,x_n,t)-(2c_{P_1}-y_1,\ldots,y_n,u)|_{p_s}^{n+2s+2}} \dd\mu_k(y,u) \dd\mu_k(x,t)\\
            &=-\int_{P_{>}} \int_{F_1} \frac{(x_1-y_1)(t-u)}{|(x,t)-(y,u)|_{p_s}^{n+2s+2}}\chi_{t-u>0} \dd\mu_k(y,u) \dd\mu_k(x,t),
        \end{align*}
        and from this we conclude
        \begin{align*}
            \int_P \int_{F_1} \frac{(x_1-y_1)(t-u)}{|(x,t)-(y,u)|_{p_s}^{n+2s+2}}\chi_{t-u>0} \dd\mu_k(y,u) \dd\mu_k(x,t)=0,
        \end{align*}
        and the claim follows. So we are left to study the previous integral in the remaining inner domain of integration $Q\setminus{(P\cup F_1)}$. In the latter, by the choice of $P$, the integrand is positive and thus it can be bounded as follows: name
        \begin{equation*}
            P=:\Delta_{h+1}\subset \Delta_{h}\subset \cdots \Delta_1 \subset \Delta_{0}=:Q,
        \end{equation*}
        the unique chain of $s$-parabolic cubes with $\Delta_r\in \QQ^{j+r}$ passing form $P$ to $Q$. Then,
        \begin{align*}
			\bigg\rvert \frac{1}{\mu_k(P)}&\int_P \Ps_{\mu_k}\chi_{Q\setminus{P}}\dd\mu_k \bigg\rvert \\
			& \gtrsim \frac{1}{\mu_k(P)}\sum_{i=1}^{h+1}\int_{P}\int_{(\Delta_{i-1}\setminus{\Delta_i})\setminus{F_1}} \frac{(x_1-y_1)(t-u)}{|(x,t)-(y,u)|_{p_s}^{n+2s+2}}\chi_{t-u>0} \dd\mu_k(y,u) \dd\mu_k(x,t)\\
            &\gtrsim \frac{1}{\mu_k(P)}\sum_{i=1}^{h+1} \frac{\ell_{i+j}^{1+2s}}{\ell_{i+j}^{n+2s+2}} \mu_k(\Delta_{i-1})\mu_k(P)=\frac{C(s)}{(d+1)^hd^{hn}}\sum_{i=0}^{h} \theta_{i+j,p_s}=:C_9^{-1}\sum_{i=0}^{h} \theta_{i+j,p_s}.
		\end{align*}
        Then, returning to \eqref{eq3.3},
        \begin{equation*}
            |f(\ox)|\geq C_9^{-1}\big( \theta_{j,p_s}+\theta_{j+1,p_s}+\cdots+\theta_{j+h,p_s} \big)-C_8\frac{\ell_j}{\ell_{j-1}}p_{j-1},
        \end{equation*}
        for $\ox\in P \in  \QQ^{j+h+1}$, with $P$ sharing the upper right-most corner of $Q$. Now pick $C_6\leq C_9^{-1}C_8^{-1}/2$ so that:
        \begin{align*}
            \|\chi_Qf\|^2=\int_{Q}\big\rvert S_{j+h}\Ps\mu_k &- S_Q\Ps\mu_k \big\rvert^2\dd\mu_k\\
            &\geq C^{-1}(s,h)\bigg( \sum_{i=0}^{h} \theta_{i+j,p_s} \bigg)^2 \frac{1}{(d+1)^{j+h+1}d^{n(j+h+1)}},
        \end{align*}
        and summing over all cubes $Q\in \QQ^j$ we get the result.
	\end{proof}

    \begin{lem}
    \label{lem4.7}
        Let $A,c_0$ be positive constants and $r,q\in [0,k-1]$ integers such that $q\leq r$, $\frac{\ell_q}{\ell_{q-1}}\leq c_0\theta_{q,p_s}$ and, for all $j$ with $q\leq j \leq r$,
        \begin{equation*}
            A^{-1}\theta_{q,p_s}\leq \theta_{j,p_s}\leq A\theta_{q,p_s}.
        \end{equation*}
        There exists $N_1(n,s,c_0,A)$ such that if $|q-r|>N_1$, then
        \begin{equation*}
            \sum_{j=q}^r \|D_j\Ps\mu_k\|^2 \geq C|q-r|\theta_{q,p_s}^2, \qquad \text{where $C(n,s,c_0,A)$}.
        \end{equation*}
    \end{lem}
    \begin{proof}
        Set $f:=\sum_{j=q}^r D_j\Ps\mu_k$. We have to show that
        \begin{equation*}
            \|f\|^2 \geq C|q-r|\theta_{q,p_s}^2.
        \end{equation*}
        Let $M_0$ be some positive integer depending on $n,s,c_0$ and $A$ to be fixed below. We decompose $f$ as follows
        \begin{equation*}
            f = \sum_{j=q}^{q+tM_0-1} D_j\Ps\mu_k + \sum_{j=q+tM_0}^r D_j\Ps\mu_k,
        \end{equation*}
        where $t$ is the biggest integer with $q+tM_0-1\leq r$. Assuming $N_1$ big enough (take $N_1>2M_0-1$), we have $|q-r|<M_0t\leq 2|q-r|$. For the first sum on the right side, write
        \begin{equation*}
            \sum_{j=q}^{q+tM_0-1} D_j\Ps\mu_k = \sum_{l=0}^{t-1} \sum_{j=q+lM_0}^{q+(l+1)M_0-1} D_j\Ps\mu_k=:\sum_{l=0}^{t-1}U_l(\mu_k).
        \end{equation*}
        By orthogonality we have
        \begin{equation*}
            \|f\|^2\geq \sum_{l=0}^{t-1} \|U_l(\mu_k)\|.
        \end{equation*}
        We shall show that if the parameter $M_0(n,s,c_0,A)$ is chosen big enough,
        \begin{equation}
        \label{eq4.4}
            \|U_l(\mu_k)\|^2\geq C(n,s,c_0,A)\theta_{q,p_s}^2, \qquad \text{for all $0\leq l \leq t-1$},
        \end{equation}
        and then $\|f\|^2\geq C|q-r|\theta_{q,p_s}^2$, and we would be done. To prove \eqref{eq4.4}, we intend to apply Lemma \ref{lem4.6}. Observe that
        \begin{align*}
            p_{q+lM_0-1} &= \sum_{i=0}^{q+lM_0-1} \theta_{i,p_s}\frac{\ell_{q+lM_0-1}}{\ell_i} = \sum_{i=q}^{q+lM_0-1} \theta_{i,p_s}\frac{\ell_{q+lM_0-1}}{\ell_i} + \sum_{i=0}^{q-1} \theta_{i,p_s}\frac{\ell_{q+lM_0-1}}{\ell_i}\\
            &=\sum_{i=q}^{q+lM_0-1} \theta_{i,p_s}\frac{\ell_{q+lM_0-1}}{\ell_i}+\frac{\ell_{q+lM_0-1}}{\ell_{q-1}}p_{q-1} \leq A\frac{d}{d-1}\theta_{q,p_s}+\frac{\ell_{q+lM_0-1}}{\ell_{q-1}}p_{q-1}.
        \end{align*}
        Then,
        \begin{align*}
            \frac{\ell_{q+lM_0}}{\ell_{q+lM_0-1}}p_{q+lM_0-1} &\leq A\frac{d}{d-1}\theta_{q,p_s}+\frac{\ell_{q+lM_0}}{\ell_{q-1}}p_{q-1}\leq 2A\theta_{q,p_s}+\frac{\ell_q}{\ell_{q+1}}p_{q-1}\\
            &\leq (2A+c_0)\theta_{q,p_s}.
        \end{align*}
        On the other hand,
        \begin{equation*}
            \sum_{j=q+lM_0}^{q+(l+1)M_0-1}\theta_{j,p_s} \geq M_0A^{-1}\theta_{q,p_s}.
        \end{equation*}
        Then, if $M_0$ is big enough, $2A+c_0\leq C_6M_0A^{-1}$, and so we are able to apply Lemma \ref{lem4.6} with $j:=q+lM_0$ and $h:=M_0-1$, so that
        \begin{align*}
            \|U_l(\mu_k)\|^2 \geq C_7^{-1}\frac{1}{(d+1)^{M_0-1}d^{(M_0-1)n}}\bigg( \sum_{j=q+lM_0}^{q+(l+1)M_0-1}\theta_{j,p_s} \bigg)^2 \geq  C_7^{-1}\frac{ M_0^2A^{-2}}{(d+1)^{M_0}d^{M_0n}}\theta_{q,p_s}^2,
        \end{align*}
        and \eqref{eq4.4} follows.
    \end{proof}
    Now we are able to provide the following estimate for intervals which are long and good:
    \begin{lem}
        \label{lem4.8}
        Suppose that the constant $N_L$ is chosen big enough (depending on $B$). If $I_j$ is long and good, then
        \begin{equation*}
            \sigma(I_j) \leq C(B)\|T_j\mu_k\|^2,
        \end{equation*}
        where recall that $T_j\mu_k := \sum_{s_j\leq i < s_{j+1}} D_i\Ps\mu_k$.
    \end{lem}
    \begin{proof}
        Set $\ell:=s_{j+1}-s_j$. Notice that by the definition of $s_{j+1}$,
        \begin{equation*}
            \sigma(I_j)=\sum_{s_j\leq i < s_j} \theta_{j,p_s}^2\leq \ell B^2 \theta_{s_j,p_s}^2.
        \end{equation*}
        If $j_0:=\min (I_j\cap \pazocal{G})$ and we take $N_l \gg 400B^4+1$, by Lemma \ref{lem4.5},
        \begin{equation*}
            s_j-j_0 = \ell-(j_0-s_j)\geq \bigg( 1- \frac{400B^4}{400B^4+1}\bigg)\ell = \frac{1}{400B^4+1}\ell \gg 1.
        \end{equation*}
        We write
        \begin{equation*}
            T_j\mu_k = \sum_{i=s_j}^{j_0-1} D_i\Ps\mu_k + \sum_{i=j_0}^{s_{j+1}-1} D_i \Ps\mu_k,
        \end{equation*}
        and notice that condition $\frac{\ell_q}{\ell_{q-1}}p_{q-1}\leq c_0\theta_{q,p_s}$ is rewritten simply as $p_q\leq (c_0+1)\theta_{q,p_s}$. We apply \ref{lem4.7} with $A=B, q=j_0, r=s_{j+1}-1$ and $c_0=40$, so that
        \begin{equation*}
            \sum_{i=j_0}^{s_{j+1}-1}\|D_i\Ps\mu_k\|\geq \frac{1}{C'(B)}|s_{j+1}-j_0|\theta_{s_j,p_s}^2.
        \end{equation*}
        By orthogonality,
        \begin{align*}
            \|T_j\mu_k\|^2&\geq \frac{1}{C'(B)}|s_{j+1}-j_0|\theta_{s_j,p_s}^2\geq \frac{\ell}{C'(B)(400B^4+1)}\theta_{s_j,p_s}^2\\
            &\geq \frac{1}{C'(B)(400B^4+1)}\sigma(I_j)=:C^{-1}(B)\sigma(I_j),
        \end{align*}
        and we are done.
    \end{proof}
    By Lemmas \ref{lem4.4} and \ref{lem4.8}, to prove the desired bound
    \begin{equation*}
        \sigma([0,k-1])\lesssim \sum_{j=0}^{k-1}\|D_j\Ps\mu_k\|^2= \sum_{i=0}^{m-1}\|T_i\mu_k\|^2,
    \end{equation*}
    we are only left to check
    \begin{equation*}
        \sum_{I_j\, \text{short good}} \sigma(I_j) \lesssim \sum_{j=0}^{k-1}\|D_j\Ps\mu_k\|^2.
    \end{equation*}
    To do so, now we only need to follow the exact same arguments to those in \cite[\textsection 5.7, \textsection 5.8, \textsection 5.9]{T2}, which do not depend on the nature of the convolution kernel nor the change in the geometry of our particular Cantor set. To be more precise, one needs to distinguish three types of intervals $I_j$ depending on the three conditions a), b) and c) used to define Stop. We say that
	\begin{itemize}
		\item $I_j$ is \textit{terminal} if $s_j$ satisfies a), so that $j+1=m$.
		\item $I_j$ has \textit{increasing density} ($I_j\in ID$) if $s_j$ satisfies b) and not a).
		\item $I_j$ has \textit{decreasing density} ($I_j\in DD$) if $s_j$ satisfies c) and not a).
	\end{itemize}
    By means of these notions, one is able to estimate $\sigma(I_j)$ for intervals which are short and good and complete the proof of Lemma \ref{lem2.2}.

    \section{The upper bound for the capacity}
    \label{sec6}
    Let us begin by recalling an alternative definition of the capacity $\Gamma_{\Theta^s,+}$. Fix $s\in(1/2,1]$ and denote by $\Sigma_{n+1}^s(E)$ the collection of positive Borel measures supported on $E$ with upper $s$-parabolic growth of degree $n+1$ with constant 1. By means of \cite[Lemma 4.2]{HMPr}, we shall redefine $\Gamma_{\Theta^s,+}$ simply as follows
    \begin{equation*}
    	\Gamma_{\Theta^s,+}(E):=\sup \big\{ \mu(E)\,:\, \mu\in \Sigma_{n+1}^s(E), \; \|\pazocal{P}^{s}\mu\|_\infty \leq 1 \big\}.
    \end{equation*}
    For a fixed generation $0\leq j \leq k$ of the Cantor set, we define the following auxiliary capacity just for $E_{j,p_s}$,
    \begin{equation*}
        \Gamma_{j}(E_{j,p_s}):=\text{sup}\big\{ \alpha>0\,:\, \|\Ps_{\alpha\mu_j}\|_{L^2(\alpha\mu_j)\to L^2(\alpha\mu_j)}\leq 1 \big\},
    \end{equation*}
    where $\mu_j:=|E_{j,p_s}|^{-1}\pazocal{L}^{n+1}|_{E_{j,p_s}}$.

    \begin{lem}
        \label{lem5.1}
        There exists $C(n,s)>0$ such that for all $0\leq j \leq k$,
        \begin{equation*}
            C^{-1}\bigg( \sum_{i=0}^{j} \theta_{i,p_s}^2 \bigg)^{-1/2}\leq \Gamma_{j}(E_{j,p_s})\leq C\bigg( \sum_{i=0}^{j} \theta_{i,p_s}^2 \bigg)^{-1/2}.
        \end{equation*}
    \end{lem}
    \begin{proof}
        To simplify notation, in this proof we shall write $\langle f, g \rangle := \int fg \dd\mu_j$. Then, since $\mu_j$ is a probability measure,
        \begin{align*}
            \|\Ps(\alpha\mu_j)\|_{L^2(\alpha\mu_j)}^2=|\langle \Ps(\alpha\mu_j), \Ps(\alpha\mu_j) \rangle |\leq \|\Ps_{\alpha\mu_j}\|_{L^2(\alpha\mu_j)\to L^2(\alpha\mu_j)} \alpha^{1/2} \|\Ps(\alpha\mu_j)\|_{L^2(\alpha\mu_j)}^2,
        \end{align*}
        in other words,
        \begin{equation*}
            \|\Ps(\alpha\mu_j)\|_{L^2(\alpha\mu_j)} \leq \alpha^{1/2}\|\Ps_{\alpha\mu_j}\|_{L^2(\alpha\mu_j)\to L^2(\alpha\mu_j)}.
        \end{equation*}
        This implies
        \begin{align*}
            \|\Ps_{\alpha\mu_j}\|_{L^2(\alpha\mu_j)\to L^2(\alpha\mu_j)} &\geq \alpha^{-1/2}\|\Ps(\alpha\mu_j)\|_{L^2(\alpha\mu_j)} = \alpha\|\Ps\mu_j\|_{L^2(\mu_j)} \geq \alpha C_1^{-1} \bigg( \sum_{i=0}^{j} \theta_{i,p_s}^2 \bigg)^{1/2},
        \end{align*}
        where at the last step we have applied Theorem \ref{thm4.1}. So by definition of $\Gamma_{j}(E_{j,p_s})$ we get the upper bound with constant $C_1$. On the other hand, for any $f,g\in L^2(\mu_j)$, applying relation \eqref{eq3.7} we get
        \begin{align*}
            |\langle \Ps_{\alpha\mu_j}f,g\rangle| \leq \alpha C_2^{-1}\bigg( \sum_{i=0}^{j} \theta_{i,p_s}^2 \bigg)^{1/2}\|f\|_{L^2(\alpha\mu_j)}\|g\|_{L^2(\alpha\mu_j)}.
        \end{align*}
        Therefore,
        \begin{equation*}
            \|\Ps_{\mu_j}\|_{L^2(\mu_j)\to L^2(\mu_j)} \leq \alpha C_2^{-1} \bigg( \sum_{i=0}^{j} \theta_{i,p_s}^2 \bigg)^{1/2},
        \end{equation*}
        and by the definition of $\Gamma_j(E_{j,p_s})$ we get the lower bound with constant $C_2$. Hence, setting $C:=\max{(C_1,C_2)}$ we are done.
    \end{proof}

    As mentioned in the proof of Theorem \ref{thm4.1.12}, by Lemma \ref{lem5.1} and \cite[Theorem 4.3]{H} we get
    \begin{equation*}
        \Gamma_{\Theta^s,+}(E_{k,p_s})\geq \widetilde{\Gamma}_{\Theta^s,+}(E_{k,p_s})\gtrsim \bigg( \sum_{j=0}^{k} \theta_{j,p_s}^2 \bigg)^{-1/2}\approx \Gamma_k(E_{k,p_s}).
    \end{equation*}
    We aim at proving the existence of a constant $C_0(n,s)>0$ such that for all $k=1,2,\ldots$
    \begin{equation}
    \label{eq5.2}
        \Gamma_{\Theta^s,+}(E_{k,p_s})\leq C_0\Gamma_k(E_{k,p_s}).
    \end{equation}
    Before proceeding, let us prove that we can assume three assumptions to simplify our problem without loss of generality. Namely, we will assume the existence of $1 \leq M \leq k$ such that
    \begin{enumerate}
        \item[\textbf{A1}:] $\sigma_{M+1}\leq 2\sigma_M \leq \sigma_k\leq 2\sigma_{M+1}$.
        \item[\textbf{A2}:] $\Gamma_{\Theta^s,+}(E_{j,p_s})\leq C_0 \Gamma_j(E_{j,p_s}), \qquad 0<j<k.$
        \item[\textbf{A3}:] For some constant $A_0(n,s)\geq \sqrt{2}C^2$,
        \begin{equation*}
            \Gamma_{\Theta^s,+}\big( E_{p_s}(\lambda_{M+1},\ldots, \lambda_k) \big) \leq A_0\theta_{M,p_s}\widetilde{\Gamma}_{\Theta^s,+}(E_{k,p_s}).
        \end{equation*}
    \end{enumerate}
    
    Let us justify the above. Fix a generation $k>1$ and write
    \begin{equation*}
        \sigma_j:=\theta_{0,p_s}^2+\theta_{1,p_s}^2+\cdots+\theta_{j,p_s}^2, \qquad 0\leq j \leq k.
    \end{equation*}
    We assume that there is $1\leq M \leq k$ such that
    \begin{equation}
    \label{eq5.3}
        \sigma_M \leq \frac{\sigma_k}{2}<\sigma_{M+1}.
    \end{equation}
    If this was not the case, we would have
    \begin{equation*}
        \frac{\sigma_k}{2}<\sigma_1=\theta_{0,p_s}^2+\theta_{1,p_s}^2=1+\frac{1}{\big[(d+1)d^n\lambda_1^{n+1}\big]^2}<\bigg( 2+\frac{1}{d} \bigg) \frac{1}{\big[(d+1)d^n\lambda_1^{n+1}\big]^2}.
    \end{equation*}
    That is, $\sigma_k^{-1/2}\geq C^{-1}\lambda_1^{n+1}$. Then, by Lemma \ref{lem5.1} and \cite[Theorem 4.2]{H},
    \begin{equation*}
        \Gamma_{\Theta^s,+}(E_{k,p_s})\leq \Gamma_{\Theta^s}(E_{1,p_s})\lesssim \pazocal{H}^{n+1}_{\infty,p_s}(E_{1,p_s})\lesssim \lambda_1^{n+1},
    \end{equation*}
    and we would be done. We notice that we can also assume $\theta_{M+1,p_s}^2\leq \sigma_M$. Indeed, if this was not satisfied, then
    \begin{equation*}
        \theta_{M+1,p_s}^{2}<\sigma_{M+1}=\sigma_M+\theta_{M+1,p_s}^2<2\theta_{M+1,p_s}^2,
    \end{equation*}
    and thus
    \begin{align*}
        \Gamma_{M+1}(E_{M+1,p_s})&\geq C^{-1}\sigma_{M+1}^{-1/2} > (\sqrt{2}C)^{-1}\theta_{M+1,p_s}^{-1}\geq (\sqrt{2}C)^{-1}\pazocal{H}^{n+1}_{\infty,p_s}(E_{M+1,p_s})\\
        &\geq (\sqrt{2}C)^{-1}\widetilde{C}\,\Gamma_{\Theta^s,+}(E_{M+1,p_s})\geq (\sqrt{2}C)^{-1}\widetilde{C}\,\Gamma_{\Theta^s,+}(E_{k,p_s}).
    \end{align*}
    Then,
    \begin{align*}
        \Gamma_{\Theta^s,+}(E_{k,p_s})\leq \frac{\sqrt{2}C}{\widetilde{C}} \Gamma_{M+1}(E_{M+1,p_s}) \leq \frac{\sqrt{2}C^2}{\widetilde{C}} \sigma_{M+1}^{-1/2} < \frac{C^2}{\widetilde{C}} \sigma_k^{-1/2},
    \end{align*}
    where in the last step we have applied \eqref{eq5.3}. Then, we deduce \eqref{eq5.2} by redefining $C_0$ if necessary, since $C$ and $\widetilde{C}$ are constants depending only on $n$ and $s$. Hence, combining \eqref{eq5.3} with assumption $\theta_{M+1,p_s}^2\leq \sigma_M$ we get \textbf{A1}.
    
    Now, to prove \eqref{eq5.2} we would proceed by induction. Since $\sigma_1^{-1/2}\gtrsim \lambda_1^{n+1}$ (as seen above), the case $k=1$ holds. So the induction hypothesis is
    \begin{equation*}
        \Gamma_{\Theta^s,+}(E_{j,p_s})\leq C_0 \Gamma_j(E_{j,p_s}), \qquad 0<j<k.
    \end{equation*}
    that is precisely assumption \textbf{A2}. We will also denote by
    \begin{equation*}
        E_{p_s}(\lambda_{i_1},\ldots, \lambda_{i_j})
    \end{equation*}
    the $j$-th generation of a Cantor set constructed as in \eqref{eq4.1.1} with $\lambda_{i_l}$ its $l$-th contraction ratio. Now we distinguish two cases:
    \begin{enumerate}
        \item[\textit{1.}] For some constant $A_0(n,s)$ to be determined below, 
        \begin{equation*}
            \Gamma_{\Theta^s,+}\big( E_{p_s}(\lambda_{M+1},\ldots, \lambda_k) \big) \leq A_0\theta_{M,p_s}\Gamma_{\Theta^s,+}(E_{k,p_s}).
        \end{equation*}
        \item[\textit{2.}] The above relation does not hold.
    \end{enumerate}
    We deal first with case \textit{2}. By the induction hypothesis applied to the sequence $\lambda_{M+1},\ldots,\lambda_k$ we have
    \begin{align*}
        \Gamma_{\Theta^s,+}(E_{k,p_s}) &\leq A_0^{-1}\theta_{M,p_s}^{-1} \Gamma_{\Theta^s,+}\big( E_{p_s}(\lambda_{M+1},\ldots, \lambda_k) \big)\\
        &\leq A_0^{-1}C_0\theta_{M,p_s}^{-1}\Gamma_{k-M}\big( E_{p_s}(\lambda_{M+1},\ldots, \lambda_k) \big)\\
        &\leq A_0^{-1}C_0C\theta_{M,p_s}^{-1}\Bigg[ \sum_{j=M+1}^{k} \frac{1}{\big( (d+1)d^n\lambda_{M+1}^{n+1}\cdots (d+1)d^n\lambda_{i}^{n+1} \big)^2} \Bigg]^{-1/2}\\
        &=A_0^{-1}C_0C\Bigg[ \sum_{j=M+1}^{k}\theta_{j,p_s}^2 \Bigg]^{-1/2}= A_0^{-1}C_0C\big[ \sigma_k-\sigma_M \big]^{-1/2}.
    \end{align*}
    Assumption \eqref{eq5.3} implies $\sigma_k\leq 2(\sigma_k-\sigma_M)$, so by Lemma \ref{lem5.1} we get
    \begin{equation*}
        \Gamma_{\Theta^s,+}(E_{k,p_s})\leq \sqrt{2}A_0^{-1}C_0C\sigma_k^{-1/2}\leq \sqrt{2}A_0^{-1}C_0C^2\Gamma_{k}(E_{k,p_s}),
    \end{equation*}
    so taking $A_0\geq \sqrt{2}C^2$ we are done. Hence we are left to study case \textit{1}, which is stated in \textbf{A3}.
    
    Now, let us proceed by considering the measure
    \begin{equation*}
        \mu:=\Gamma_{\Theta^s,+}(E_{k,p_s})\,\mu_M, \qquad \text{where recall} \quad \mu_M:=|E_{M,p_s}|^{-1}\pazocal{L}^{n+1}|_{E_{M,p_s}}.
    \end{equation*}
    If we prove that $\Ps_\mu$ is a bounded operator in $L^2(\mu)$, by definition of $\Gamma_M(E_{M,p_s})$ we would get
    \begin{equation*}
        \Gamma_{\Theta^s,+}(E_{k,p_s}) \leq \Gamma_M(E_{M,p_s})\leq C\sigma_M^{-1/2} \lesssim \sigma_k^{-1/2},
    \end{equation*}
    where the last step holds by \textbf{A1}. Then, by Lemma \ref{lem5.1} we would be done. Thus, the desired upper bound will follow from the next result:
    \begin{lem}
    \label{lem5.2}
        Under assumptions \normalfont{\textbf{A1}}, \normalfont{\textbf{A2}} and \normalfont{\textbf{A3}}, $\Ps_\mu$ is a bounded operator in $L^2(\mu)$.
    \end{lem}
    \begin{proof}
        We shall verify the hypothesis of a local $Tb$ theorem in the setting of spaces of homogeneous type found in \cite[Theorem 3.5]{AR}. Using the notation of the previous article, we work with the metric $\rho$ induced by the $s$-parabolic distance, with our particular measure $\mu$, and we choose as dominating function $\Lambda(\ox,r):=cr^{n+1}$, with $c=c(n,s)$ big enough. Observe that given any $s$-parabolic cube $Q\subset \mathbb{R}^{n+1}$ centered at $E_{M,p_s}$ we have:
        \begin{enumerate}
            \item[\textit{1.}] If $\ell(Q)\leq \ell_M$, since $2s-1>0$,
            \begin{equation*}
                \mu(Q)\leq \frac{\Gamma_{\Theta^s,+}(E_{k,p_s})}{|E_{M,p_s}|}\ell(Q)^{n+2s}\lesssim \frac{\ell_k^{n+1}}{\ell_{M}^{n+2s}}\ell(Q)^{n+2s} \leq \ell(Q)^{n+1}.
            \end{equation*}
            \item[2.] If $\ell(Q)>\ell_M$, we may assume that there is $0\leq N < M$ such that $\ell_N\leq \ell(Q) < \ell_{N+1}$. Since $Q$ intersects a bounded number (depending on $n$ and $s$) of $s$-parabolic cubes of the $N$-th generation, we have
            \begin{align*}
                \mu(Q) &\lesssim \mu(Q^N) = {\Gamma_{\Theta^s,+}(E_{k,p_s})}\frac{1}{(d+1)^Nd^{nN}}\leq {\Gamma_{\Theta^s,+}(E_{N,p_s})}\frac{1}{(d+1)^Nd^{nN}}\\
                &\lesssim \pazocal{H}^{n+1}_{\infty,p_s}(E_{N,p_s})\frac{1}{(d+1)^Nd^{nN}}\leq \ell_N^{n+1} \leq \ell(Q)^{n+1}.
            \end{align*}
        \end{enumerate}
        This implies that $\mu(B(\ox,r))\leq \Lambda(\ox,r)$, for any $\ox\in\mathbb{R}^{n+1}$ and $r>0$. Following \cite{MT}, the same kind of ideas can be used to prove that $\mu(B(\ox,2r))\lesssim \mu(B(\ox,r))$ holds. Now, given any $s$-parabolic cube $Q^j_i$ of the $j$-th generation, $0\leq j \leq M$, $1\leq i \leq (d+1)^jd^{jn}$, we will construct two functions $b_i^j$ and $b_i^{j,\ast}$ supported on $Q_i^j$ with
        \begin{align}
        \label{eq5.5}
            \|b_i^j\|_{L^\infty(\mu)}\leq 1,& \qquad \|b_i^{j,\ast}\|_{L^\infty(\mu)}\leq 1,\\
            \label{eq5.6}
            \mu(Q_i^j)\leq C\bigg\rvert \int b_i^j\dd\mu \bigg\rvert,& \qquad \mu(Q_i^j)\leq C\bigg\rvert \int b_i^{j,\ast}\dd\mu \bigg\rvert \quad \text{and}\\
            \label{eq5.7}
            \|\Ps_{\mu,\varepsilon}b_i^j\|_{L^\infty(\mu)}\leq 1,& \qquad \|\pazocal{P}^{s,\ast}_{\mu,\varepsilon}b_i^{j,\ast}\|_{L^\infty(\mu)}\leq 1, \qquad \text{uniformly on $\varepsilon>0$.}
        \end{align}
        Begin by considering a positive Borel measure $\nu$ admissible for $\Gamma_{\Theta^s,+}(E_{k,p_s})$ such that $\Gamma_{\Theta^s,+}(E_{k,p_s})\leq 2\nu(E_{k,p_s})$. Recall that there exists an absolute parameter $\tau_0$ so that $0<\lambda_j\leq \tau_0<1/d$, for each $j$. Hence, we shall consider $\alpha(\tau_0)>1$ an $s$-parabolic dilation factor small enough so that the dilated $M$-th generation $ \alpha\,E_{M,p_s}$ is still conformed by disjoint $s$-parabolic cubes. For each $1\leq i \leq (d+1)^Md^{Mn}$ take $\psi_i^M$ a smooth function satisfying $\chi_{Q_i^M}\leq \psi_i^M \leq \chi_{\alpha Q_i^M}$ and such that
        \begin{equation*}
            \|\nabla_x\psi_i^M\|_{\infty} \leq C(\tau_0)\ell_M^{-1}, \qquad \|\partial_t\psi_i^M\|_\infty \leq C'(\tau_0)\ell_{M}^{-2s}, \qquad \|\Delta_x\psi_i^M\|_\infty \leq C''(\tau_0)\ell_M^{-2}.
        \end{equation*}
        Notice that for $i_1\neq i_2$ the functions $\psi_{i_1}^M$ and $\psi_{i_2}^M$ have disjoint supports. For a fixed generation $0\leq j \leq M$ we will also write
        \begin{equation*}
            \psi_i^j := \sum_{Q_m^M \subset Q_i^j} \psi_m^M, \qquad  1 \leq i \leq (d+1)^{j}d^{jn}.
        \end{equation*}
        Observe that $\psi_i^j$ is admissible for $Q_i^j$ (with implicit constants depending on $\tau_0$). Since $\text{supp}(\nu)\subset E_{k,p_s}$, for each generation $j$ there is an index $i_0$ such that 
        \begin{equation*}
            \Gamma_{\Theta^s,+}(E_{k,p_s})\leq 2 \nu(E_{j,p_s}) \leq 2 \sum_{i=1}^{(d+1)^jd^{jn}}\langle \nu, \psi_i^j \rangle \leq 2(d+1)^jd^{jn}\langle \nu, \psi_{i_0}^j \rangle.
        \end{equation*}
        In other words,
        \begin{equation}
        \label{eq5.8}
            \mu(Q_i^j)\leq 2\langle \nu, \psi_{i_0}^j \rangle, \qquad \text{for all} \quad 1 \leq i \leq (d+1)^jd^{jn}.
        \end{equation}
        With this in mind, we fix a generation $0\leq j \leq M$ and $1\leq i \leq (d+1)^{j}d^{jn}$ and define $b_i^j$. Consider $\varphi$ test function supported on $Q^0$, $0\leq \varphi \leq 1$, $\int_{Q^0}\varphi \geq 1/2$ and such that
        \begin{equation}
        \label{eq5.9}
            \|\Ps_{\pazocal{L}^{n+1}|_{Q^0},\varepsilon}\varphi \|_\infty \lesssim \ell(Q^0)^{2s-1}\|\varphi\|_\infty, \qquad \text{uniformly on $\varepsilon>0$}.
        \end{equation}
        The above property can be imposed since $|\nabla_xP_s(\ox)|\lesssim |\ox|_{p_s}^{-n-1}$ \cite[Theorem 2.2]{HMPr} and $2s-1>0$. Now set
        \begin{equation*}
            \varphi_i^M(\ox):= \varphi \bigg( \frac{\ox-\overline{v}_i^M}{\ell_M} \bigg),
        \end{equation*}
        where $\overline{v}_i^M$ is the vertex of $Q_i^M$ closest to the origin. We have $\text{supp}(\varphi_i^M)\subset Q_i^M$ and equally $\|\Ps_{\pazocal{L}^{n+1}|_{Q^M_i},\varepsilon}\varphi^M_i\|_\infty \lesssim \ell(Q^0)^{2s-1}\|\varphi\|_\infty$ uniformly on $\varepsilon>0$. Moreover,
        \begin{align}
        \nonumber
            \int \varphi_i^M \dd\mu &= \frac{\widetilde{\Gamma}_{\Theta^s,+}(E_{k,p_s})}{|E_{M,p_s}|}\int_{Q_i^M}\varphi_i^M \dd\pazocal{L}^{n+1} = \frac{\widetilde{\Gamma}_{\Theta^s,+}(E_{k,p_s})}{(d+1)^Md^{Mn}}\int_{Q^0}\varphi \dd\pazocal{L}^{n+1}\\
            &\geq \frac{1}{2} \frac{\widetilde{\Gamma}_{\Theta^s,+}(E_{k,p_s})}{(d+1)^Md^{Mn}} = \frac{1}{2} \mu(Q_i^M).
            \label{eq5.10}
        \end{align}
        The last equality of the first line clarifies that the integral of $\varphi_i^M$ with respect to $\mu$ is independent of $i$. Now we shall define $b_{i_0}^j$ as follows,
        \begin{equation*}
            b_{i_0}^j:=\sum_{Q_m^M \subset Q_{i_0}^j} \langle \nu, \psi_m^M \rangle \frac{\varphi_m^M}{\int \varphi_m^M \dd\mu }.
        \end{equation*}
        For $i\neq i_0$ we construct $b_i^j$ by translation of $b_{i_0}^j$. More precisely, if $Q_i^j = \overline{w}_i^j+Q_{i_0}^j$, we put
        \begin{equation*}
            b_i^j(\ox):=b_{i_0}^{j}(\ox-\overline{w}_i^j), \qquad \overline{x}\in\mathbb{R}^{n+1}.
        \end{equation*}

        Now we define $b_i^{j,\ast}$ composing $b_i^{j}$ with a proper temporal reflection. More precisely, if $\overline{c}_{i}^j$ denotes the center of $Q_i^j$ and $t_{i}^j$ its temporal component, we name $\mathsf{R}_{i}^j$ the temporal reflection with respect to the horizontal hyperplane $\{t=t_{i}^j\}$. That is,
        \begin{equation*}
            \mathsf{R}_i^j(x_1,\ldots,x_n,t):=(x_1,\ldots,x_n,2t_{i}^j-t).
        \end{equation*}
        With this, we define
        \begin{equation*}
            b_i^{j,\ast}(\ox):=b_i^j(\mathsf{R}_i^j(\ox)).
        \end{equation*}
        Since $\mathsf{R}_i^j(Q_i^j)=Q_i^j$, it is clear that $\text{supp}(b_i^{j,\ast})\subset Q_i^j$. Another key observation is that, by construction of the Cantor set, the set
        \begin{equation*}
            \QQ_i^{M\nearrow j}:=\bigcup_{Q_m^M\subset Q_i^j} Q_m^M\qquad \text{also satisfies} \qquad \mathsf{R}_i^j(\QQ_i^{M\nearrow j}) = \QQ_i^{M\nearrow j}.
        \end{equation*}

        We now begin by proving that \eqref{eq5.6} holds. Observe that with estimate \eqref{eq5.8} and the independence of $\int \varphi_i^M$ with respect to $i$, writing $\varphi_\mathsf{R_i^j(m)}^M:=\varphi_m^M\circ\mathsf{R}_i^j$ we directly deduce the desired estimates:\medskip\\
        \begin{align}
        \label{eq5.11}
            \bigg\rvert \int b_i^j \dd\mu \bigg\rvert &= \sum_{Q_m^M \subset Q_{i_0}^j} \langle \nu, \psi_m^M \rangle \frac{\int \varphi_m^M(\,\cdot-\overline{w}_i^j) \dd\mu}{\int \varphi_m^M \dd\mu } = \langle \nu, \psi_{i_0}^j \rangle \geq \frac{1}{2}\mu(Q_i^j), \quad \text{and}\\
        \label{eq5.11_conj}
            \bigg\rvert \int b_{i}^{j,\ast}\dd\mu \bigg\rvert &= \sum_{Q_m^M \subset Q_{i_0}^j} \langle \nu, \psi_m^M \rangle \frac{\int \varphi_{\mathsf{R}_i^j(m)}^M \dd\mu}{\int \varphi_m^M \dd\mu } = \langle \nu, \psi_{i_0}^j \rangle \geq \frac{1}{2}\mu(Q_i^j).
        \end{align}
        Moreover, since $\nu$ is admissible for $\Gamma_{\Theta^s,+}(E_{k,p_s})$ and thus has upper $s$-parabolic growth of degree $n+1$, by \cite[Lemma 4.2]{HMPr} and the localization result \cite[Lemma 3.2]{H} we get
        \begin{equation*}
            \|\Ps_\nu \psi_i^j\|_\infty \lesssim 1,
        \end{equation*}
        with implicit constants depending on $n,s$ and $\tau_0$, for all $0\leq j \leq M$ and $1\leq i \leq (d+1)^jd^{jn}$. Then, in particular,
        \begin{equation*}
            \langle \nu, \psi_{i_0}^M \rangle \lesssim \Gamma_{\Theta^s,+}\big( \alpha Q_{i_0}^M\cap E_{k,p_s}\big) = \Gamma_{\Theta^s,+}\big(Q_{i_0}^M\cap E_{k,p_s}\big).
        \end{equation*}
        But the set $Q_{i_0}^M\cap E_{k,p_s}$ can be obtained by dilating $E_{p_s}(\lambda_{M+1},\ldots,\lambda_k)$ an $s$-parabolic factor $\ell_M$. Thus, by assumption \textbf{A3},
        \begin{align*}
            \Gamma_{\Theta^s,+}\big(Q_{i_0}^M\cap E_{k,p_s}\big) &= \ell_{M}^{n+1}\Gamma_{\Theta^s,+}\big(E_{p_s}(\lambda_{M+1},\ldots,\lambda_k) \big)\\
            &\leq \ell_M^{n+1}A_0\theta_{M,p_s}\Gamma_{\Theta^s,+}(E_{k,p_s}) = \frac{A_0}{(d+1)^Md^{Mn}}\Gamma_{\Theta^s,+}(E_{k,p_s}) =A_0\mu(Q_{i_0}^M). 
        \end{align*}
        So we get
        \begin{align*}
            \bigg\rvert \int b_i^j \dd\mu \bigg\rvert &= \bigg\rvert \int b_i^{j,\ast} \dd\mu \bigg\rvert = \langle \nu, \psi_{i_0}^j \rangle = \sum_{Q_m^M \subset Q_{i_0}^j} \langle \nu, \psi_m^M \rangle \leq  \sum_{Q_m^M \subset Q_{i_0}^j} \langle \nu, \psi_{i_0}^M \rangle \\
            &\leq A_0 \sum_{Q_m^M \subset Q_{i_0}^j} \mu(Q_{i_0}^M)=\mu(Q_{i_0}^j)=\mu(Q_i^j)
        \end{align*}
        Therefore, by \eqref{eq5.11} and \eqref{eq5.11_conj} we get
        \begin{equation}
            \label{eq5.12}
            \mu(Q_i^j)\geq \bigg\rvert \int b_i^j \dd\mu \bigg\rvert = \bigg\rvert \int b_i^{j,\ast} \dd\mu \bigg\rvert = \langle \nu, \psi_{i_0}^j \rangle \geq \frac{1}{2}\mu(Q_i^j),
        \end{equation}
        and \eqref{eq5.5} follows. 
        
        We are left to verify relations \ref{eq5.7}. In fact, we will bound $|\pazocal{P}^s_{\mu,\varepsilon}b_i^j|$ at every point, and from this we will get the same estimate for the conjugate operator by the following observation: since $\QQ_i^{M\nearrow j}$ is invariant under $\mathsf{R}_i^j$ and $\mathsf{R}_i^j$ is its own inverse,
        \begin{align*}
            \pazocal{P}^{s,\ast}_{\mu}b_i^{j,\ast}(\ox) &= \frac{\Gamma_{\Theta^s,+}(E_{k,p_s})}{|E_{M,p_s}|} \int_{\QQ_i^{M\nearrow j}} \nabla_xP_s(x-y,u-t)b_i^{j}(y,2t_i^j-u)\dd y \dd u\\
            &=\frac{\Gamma_{\Theta^s,+}(E_{k,p_s})}{|E_{M,p_s}|} \int_{\QQ_i^{M\nearrow j}} \nabla_xP_s(x-y,2t_i^j-t-u)b_i^{j}(y,u)\dd y \dd u = \pazocal{P}^s_\mu b_i^j(\mathsf{R}_i^j(\ox)).
        \end{align*}
        So we focus our efforts on estimating $|\pazocal{P}^s_{\mu,\varepsilon}b_i^j|$. Let us begin our arguments by choosing $\Psi^\varepsilon:=\varepsilon^{-n-2s}\Psi(\frac{\cdot}{\varepsilon})$ a standard mollifier in the $s$-parabolic space $\mathbb{R}^{n+1}$. Estimate $\|\Ps(\psi_i^j\nu)\|_\infty \lesssim 1$ is equivalent to
        \begin{equation}
        \label{eq5.13}
            \|\nabla_xP_s\ast \Psi^\varepsilon\ast (\psi_i^j\nu)\|_\infty \lesssim 1,
        \end{equation}
        where we convey that $\|\Psi\|_{L^1(\mathbb{R}^{n+1})}=1$. We call $\nabla_x P_s^\varepsilon := \nabla_xP_s\ast \Psi^\varepsilon$ the regularized kernel and $\pazocal{P}_\mu^{s,\varepsilon}$ its associated convolution operator. Notice the different position of the symbol $\varepsilon$ in the previous operator with respect to $\Ps_{\mu,\varepsilon}$, that recall that is a truncation of $\Ps_\mu$.

        It is not hard to prove that for $|\ox|_{p_s}<\varepsilon$, one has $|\nabla_xP_s^\varepsilon(\ox)|\lesssim \varepsilon^{-(n+1)}$. Then, given $\sigma$ any Borel measure (possibly signed) with upper $s$-parabolic growth of degree $n+1$, by the growth estimates of \cite[Theorem 2.2]{HMPr} we get for any $\ox\in\mathbb{R}^{n+1}$\medskip
        \begin{align*}
            \big\rvert \pazocal{P}^{s,\varepsilon}\sigma(\ox) &- \Ps_\varepsilon \sigma (\ox) \big\rvert \\
            &= \bigg\rvert \int_{|\ox-\oy|_{p_s}\leq 4\varepsilon} \nabla_xP_s^\varepsilon(\ox-\oy)\dd\sigma(\oy) + \int_{|\ox-\oy|_{p_s}> 4\varepsilon} \nabla_xP_s^\varepsilon(\ox-\oy)\dd\sigma(\oy)\\
            &\hspace{1cm}-\int_{\varepsilon<|\ox-\oy|_{p_s}<\leq 4\varepsilon} \nabla_xP_s(\ox-\oy)\dd\sigma(\oy)-\int_{|\ox-\oy|_{p_s}> 4\varepsilon} \nabla_xP_s(\ox-\oy)\dd\sigma(\oy)\bigg\rvert\\
            &\lesssim \frac{|\sigma|(B(\ox,4\varepsilon))}{\varepsilon^{n+1}} + \int_{|\ox-\oy|_{p_s}> 4\varepsilon} |\nabla_xP_s^\varepsilon(\ox-\oy)-\nabla_xP_s(\ox-\oy)|\dd|\sigma|(\oy)\\
            &\lesssim 1 +\int_{|\ox-\oy|_{p_s}>4\varepsilon}\int_{|\oz|_{p_s}<\varepsilon} |\nabla_xP_s(\ox-\oy-\oz)-\nabla_xP_s(\ox-\oy)|\Psi^\varepsilon(\oz)\dd\oz \dd|\sigma|(\oy)\\
            &\lesssim 1+ \int_{|\ox-\oy|_{p_s}>4\varepsilon}\int_{|\oz|_{p_s}<\varepsilon} \frac{|\oz|_{p_s}}{|\ox-\oy|_{p_s}^{n+2}}\Psi^\varepsilon(\oz) \dd\oz \dd|\sigma|(\oy)\\
            &\lesssim 1 +\varepsilon\int_{|\oz|_{p_s}<\varepsilon}\bigg(  \int_{|\ox-\oy|_{p_s}>4\varepsilon} \frac{\dd|\sigma|(\oy)}{|\ox-\oy|_{p_s}^{n+2}} \bigg)\Psi^\varepsilon(\oz)\dd\oz\lesssim 1+\varepsilon\frac{\|\Psi^\varepsilon\|_{L^1}}{\varepsilon}\lesssim 1.
        \end{align*}
        This implies that the boundedness of $\Ps_{\mu,\varepsilon}b_i^j$ follows from that of $\pazocal{P}^{s,\varepsilon}_\mu b_{i_0}^j$, that in turn by \eqref{eq5.13} follows from that of
        \begin{equation}
        \label{eq5.14}
            \big\rvert \pazocal{P}^{s,\varepsilon}_\mu b_{i_0}^j - \pazocal{P}^{s,\varepsilon}_\nu \psi_{i_0}^j \big\rvert.
        \end{equation}
        Observe that the above difference can be written as
        \begin{equation*}
            \pazocal{P}^{s,\varepsilon}_\mu b_{i_0}^j - \pazocal{P}^{s,\varepsilon}_\nu \psi_{i_0}^j = \sum_{Q_m^M\subset Q_{i_0}^j} \pazocal{P}^{s,\varepsilon}\alpha_m^M,
        \end{equation*}
        where we have defined
        \begin{equation*}
            \alpha_m^M:=\langle \nu, \psi_m^M \rangle \frac{\varphi_m^M}{\int \varphi_m^M \dd\mu}\mu - \psi_m^M\nu, \qquad \text{that is such that} \quad \int \dd\alpha_m^M = 0.
        \end{equation*}
        Let us prove two claims that will help us estimate \eqref{eq5.14}:
        \begin{enumerate}[itemsep=0.3cm]
            \item[\textit{1.}] The following holds,
            \begin{equation}
                \label{eq5.15}
                |\pazocal{P}^{s,\varepsilon}\alpha_m^M(\ox)\lesssim 1|, \qquad \ox\in\mathbb{R}^{n+1}.
            \end{equation}
            Indeed, notice that on the one hand
            \begin{align*}
                |\pazocal{P}^{s,\varepsilon}_\mu \varphi_m^M(\ox)| &= \frac{\widetilde{\Gamma}_{\Theta^{s},+}(E_{k,p_s})}{|E_{M,p_s}|}|\pazocal{P}^{s,\varepsilon}_{\pazocal{L}^{n+1}|_{Q_m^M}} \varphi_m^M(\ox)|\\
                &\leq \frac{\widetilde{\Gamma}_{\Theta^{s},+}(E_{k,p_s})}{(d+1)^Md^{Mn}\ell_M^{n+1}}\| \pazocal{P}^{s,\varepsilon}_{\pazocal{L}^{n+1}|_{Q^0}} \varphi\|_\infty \\
                &\lesssim \frac{\pazocal{H}^{n+1}_{\infty,p_s}(E_{M,p_s})}{(d+1)^Md^{Mn}\ell_M^{n+1}}\| \pazocal{P}^{s,\varepsilon}_{\pazocal{L}^{n+1}|_{Q^0}} \varphi\|_\infty \leq \| \pazocal{P}^{s,\varepsilon}_{\pazocal{L}^{n+1}|_{Q^0}} \varphi\|_\infty,
            \end{align*}
            and by \eqref{eq5.9} we get $|\pazocal{P}^{s,\varepsilon}_\mu \varphi_m^M(\ox)|\lesssim 1$. We also have $|\pazocal{P}^{s,\varepsilon}_\nu \psi_m^M(\ox)|\lesssim 1$ for almost every point by \eqref{eq5.13}, and by continuity this extends to all points in $\mathbb{R}^{n+1}$. Finally, by \eqref{eq5.10} and \eqref{eq5.12}
            \begin{equation*}
                \int\varphi_m^M\dd\mu \geq \frac{1}{2}\mu(Q_m^M) \geq \langle \nu, \psi_m^M \rangle,
            \end{equation*}
            and the desired estimate follows.
            \item[\textit{2.}] Let us now fix $m$ and a cube $Q_m^M$ centered at $\ox_m^M$. Take $\ox\in\mathbb{R}^{n+1}$ with $|\ox-\ox_m^M|_{p_s}>4\ell_M$ and assume $\varepsilon<\ell_M/2$. Then,
            \begin{equation}
                \label{eq5.16}
                \pazocal{P}^{s,\varepsilon}\alpha_m^M(\ox)\lesssim \frac{\ell_M^{n+2}}{\text{dist}_{p_s}(\ox,Q_m^M)^{n+2}}.
            \end{equation}
            To prove the latter, begin by taking $\oy\in B(\ox,\varepsilon)$ and write
            \begin{equation*}
                |\pazocal{P}^s\alpha_m^M(\oy)|=\bigg\rvert \frac{\langle \nu, \psi_m^M \rangle}{\int \varphi_m^M \dd\mu} \Ps_\mu \varphi_m^M(\oy) - \Ps_\nu\psi_m^M(\oy) \bigg\rvert.
            \end{equation*}
            Notice
            \begin{align*}
                \Ps_\mu\varphi_m^M(\oy)= \int_{Q_m^M}\big( \nabla_xP_s(\oy-\oz) - \nabla_xP_s(\oy&-\ox_m^M) \big)\varphi_m^M(\oz)\dd\mu(\oz)\\
                &+\nabla_xP_s(\oy-\ox_m^M)\int_{Q_m^M} \varphi(\oz)\dd\mu(\oz),
            \end{align*}
            which implies,
            \begin{align*}
                |\Ps\alpha_m^M(\oy)|\leq \big\rvert &\Ps_\nu\psi_m^M(\oy)-\langle\nu,\psi_m^M\rangle\nabla_xP_s(\oy-\ox_m^M) \big\rvert\\
                &\hspace{-0.25cm}+\frac{\langle \nu, \psi_m^M\rangle}{\int \varphi_m^m\dd\mu}\int_{Q_m^M} |\nabla_xP_s(\oy-\oz-\nabla_xP_s(\oy-\ox_m^M))|\varphi_m^M(\oz)\dd\mu(\oz).
            \end{align*}
            For the second summand we apply the last estimate of \cite[Theorem 2.2]{HMPr} and obtain that it can be bounded by
            \begin{equation*}
                \frac{\langle \nu, \psi_m^M\rangle}{\int \varphi_m^m\dd\mu}\int_{Q_m^M}\frac{|\oz-\ox_m^M|_{p_s}}{|\oy-\oz|_{p_s}^{n+2}}\varphi_m^M(\oz)\dd\mu(\oz)\lesssim \frac{\ell_M^{n+2}}{\text{dist}_{p_s}(\oy,Q_m^M)^{n+2}},
            \end{equation*}
            where we have used $\langle \nu, \psi_m^M\rangle\lesssim \ell_M^{n+1}$ by \cite[Theorem 3.1]{H}. Here, implicit constants may depend on $\tau_0$. 
            
            For the first summand, naming $\varphi_{\oy}:=\nabla_xP_s(\oy-\cdot)-\nabla_xP_s(\oy-\ox_{m}^M)$ we get
            \begin{equation*}
                \big\rvert \Ps_\nu\psi_m^M(\oy)-\langle\nu,\psi_m^M\rangle\nabla_xP_s(\oy-\ox_m^M) \big\rvert = |\langle \varphi_{\oy}, \psi_m^M\nu \rangle|
            \end{equation*}
            It is clear that
            \begin{equation*}
                \|\varphi_{\oy}\|_{L^\infty(2Q_m^M)}\lesssim \frac{\ell_M}{\text{dist}_{p_s}(\oy,Q_m^M)^{n+2}}.
            \end{equation*}
            In addition, by \cite[Theorem 2.2]{HMPr} it is also clear that for almost every point $\oz\in 2Q_m^M$ the following bounds hold,
            \begin{align*}
                |\partial_{x_i}^2P_s(\oz-\oy)|&\lesssim \frac{1}{|\oz-\oy|_{p_s}^{n+2}}\lesssim \frac{1}{\text{dist}_{p_s}(\oy,Q_m^M)^{n+2}},\\
                |\partial_{x_i}^3P_s(\oz-\oy)|&\lesssim \frac{1}{|\oz-\oy|_{p_s}^{n+3}}\lesssim \frac{1}{\text{dist}_{p_s}(\oy,Q_m^M)^{n+2}\,\ell_M},\\
                |\partial_t\partial_{x_i}P_s(\oz-\oy)|&\lesssim \frac{|z-y|}{|\oz-\oy|_{p_s}^{n+2s+2}}\lesssim \frac{\ell_M}{\text{dist}_{p_s}(\oy,Q_m^M)^{n+2}\,\ell_M^{2s}}.
            \end{align*}
            Now we observe that the function
            \begin{equation*}
                \eta:= \bigg[ \frac{\ell_M}{\text{dist}_{p_s}(\oy,Q_m^M)^{n+2}} \bigg]^{-1} \varphi_{\oy}\psi_{m}^M
            \end{equation*}
            becomes admissible for $\widetilde{\Gamma}_{\Theta^s,+}(2Q_m^M)$, with implicit constants also depending on $\tau_0$. To be precise, the previous function may lack being differentiable with respect to time in a set of null $\pazocal{L}^{n+1}$-measure. Nevertheless, \cite[Theorem 3.1]{HMPr} can be also proved, with minor modifications, for such functions, and so the growth result \cite[Theorem 3.1]{H} can be also extended to this setting. Then,
            \begin{align*}
                |\langle \varphi_{\oy}, \psi_m^M\nu \rangle| = \frac{\ell_M}{\text{dist}_{p_s}(\oy,Q_m^M)^{n+2}}|\langle \eta, \nu \rangle|\lesssim \frac{\ell_M^{n+2}}{\text{dist}_{p_s}(\oy,Q_m^M)^{n+2}}.
            \end{align*}
            Therefore, given $\ox\in\mathbb{R}^{n+1}$ with $|\ox-\ox_m^M|>4\ell_M$ and $\varepsilon<\ell_M/2$ we have proved:
            \begin{equation*}
                |\Ps\alpha_m^M(\oy)|\lesssim \frac{\ell_M^{n+2}}{\text{dist}_{p_s}(\oy,Q_m^M)^{n+2}}, \qquad \oy\in B(\ox,\varepsilon).
            \end{equation*}
            This in turn implies
            \begin{equation*}
                \pazocal{P}^{s,\varepsilon}\alpha_m^M(\ox)\leq \int_{B(\ox,\varepsilon)}\Psi^\varepsilon(\ox-\oy)|\Ps\alpha_m^M(\oy)|\dd\oy \lesssim \frac{\ell_M^{n+2}}{\text{dist}_{p_s}(\oy,Q_m^M)^{n+2}},
            \end{equation*}
            since $\|\Psi^\varepsilon\|_{L^1}=1$, that is what we wanted to prove.
        \end{enumerate}
        \medskip
    Now let $\varepsilon<\ell_m/2$, $\ox\in \mathbb{R}^{n+1}$ and fix $Q_m^M(\ox)$ one of the $s$-parabolic cubes of the $M$-th generation that is closest to $\ox$. Since the number of cubes such that $\text{dist}_{p_s}(Q_m^M, Q_{m}^M(\ox))\leq 4\ell_M$ is bounded by a constant depending on $n$ and $s$, we get by relations \eqref{eq5.15} and \eqref{eq5.16},
    \begin{align*}
        \big\rvert \pazocal{P}^{s,\varepsilon}_\mu &b_{i_0}^j(\ox) - \pazocal{P}^{s,\varepsilon}_\nu \psi_{i_0}^j(\ox) \big\rvert\leq \sum_{Q_m^M\subset Q_{i_0}^{j}} |\pazocal{P}^{s,\varepsilon}\alpha_m^M(\ox)|\\
        &=\sum_{\substack{ Q_m^M\subset Q_{i_0}^{j} \\ \text{dist}_{p_s}(Q_m^M, Q_{m}^M(\ox))\leq 4\ell_M }} |\pazocal{P}^{s,\varepsilon}\alpha_m^M(\ox)| + \sum_{\substack{ Q_m^M\subset Q_{i_0}^{j} \\ \text{dist}_{p_s}(Q_m^M, Q_{m}^M(\ox))> 4\ell_M }} |\pazocal{P}^{s,\varepsilon}\alpha_m^M(\ox)|\\
        &\lesssim 1 + \sum_{\substack{ Q_m^M\subset Q_{i_0}^{j} \\ Q_m^M \neq Q_m^M(\ox) }} \frac{\ell_M^{n+2}}{\text{dist}_{p_s}(\ox,Q_m^M)^{n+2}} \leq 1+\sum_{r=j}^{M-1}\sum_{Q_m^M\subset Q_{i_0}^r\setminus{Q_{i_0}^{r+1}}} \frac{\ell_M^{n+2}}{\text{dist}_{p_s}(\ox,Q_m^M)^{n+2}}\\
        &\leq 1+ \sum_{r=j}^{M-1}\sum_{Q_m^M\subset Q_{i_0}^r\setminus{Q_{i_0}^{r+1}}} \frac{\ell_M^{n+2}}{\ell_r^{n+2}} (d+1)^{M-r}d^{(M-r)n}\lesssim 1+\sum_{r=j}^{M-1}\bigg[ \frac{1}{d}\bigg(1+\frac{1}{d}\bigg) \bigg]^{M-r}\lesssim 1.
    \end{align*}
    This finally implies
    \begin{equation*}
        |\Ps_{\mu,\varepsilon}b_i^j(\ox)|\lesssim 1 \quad \text{and} \quad |\pazocal{P}^{s,\ast}_{\mu,\varepsilon}b_i^{j,\ast}(\ox)|\lesssim 1, \qquad \forall \ox\in\mathbb{R}^{n+1} \; \text{and} \; \,0<\varepsilon<\frac{\ell_M}{2}\; \text{uniformly}
    \end{equation*}
    With this, by \cite[Theorem 3.5]{AR} we would get the $L^2(\mu)$-boundedness of $\Ps_{\mu,\varepsilon}$ uniformly on $0<\varepsilon<\frac{\ell_M}{2}$. Now, applying Cotlar's inequality (see for example \cite[Theorem 2.18]{T3}), we would deduce
    \begin{align*}
        |\Ps_{\mu,\varepsilon}b_i^j(\ox)|\lesssim 1 \quad \text{and} \quad |\pazocal{P}^{s,\ast}_{\mu,\varepsilon}b_i^{j,\ast}(\ox)|\lesssim 1, \qquad \forall \ox\in\mathbb{R}^{n+1} \; \text{and} \; \,\varepsilon>0\; \text{uniformly},
    \end{align*}
    and we get the $L^2(\mu)$-boundedness of $\Ps_{\mu}$, and the proof of the lemma is complete.
    \end{proof}
    Thus, in light of Theorems \ref{thm4.1.12} and \ref{thm3.6} we have finally obtained:
    \begin{thm}
		\label{thm5.3}
		Let $(\lambda_j)_j$ be such that $0<\lambda_j\leq \tau_0<1/d$, for every $j$, and denote by $E_{p_s}$ its associated $s$-parabolic Cantor set as in \eqref{eq4.1.1}. Then, there exists $C(n,s,\tau_0)>0$ such that for every generation $k$,
		\begin{equation*}
			\Gamma_{\Theta^s,+}(E_{k,p_s})\leq C \Bigg( \sum_{j=0}^{k} \theta_{j,p_s}^2 \Bigg)^{-1/2}.
		\end{equation*}
        Moreover,
        \begin{equation*}
            C^{-1} \Bigg( \sum_{j=0}^{\infty} \theta_{j,p_s}^2 \Bigg)^{-1/2} \leq \widetilde{\Gamma}_{\Theta^s,+}(E_{p_s}) \leq \Gamma_{\Theta^s,+}(E_{p_s})\leq C \Bigg( \sum_{j=0}^{\infty} \theta_{j,p_s}^2 \Bigg)^{-1/2}.
        \end{equation*}
	\end{thm}

	\vspace{0.75cm}
	{\small
		\begin{tabular}{@{}l}
			\textsc{Joan\ Hernández,} \\ \textsc{Departament de Matem\`{a}tiques, Universitat Aut\`{o}noma de Barcelona,}\\
			\textsc{08193, Bellaterra (Barcelona), Catalonia.}\\
			{\it E-mail address}\,: \href{mailto:joan.hernandez@uab.cat}{\tt{joan.hernandez@uab.cat}}
		\end{tabular}
	}

\bigskip
    \bigskip
	
	\begin{thebibliography}{CMM+2}

        \bibitem[AR]{AR} Auscher, R., Routin, E.: Local $Tb$ theorems and Hardy inequalities. Journal of Geometric Analysis \textbf{23}, 303--374 (2013).

        \bibitem[BG]{BG} Blumenthal, R.M., Getoor, R. K.: Some theorems on stable processes. Transactions of the American Mathematical Society \textbf{95}(2), 263--273 (1960).

        \bibitem[DPV]{DPV} Di Nezza, E., Palatucci, G., Valdinoci, E.: Hitchhiker's guide to the fractional Sobolev spaces. Bulletin of Mathematical Sciences \textbf{136}(5), 521--573 (2012).
		
		\bibitem[H]{H} Hernández, J.: Removable singularities for Lipschitz fractional caloric functions in time varying domains. \textit{ArXiv e-prints} (Feb. 2025) \url{https://doi.org/10.48550/arXiv.2412.18402}
		
		\bibitem[HMPr]{HMPr} Hernández, J., Mateu, J., Prat, L.: On fractional parabolic BMO and \mathinhead{\text{Lip}_{\alpha}}{} caloric capacities. \textit{ArXiv e-prints} (Feb. 2025) \url{https://doi.org/10.48550/arXiv.2412.16520}.

        \bibitem[Ho1]{Ho1} Hofmann, S.: A characterization of commutators of parabolic singular integrals. In: Garc\'ia-Cuerva, J. (ed.) Fourier Analysis and Partial Differential Equations, pp. 195--210. CRC Press, Boca Raton (1995).

        \bibitem[Ho2]{Ho2} Hofmann, S.: Parabolic singular integrals of Calderón-type, rough operators, and caloric layer potentials. Duke Mathematical Journal \textbf{90}(2), 209--259 (1997).

        \bibitem[HoL]{HoL} Hofmann, S., Lewis, J. L.: $L^2$ solvability and representation by caloric layer potentials in time-varying domains. Annals of Mathematics \textbf{144}(2), 349--420 (1996).
		
		\bibitem[Hy]{Hy} Hytönen, T.: Framework for non-homogeneous analysis on metric spaces, and the RMBO space of Tolsa. Publicacions Matemàtiques \textbf{54}(2), 485--504 (2010).
		
		\bibitem[HyMa]{HyMa} Hytönen, T., Martikainen, H.: Non-homogeneous $Tb$ theorem and random dyadic cubes on metric measure spaces. Journal of Geometric Analysis \textbf{22}(4), 1071--1107 (2012).


        \bibitem[LMu]{LMu} Lewis, J. L., Murray M. A. M.: The method of layer potentials for the heat equation in time-varying domains. Memoirs of the American Mathematical Society, Volume 114, Number 545. American Mathematical Society, Rhode Island (1995).

        \bibitem[LS]{LS} Lewis, J. L., Silver, J.: Parabolic measure and the Dirichlet problem for the heat equation in two dimensions. Indiana University Mathematics Journal \textbf{37}(4), 801--839 (1988).
		
		\bibitem[MT]{MT} Mateu, J., Tolsa, X.: Riesz Transforms and Harmonic Lip$_1$-Capacity in Cantor Sets. Proceedings of the London Mathematical Society \textbf{89}(3), 676--696 (2004).
		
		\bibitem[MPr]{MPr} Mateu, J., Prat, L.: Removable singularities for solutions of the fractional heat equation in time varying domains. Potential Analysis \textbf{60}, 833--873 (2024).
		
		\bibitem[MPrT]{MPrT} Mateu, J., Prat, L., Tolsa, X.: Removable singularities for Lipschitz caloric functions in time varying domains. Revista Matemática Iberoamericana \textbf{38} (2), 547--588 (2022).
		
		\bibitem[Mat]{Mat} Mattila, P.: Geometry of Sets and Measures in Euclidean Spaces: Fractals and Rectifiability. Cambridge University Press, Cambridge (1995).

        \bibitem[MoPu]{MoPu} Mourgoglou, M., Puliatti, C.: Blow-ups of caloric measure in time varying domains and applications to two-phase problems. Journal de Mathématiques Pures et Appliquées \textbf{152}, 1--68 (2021).

        \bibitem[NSt]{NSt} Nyström, K., Strömqvist, M.: On the parabolic Lipschitz approximation of parabolic uniform rectifiable sets. Revista Matemática Iberoamericana \textbf{33}(4), 1397--1422 (2017).

        \bibitem[RoSe1]{RoSe1} Ros-Oton, X., Serra, J.: The Dirichlet problem for the fractional Laplacian: Regularity up to the boundary. Journal de Mathématiques Pures et Appliquées \textbf{101}(3), 275--302 (2014).

        \bibitem[RoSe2]{RoSe2} Ros-Oton, X., Serra, J.: Regularity theory for general stable operators. Journal of Differential Equations \textbf{260}(12), 8675--8715 (2016).

        \bibitem[Ste]{Ste} Stein, E. M.: Singular integrals and differentiability properties of functions. Princeton University Press, Princeton (1970).
		
		\bibitem[T1]{T1} Tolsa, X.: BMO, $H^1$, and Calderón-Zygmund operators for non doubling measure. Mathematische Annalen, \textbf{319}(1), 89--149 (2001).
		
		\bibitem[T2]{T2} Tolsa, X.: Calderón-Zygmund Capacities and Wolff Potentials on Cantor Sets. Journal of Geometric Analysis \textbf{21}, 195--223 (2011).
		
		\bibitem[T3]{T3} Tolsa, X.: Analytic Capacity, the Cauchy Transform, and Non-homogeneous Calderón-Zygmund Theory. Birkhäuser Cham, Switzerland (2014).

        \bibitem[U]{U} Uy, N. X.: Removable sets of analytic functions satisfying a Lipschitz condition. Arkiv för Matematik \textbf{17}(1--2), 19--27 (1979).
		
	\end{thebibliography}
\end{document}